\documentclass[11pt,twoside]{preprint}
\usepackage{hyperref}
\usepackage{amssymb}
\usepackage{amsfonts}
\usepackage{times}
\usepackage{mhsymb}
\usepackage{mhequ}
\usepackage{mhenvs}
\usepackage{mhfig}
\usepackage{mathrsfs}
\usepackage{microtype}
\usepackage{booktabs}
\usepackage{tikz}

\newtheorem{algorithm}{Algorithm}
\newtheorem{assumption}[lemma]{Assumption}
\newtheorem{example}[lemma]{Example}

\def\1{{\mhpastefig{root}}}
\def\2{{\mhpastefig[2/3]{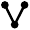}}}
\def\3{{\!\mhpastefig[1/2]{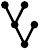}}}
\def\9{{\mhpastefig[1/2]{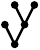}}}
\def\4{{\mhpastefig[1/2]{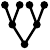}}}
\def\5{{\!\mhpastefig[1/2]{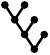}}}
\def\6{{\!\mhpastefig[1/2]{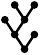}}}
\def\8{{\mhpastefig[1/2]{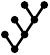}}}
\def\7{{\mhpastefig[1/2]{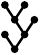}}}

\def\smalltree{{\raisebox{-0.4mm}{\mhpastefig[3/5]{tree12}}}}
\def\bigtree{{\raisebox{-0.6mm}{\mhpastefig[3/5]{tree112}}}}
\def\nicetree{{\raisebox{-0.4mm}{\mhpastefig[3/5]{tree22}}}}

\newcommand{\sims}{\stackrel{\mbox{\tiny $P$}}{\sim}}
\newcommand{\simp}{\stackrel{\mbox{\tiny $P$}}{\approx}}
\newcommand{\sym}{{\symb[\rm]{sym}}}

\def\KK{\mathscr{K}}

\def\LL{\mathscr{L}}
\def\JJ{\mathscr{J}}
\def\TT{\mathscr{T}}
\def\GG{\mathscr{G}}
\def\cC{\mathscr{C}}

\def\XX{\mathbf{X}}
\def\YY{\mathbf{Y}}

\def\iiint{\rlap{\hspace{0.43em}\raisebox{0.2em}{\rule{0.5em}{0.1em}}}}
\def\siint{\rlap{\hspace{0.36em}\raisebox{0.26em}{\rule{0.3em}{0.06em}}}}
\def\iint{\mathchoice{\iiint}{\siint}{\siint}{\siint}\int}

\def\E{\mathbf{E}}
\def\P{\mathbf{P}}

\def\${|\!|\!|}
\definecolor{darkred}{rgb}{0.9,0.1,0.1}

\def\Fd{F_{\eps}^{\bar\eps}}
\def\Cd{C_{\eps}^{\bar\eps}}

\begin{document}

\title{Solving the KPZ equation}
\author{Martin Hairer}
\institute{The University of Warwick, \email{M.Hairer@Warwick.ac.uk}}
\maketitle

\begin{abstract}
We introduce a new concept of solution to the KPZ equation which is shown to
extend the classical Cole-Hopf solution. This notion provides a factorisation of the
Cole-Hopf solution map into a ``universal'' measurable map from the probability space into
an explicitly described auxiliary metric space, composed with a new solution map that
has very good continuity properties. 
The advantage of such  a formulation is that it essentially provides a pathwise notion of a solution,
together with a very detailed approximation theory. In particular, our construction completely bypasses the 
Cole-Hopf transform, thus laying the groundwork for proving that the KPZ equation describes the
fluctuations of systems in the KPZ universality class.

As a corollary of our construction, we obtain very detailed new regularity results about the solution, 
as well
as its derivative with respect to the initial condition. Other byproducts of the proof
include an explicit approximation to the stationary solution of the KPZ equation,
a well-posedness result for the Fokker-Planck equation associated to a particle 
diffusing in a rough space-time dependent potential, and a new periodic homogenisation result 
for the heat equation with a space-time periodic potential. One ingredient in our construction
is an example of a non-Gaussian
rough path such that the area process of its natural approximations needs to be renormalised
by a diverging term for the approximations to converge.
\end{abstract}

\tableofcontents
\section{Introduction}

The aim of this article is to construct and describe solutions to the KPZ equation. 
At a purely formal level, this equation is given by
\begin{equ}[e:KPZ]
\d_t h = \d_x^2 h + \lambda (\d_x h)^2 - \infty + \xi\;,
\end{equ}
where ``$\infty$'' denotes an ``infinite constant'' required to renormalise the divergence
appearing in the term $(\d_x h)^2$ and $\lambda > 0$ is a ``coupling strength''. Here, $h(x,t)$ is a continuous stochastic process
with $x \in S^1$ (which we usually identify with $[0,2\pi]$, but we will always assume periodic boundary conditions) and 
$\xi$ denotes space-time white noise which is a distribution-valued Gaussian field with correlation function
\begin{equ}[e:convention]
\E \xi(x,t) \xi(y,s) = 4\pi \delta(x-y)\delta(t-s)\;.
\end{equ}
The prefactor $1$ in front of the term $\d_x^2 h$ and the strange-looking prefactor $4\pi$ in the definition of $\xi$ are 
normalisation constants which could be set to any positive value by rescaling time, $h$ and $\lambda$,
but our particular choice will simplify some expressions in the sequel.

At this stage, it is of course completely unclear what \eref{e:KPZ} actually means and, in a way, this
is the main question that will be addressed in this article. 
Originally, the equation \eref{e:KPZ} was proposed by Kardar, Parisi and Zhang as a model of surface growth 
\cite{PhysRevLett.56.889}. 
However, it was later realised that it is a universal object that describes the fluctuations
of a number of strongly interacting models of statistical mechanics with space-time  dependencies. 
For example, it is known rigorously to arise as the fluctuation process for
the weakly asymmetric simple exclusion process \cite{MR1462228}, as well as the partition 
function for directed polymer models 
\cite{KardarRough,ReplicaKPZ,MR2796514}. More generally, the solution to the KPZ equation is expected to describe the fluctuations
of a much larger class of systems, namely the systems in the KPZ universality class which is associated
to the dynamic scaling exponents ${3\over 2}$, see for example \cite{SeppQuas}. We refer to the excellent review
article \cite{Ivan} for many more references and a more detailed historical account of the KPZ equation.

Over the past ten years or so, substantial progress has been made in the understanding of
the solutions to \eref{e:KPZ} (especially in the extended case $x \in \R$), but very few 
results had been established rigorously until an explosion of
recent results yielding \textit{exact} formulae for the one-point distribution of solutions to \eref{e:KPZ}. 
A foundation for these results was laid by the groundbreaking work of Johansson \cite{Johansson:00},
who noted a link between discrete approximations to \eref{e:KPZ} and random matrix theory,
and who used this to prove that the Tracy-Widom distribution arises as the long-time limit of this discrete model.
One  stunning recent result was the rigorous proof in \cite{SeppQuas,MR2796514,HalfBrownian} 
of the fact that, also for the continuous model
\eref{e:KPZ}, one has 
$u(t) \approx t^{1/3}$ for large times (this had already been conjectured in \cite{PhysRevLett.56.889}
and the results in \cite{Johansson:00} provided further evidence, but the lack of a good approximation theory for
\eref{e:KPZ} had defeated earlier attempts) and that, at least for the ``infinite wedge'' and the ``half-Brownian'' initial
distributions, the law of $t^{-1/3}u(0,t)$, appropriately recentred, does  converge, as $t \to \infty$,
to the Tracy-Widom distribution. Another very recent achievement exploiting this link is the series of articles 
\cite{MR2570756,MR2628936,MR2796514,SasaSpo:10} in which the authors provide
an \textit{exact} formula for the law of the solution to the KPZ equation at a fixed time and fixed spatial location. 
These results built on a number of previous results using related ideas, in particular Tracy and Widom's exact
formulae for the asymmetric simple exclusion process \cite{TW2,TW1,TW3}. 

Together with this explosion of exact results on the solutions to \eref{e:KPZ}, 
there has been renewed interest in giving a rigorous interpretation of \eref{e:KPZ}.
Ever since the seminal work of Bertini and Giacomin \cite{MR1462228}, there
has been an accepted notion of solution to \eref{e:KPZ} via the so-called ``Cole-Hopf transform'',
which had long been known to be useful in the study of the deterministic KPZ / Burgers equation \cite{Hopf,Cole}.
The idea is to consider the solution $Z$ to the linear multiplicative stochastic heat equation
\begin{equ}[e:linear]
dZ = \d_x^2Z\,dt + \lambda Z\,dW(t)\;,
\end{equ}
where $W$ is a cylindrical Brownian motion on $L^2(S^1)$ (i.e.\ it is the time integral of the space-time white
noise $\xi$). Here, the term $Z\,dW(t)$ should be interpreted as an It\^o integral. It is well-known
(see for example the monograph \cite{DaPrato-Zabczyk92}) that the mild form of \eref{e:linear} admits a unique positive
solution in a suitable space of adapted processes. One then \textit{defines} the process $h$ to be given by
\begin{equ}[e:defh]
h(x,t) = \lambda^{-1} \log Z(x,t)\;.
\end{equ}
In the sequel, we denote this solution by 
$h = \CS_{\mathrm{CH}}(h_0,\omega)$, where $h_0 =  \lambda^{-1} \log Z_0$ is an initial condition for \eref{e:KPZ}.
The map $\CS_{\mathrm{CH}}$ is a jointly measurable map from $\CC \times \Omega$ into $\CC(\R_+,\CC^\alpha)$
for every $\alpha < {1\over 2}$.

There are two powerful arguments for this to be the ``correct'' notion of solution to \eref{e:KPZ}. First, one
can consider the solution $Z_\eps$ to \eref{e:linear} with $W$ replaced by $W_\eps$, which is
obtained by multiplying the $k$th Fourier component with $\phi(\eps k)$ for some smooth cut-off 
function $\phi$ with compact support and $\phi(0) = 1$. Defining $h_\eps$ via \eref{e:defh} and applying It\^o's formula, it is then possible 
to verify that
$h_\eps$ solves the equation
\begin{equ}[e:KPZeps]
\d_t h_\eps = \d_x^2 h_\eps + \lambda (\d_x h_\eps)^2 - \lambda C_\eps + \xi_\eps\;,
\end{equ}
where the constant $C_\eps$ is given by $C_\eps = \sum_{k\in \Z} \phi^2(k\eps) \approx {1\over \eps} \int_{\R} \phi^2(x)\,dx$.
Since $Z_\eps \to Z$ as $\eps \to 0$ by standard SPDE arguments, 
it follows that $h_\eps$ converges to a limiting process $h$ which,
in light of \eref{e:KPZeps}, does indeed formally solve \eref{e:KPZ}. 

The second argument in favour of the Cole-Hopf solution is that, as shown in \cite{MR1462228},
the fluctuations of the stationary weakly asymmetric simple exclusion process (WASEP) converge, 
under a suitable rescaling, to the Cole-Hopf solution to \eref{e:KPZ}. This result was further improved recently
in \cite{MR2796514} where, among other things, the authors show that the fluctuations for 
the WASEP with ``infinite wedge'' initial condition
are also given by the Cole-Hopf solution.

The problem with the Cole-Hopf solution is that it does not provide a satisfactory theory of approximations to 
\eref{e:KPZ}. Indeed, all approximations to \eref{e:KPZ} must first be reinterpreted as approximations to
\eref{e:linear}, which is not always convenient. While it works well for the approximation by mollification of the noise
that we just mentioned,
it does not work at all for other natural approximations to \eref{e:KPZ}, like for example adding a small
amount of hyperviscosity or performing a spatio-temporal mollification of the noise. 
This is also why only the fluctuations of the WASEP have so far been shown to
converge to the solutions to the KPZ equation: this is one of the rare discrete systems that behave well under
the corresponding version of the Cole-Hopf transform. 

As a consequence, there have been a number of, 
unfortunately unsuccessful, attempts
over the past decade to provide a more natural notion of solution without making use of the Cole-Hopf transform.
For example, as illustrated by \eref{e:KPZeps}, the Cole-Hopf solution really corresponds to an interpretation of the nonlinearity
as a Wick product $\d_x h \diamond \d_x h$, where the Wick product is defined relative to the Gaussian structure
given on the space of solutions by the linearised equation (i.e. the one where we simply drop the
nonlinearity altogether). One could also imagine interpreting the nonlinearity as a Wick product with respect
to the Gaussian structure given on the underlying probability space by the driving noise $\xi$. This yields
a \textit{different} concept of solution that was studied in \cite{BookZhang,MR1743612}. In the spatially extended situation, this solution
appears however to behave in a non-physical way in the sense that it does not exhibit the correct scaling exponents.

Following a similar line of though, one may try to apply ``standard'' renormalisation theory to interpret 
\eref{e:KPZ}. This programme was initiated in
\cite{MR2365646}, where the authors were able to treat a mollified version of \eref{e:KPZ}, namely
\begin{equ}[e:KPZalpha]
\d_t h = \d_x^2 h + (-\d_x^2)^{-2\alpha} \bigl((\d_x h)^2 - \infty\bigr) + (-\d_x^2)^{-\alpha} \xi\;.
\end{equ}
Unfortunately, the techniques used there seem to break down at $\alpha = {1\over 8}$. 
We refer to Remark~\ref{rem:renorm} below
for an explanation why ${1\over 8}$ is one natural barrier arising for ``conventional'' techniques
and what other barriers (the largest of which being the passage from $\alpha > 0$ to $\alpha = 0$) 
must be crossed before reaching \eref{e:KPZ}.

Another way to make sense of \eref{e:KPZ} could be to formulate a corresponding
martingale problem. This is a technique that was explored in \cite{MR1888875} for example.
Very recently, a somewhat related notion of ``weak energy solution''
was introduced in \cite{Milton} and further refined in \cite{SigurdRecent}, but there is so far no 
corresponding uniqueness result. Furthermore, this notion does not 
seem to provide any way of distinguishing solutions that differ by spatial constants.

Some recent progress has also been made in providing an approximation theory to variants of \eref{e:linear}, 
but the results are only partial \cite{Etienne,Guillaume}.
To a large extent, this long-standing problem is solved (or at least a programme is established on how to
solve some of its variants) by the results of this article. In particular, we provide a ``pathwise'' interpretation
of \eref{e:KPZ}, together with a robust approximation theory. 

Before we state the theorem, we introduce some notation. 
We denote by $\bar \CC^\alpha$ the space $\CC^\alpha$ to which we add a ``point at infinity'' $\infty$ with
neighbourhoods of the form $\{h\,:\, \|h\|_{\alpha}> R\} \cup \{\infty\}$, which turns $\bar \CC^\alpha$ into a Polish space.
We need to work with the space $\bar \CC^\alpha$ since our construction only provides local solution so that,
for a given $\Psi \in \CX$ and a given initial condition $h_0$, we cannot guarantee that solutions will not explode
in finite time. However, if solutions do explode in finite time, it is always because the $\CC^\alpha$-norm diverges.
With this terminology in place, our result can be stated as follows:

\begin{theorem}\label{theo:mainCont}
There exists a Polish space $\CX$, a measurable map $\Psi\colon \Omega \to \CX$ and, for every $\beta \in (0,{1\over 2})$, 
a lower semicontinuous map $T_\star \colon \CC^\beta\times\CX \to (0,+\infty]$ and a map $\CS_{\mathrm{R}} \colon \CC^\beta \times \CX \to \CC(\R_+, \bar \CC^{{1\over 2}-\beta})$ such that 
\begin{equ}
(t, h_0, \Psi) \mapsto \CS_{\mathrm{R}}(h_0, \Psi)(t)\;,
\end{equ}
is continuous on all triples such that $t \in (0,T_\star(h_0,\Psi))$.
Furthermore, for every $h_0 \in \CC^\beta$, one has $T_\star(h_0, \Psi(\omega)) = +\infty$ almost surely and
the identity
\begin{equ}
\CS_{\mathrm{CH}} (h_0, \omega) = \CS_{\mathrm{R}}\bigl(h_0, \Psi(\omega)\bigr)\;,
\end{equ}
holds for almost every $\omega \in \Omega$. 

Finally, there exists a separable Fr\'echet space $\CW$ such that $\CX \subset \CW$ (with the
topology of $\CX$ given by the induced topology of $\CW$) and such that, for every $\ell \in \CW^\star$, 
the random variable $\ell(\Psi(\cdot))$ belongs to the union of the first four Wiener chaoses of $\xi$ 
(see Section~\ref{sec:chaos} for a short reminder of the definition
of the Wiener chaos). 
\end{theorem}

\begin{remark}
The letter `R' in $\CS_{\mathrm{R}}$ stands for ``Rough''. It will become clear later why we chose this terminology.
\end{remark}

\begin{remark}
Loosely speaking, our result states that one can find a Polish space $\CX$ and a jointly continuous map $\CS_{\mathrm{R}}$
such that the following diagram commutes, where arrows without label denote the identity:
\begin{equ}[e:comm1]
\begin{minipage}[c]{.8\textwidth}
\centering
\begin{tikzpicture}[thick]
  \node (X) at ( 0,1.5) [] {$\CX$};
  \node at ( 0.5,1.5) [] {$\times$};
  \node (CC) at ( 1,1.5) [] {$\CC^\alpha$};
  \node (end) at ( 3.5,1.5) [] {$\CC(\R_+,\CC^\alpha)$};

  \node  at ( 2.25,0.75) [] {$\cdot$};

  \node (O) at ( 0,0) [] {$\Omega$};
  \node at ( 0.5,0) [] {$\times$};
  \node (CC1) at ( 1,0) [] {$\CC^\alpha$};
  \node (end1) at ( 3.5,0) [] {$\CC(\R_+,\CC^\alpha)$};
  \draw [->] (end) -- (end1);
  \draw [->] (O) -- node [left] {$\Psi$} (X);
  \draw [->] (CC1) -- (CC);
  \draw [->] (CC) --node [above] {$\CS_{\mathrm{R}}$} (end);
  \draw [->] (CC1) --node [above] {$\CS_{\mathrm{CH}}$}  (end1);
\end{tikzpicture}
\end{minipage}
\end{equ}
As it turns out, $\CS_{\mathrm{R}}$ also extends the usual (deterministic) notion $\CS_{\mathrm{D}}$
of solution to the KPZ equation with regular data:
\begin{equ}[e:KPZsmooth]
\d_t h = \d_x^2 h + \lambda (\d_x h)^2 + g(x,t)\;.
\end{equ}
In other words, it is possible to find a map $\Phi$ such that the following commutes:
\begin{equ}[e:comm2]
\begin{minipage}[c]{.8\textwidth}
\centering
\begin{tikzpicture}[thick]
  \node (X) at ( -0.5,1.5) [] {$\CX$};
  \node at ( 0.5,1.5) [] {$\times$};
  \node (CC) at ( 1,1.5) [] {$\CC^\alpha$};
  \node (end) at ( 3.5,1.5) [] {$\CC(\R_+,\CC^\alpha)$};

  \node  at ( 2.25,0.75) [] {$\cdot$};

  \node (O) at ( -0.5,0) [] {$\CC(\R_+,\CC)$};
  \node at ( 0.5,0) [] {$\times$};
  \node (CC1) at ( 1,0) [] {$\CC^\alpha$};
  \node (end1) at ( 3.5,0) [] {$\CC(\R_+,\CC^\alpha)$};
  \draw [->] (end) -- (end1);
  \draw [->] (O) -- node [left] {$\Phi$} (X);
  \draw [->] (CC1) -- (CC);
  \draw [->] (CC) --node [above] {$\CS_{\mathrm{R}}$} (end);
  \draw [->] (CC1) --node [above] {$\CS_{\mathrm{D}}$}  (end1);
\end{tikzpicture}
\end{minipage}
\end{equ}
where the first argument to $\CS_{\mathrm{D}}$ is the function $g$ in \eref{e:KPZsmooth}.
Interestingly, the choice of $\Phi$ in \eref{e:comm2}
is not unique. As we will see later, the map $\Psi$ in \eref{e:comm1} is given by the limit
in probability of maps $\Phi_\eps$ that are admissible for \eref{e:comm2}, applied to 
$\xi_\eps - C_\eps$ for a suitable mollification $\xi_\eps$ of $\xi$ and 
constant $C_\eps \to \infty$.
\end{remark}

\begin{remark}
The space $\CX$ will be given explicitly later on, but it is \textit{not} a linear space. 
It is indeed not difficult to convince oneself that, even though
the probability space associated to $\xi$ carries a natural linear structure (one could take it to be given by 
the space of distributions over $S^1 \times \R$ for example), it is not possible to find a norm on it that 
would make the map $\CS_{\mathrm{CH}}$ continuous.

Similarly, the reason why we did not simply formulate the statement of the theorem with $\CX$ replaced by $\CW$ from
the beginning is that, even though $\CS_{\mathrm{R}}$
is continuous on $\CX$, it does \textit{not} extend continuously to all of $\CW$.
\end{remark}

We also have a more explicit description of $\CS_{\mathrm{R}}$ as the solution to a fixed point argument, which 
in particular implies that the Cole-Hopf solutions of the KPZ equation can be realised as a \textit{continuous} random
dynamical system.
This can for example be formulated as follows:

\begin{proposition}\label{prop:mainCont}
Fix $\beta \in (0,{1\over 2})$. For every $T>0$ there exists a  Banach space $\CB_{\star,T}$ with a canonical
projection  
$\pi \colon \CB_{\star,T} \to \CC([0,T], \CC^{{3\over 2}-\beta})$, a closed algebraic variety
$\CY_{\star,T} \subset \CX \times \CB_{\star,T}$, continuous maps
$h^\star \colon \CX \to \CC([0,T], \CC^{{1\over 2}-\beta})$ and
\begin{equ}
\hat\CM \colon \CC^\beta \times \CY_{\star,T} \to \CB_{\star,T}\;,
\end{equ}
as well as a lower semi-continuous map $T_\star \colon \CC^\beta \times \CX \to (0,+\infty]$
with the following properties:
\begin{claim}
\item The map $\hat\CM$ leaves $\CY_{\star,T}$ invariant in the sense that
$(\Psi, \hat\CM(h,\Psi,v)) \in \CY_{\star,T}$ for every pair $(\Psi, v) \in \CY_{\star,T}$ and every $h \in \CC^\beta$.
\item For every $\Psi\in \CX$, the space $\CB_{\Psi,T} = \{v \in \CB_{\star,T}\,:\, (\Psi, v) \in \CY_{\star,T}\}$
is a Banach subspace of $\CB_{\star,T}$.
\item For every $h_0 \in \CC^\beta$, $\Psi_0 \in \CX$, and $T < T_\star(h_0, \Psi)$ and
neighbourhoods $U$ of $h_0$ and $V$ of $\Psi_0$ 
such that, for every $\Psi \in V$ and every $h \in U$, the restriction of $\CS_{\mathrm{R}}(h, \Psi)$ to $[0,T]$ 
can be decomposed as
\begin{equ}[e:decompCS]
\CS_{\mathrm{R}}(h, \Psi)\bigr|_{[0,T]} = h^\star(\Psi)\bigr|_{[0,T]} + \pi \hat\CS_{\mathrm{R}}(h, \Psi)\;,
\end{equ}
where $\hat \CS_{\mathrm{R}}\colon U\times V \to \CC([0,T],\CC^{{1\over2}-\beta})$ is a continuous map that is 
the unique solution in $\CB_{\Psi,T}$ to the fixed point problem
\begin{equ}
\hat\CM(h,\Psi,\hat \CS_{\mathrm{R}}(h, \Psi)) = \hat \CS_{\mathrm{R}}(h, \Psi)\;.
\end{equ}
\item If $T_\star(h,\Psi) < \infty$, then $\lim_{t \to T_\star} \|\CS_{\mathrm{R}}(h, \Psi)(t)\|_{\beta} = \infty$
for every $\beta > 0$.
\item There exists a group of continuous transformations $\Theta_t \colon \CX \to \CX$ such that $h^\star$ and $\hat\CS_{\mathrm{R}}$ satisfy the cocycle property in the sense that the identities
\begin{equ}
h^\star(\Psi)(t+s) = h^\star(\Theta_t \Psi)(s)\;,\quad  \hat \CS_{\mathrm{R}}(h, \Psi)(t+s) =  \hat \CS_{\mathrm{R}}\bigl( \hat \CS_{\mathrm{R}}(h, \Psi)(t), \Theta_t \Psi\bigr)(s)\;,
\end{equ}
hold for every $\Psi\in \CX$, every $h_0 \in \CC^\beta$ and every $s,t > 0$ with $s+t < T_\star$.
\end{claim}
\end{proposition}

\begin{remark}
The reason for requiring the decomposition \eref{e:decompCS} instead of writing $\CS_{\mathrm{R}}$ itself as a solution
to a fixed point problem is that $h^\star$ does not belong to $\CB_{\star,T}$ in general. 
Note also that, quite unusually in the theory of partial differential equations, the space $\CB_{\Psi,T}$ 
in which we effectively solve our fixed point problem depends on the choice of $\Psi$!
\end{remark}

\begin{remark}
In principle, Proposition~\ref{prop:mainCont} only provides a description of solutions up to the explosion
time $T_\star$. It is then natural to simply set $\CS_{\mathrm{R}}(h, \Psi)(t) = \infty$ for $t > T_\star(h,\Psi)$,
which yields a continuous path by the definition of the topology on $\bar \CC^{{1\over 2}-\beta}$ and the fact that
solutions explode when approaching $T_\star$. In order to prove Theorem~\ref{theo:mainCont}, it is therefore 
sufficient to construct  $\hat \CS_{\mathrm{R}}$ and $h^\star$ with the properties stated in Proposition~\ref{prop:mainCont}
and such that $\CS_{\mathrm{CH}} = \pi\hat \CS_{\mathrm{R}} + h^\star$ for every initial condition and almost every
realisation of $\Psi$. The fact that we know \textit{a priori} that Cole-Hopf solutions are defined for all times
ensures that, for every $h \in \CC^\beta$, one has $T_\star(h,\Psi(\omega)) = +\infty$ almost surely, but we
cannot rule out the existence of a non-trivial exceptional set that may depend on the initial condition.
\end{remark}

\begin{remark}
As an abstract result, it is not clear how useful Proposition~\ref{prop:mainCont} really is. However, we will
provide very explicit constructions of all the quantities appearing in its statement. As a consequence, in order
to approximate the Cole-Hopf solutions to \eref{e:KPZ}, it is enough to provide a good enough approximation
to the fixed point map $\CM$ in a suitable space, as well as an approximation to the map $\Psi$.
For an example of how such programme can be implemented in the context of a different equation
with similar regularity properties, see \cite{JanHendrik}.
\end{remark}

Another drawback of the Cole-Hopf solution
is that some properties of the solutions that seem natural in view of \eref{e:KPZ} turn out to be very
difficult to prove at the level of \eref{e:linear}. For example, due to the additive nature of the driving noise
in \eref{e:KPZ}, one would expect the difference between two solutions to exhibit better spatial and temporal
regularity properties than the solutions themselves. However, such a statement turns into a statement
about the regularity of the \textit{ratio} between solutions to \eref{e:linear}, which seems very difficult to obtain,
although some very recent progress was obtained in this direction in \cite{NeilJon}.

As a corollary of the construction of $\CS_{\mathrm{R}}$ however, we obtain extremely detailed information
about the solutions. In order to formulate our next result, we introduce the stationary mean zero
solution to the 
stochastic heat equation
\begin{equ}
\d_t X^\1 = \d_x^2 X^\1 + \Pi_0^\perp \xi\;,
\end{equ}
where $\Pi_0^\perp = 1-\Pi_0$, with $\Pi_0$ the orthogonal projection onto constant functions in $L^2$. (Adding 
this projection is necessary
in order to have a stationary solution.) Another vital role will be played by the process $\Phi$ given as the
centred stationary solution to 
\begin{equ}
\d_t \Phi = \d_x^2 \Phi + \d_x^2 X^\1\;.
\end{equ}
Note that, for any fixed time $t$, both $\Phi$ and $X^\1$ are equal in law to a Brownian bridge (in space!)  
which is centred so that its spatial average vanishes, but there are correlations between the two processes.

\begin{remark}
Both $X^\1$ and $\Phi$ are \textit{a priori} given as stochastic processes defined
on the underlying probability space $\Omega$. However, we will see below that $\CX$ is constructed in such
a way that there are natural counterparts to $X^\1$ and $\Phi$ that are \textit{continuous} functions
from $\CX$ into $\CC(\R, \CC^{{1\over 2}-\delta})$ for every $\delta > 0$.
\end{remark}

With this notation at hand, we have the following decomposition of the solutions:

\begin{theorem}\label{theo:decomposition}
Let $\beta > 0$ be arbitrarily small and,
 for $h_0 \in \CC^\beta$ and
$\Psi \in \CX$, set $h_t = \CS_{\mathrm{R}}(h_0, \Psi)(t)$ and $T_\star = \inf\{t>0\,:\, h_t = \infty\}$.

Then, for every $t<T_\star$, one has $h_t - X^\1_t \in \CC^{1-\beta}$. Furthermore,
there exists a continuous map $Q\colon \CX \to \CC(\R_+, \CC^{-\beta})$ such that
one has
\begin{equ}[e:decompRegular]
e^{-2\lambda \Phi_t} \d_x\bigl(h_t - X^\1_t\bigr) - Q_t \in \CC^{1-\beta}\;,
\end{equ}
for every $t < T_\star$.
\end{theorem}

\begin{proof}
In view of the construction of Section~\ref{sec:IdeasProofs}, this is an immediate consequence of
Proposition~\ref{prop:expPhi}, provided that we set
\begin{equ}
Q_t = \lambda e^{-2\lambda \Phi_t} \d_x \bigl( X^\2_t + \lambda X^\3_t + 4\lambda^2 X^\5_t\bigr) +8\lambda^4 \iint_0^\cdot e^{-2\lambda \Phi_t(z)} \d_x X^\5(z)\, d\Phi_t(z)\;.
\end{equ}
See Section~\ref{sec:IdeasProofs} for a definition of the expressions appearing here, as well
as Section~\ref{sec:roughpaths} for a definition of the ``rough integral'' $\iint$.
Actually, Proposition~\ref{prop:expPhi} provides an expression with two additional terms involving a
process $X^\4$, but since $X_t^\4 \in \CC^{2-\beta}$ and $\Phi_t \in \CC^{{3\over 2}-\beta}$ for every fixed $t$, 
one can check that the sum of these two terms
belongs to $\CC^{1-\beta}$.
\end{proof}

\begin{remark}
The product appearing on the left hand side of \eref{e:decompRegular} makes sense by Proposition~\ref{prop:Holder}
since $\Phi_t \in \CC^{{1\over 2}-\beta}$ and $\d_x\bigl(h_t - X^\1_t\bigr) \in \CC^{-\beta}$ for every $\beta > 0$.
\end{remark}

\begin{remark}
Together with the explicit construction of $Y$ given in Proposition~\ref{prop:expPhi} below, Theorem~\ref{theo:decomposition} 
provides a full description of the microscopic structure of the solutions to the KPZ equation, all the way down to
the ``level $\CC^{2-\beta}$'' for every $\beta > 0$.
\end{remark}

As a simple consequence of this decomposition, we also have a sharp regularity result for the difference between
two solutions with different initial conditions:

\begin{corollary}
Let $h_t$ and $\bar h_t$ be two solutions to \eref{e:KPZ} with different H\"older continuous initial conditions,
but driven by the same realisation of the noise. For every $\beta > 0$, one then has $h_t - \bar h_t \in \CC^{{3\over 2}-\beta}$ 
and
\begin{equ}[e:bounddiff]
e^{-2\lambda \Phi_t} \d_x\bigl(h_t - \bar h_t\bigr) \in \CC^{1-\beta}\;,
\end{equ}
for every $t$ less than the smaller of the two explosion times.
\end{corollary}

\begin{proof}
The bound \eref{e:bounddiff} follows immediately from \eref{e:decompRegular}. The fact that 
$h_t - \bar h_t \in \CC^{{3\over 2}-\beta}$  is then immediate since $\Psi_t \in \CC^{{1\over 2}-\beta}$.
\end{proof}

In a  recent article \cite{NeilJon}, O'Connell and Warren provided a
``multilayer extension'' of the solution to the stochastic heat equation \eref{e:linear}.
As a byproduct of their theory, it follows that $h_t - \bar h_t \in \CC^{1}$
so that Theorem~\ref{theo:decomposition} can be seen as a refinement of their results,
even though the decomposition considered there is quite different.
One object that arises in \cite{NeilJon} is the solution to
the linearised KPZ equation, namely
\begin{equ}[e:linearKPZ]
\d_t u = \d_x^2 u + \d_x u \, \d_x h\;,
\end{equ}
where $h$ is itself a solution to \eref{e:KPZ} (see equation~(20) in \cite{NeilJon}). 
One byproduct of our construction is that we are able to
provide a rigorous meaning to equations of the type \eref{e:linearKPZ} or, more generally,
equations of the type
\begin{equ}
\d_t u = \d_x^2 u + G(t,u,\d_x u)\,\d_x X^\1_t + F(t,u,\d_x u)\;,
\end{equ}
where $F$ and $G$ are suitable nonlinearities; see Theorem~\ref{theo:main} below.
In particular, this theorem also allows to provide a rigorous meaning for the Fokker-Planck equation
associated to a one-dimensional diffusion in the time-dependent potential $X^\1_t$, which
does not seem to be covered by existing techniques. Indeed, the well-posedness
of such a Fokker-Planck equation is quite well-known in the time independent case, also with
even weaker regularity assumptions, but the time-dependent case seems to be new and highly non-trivial. 
See for example \cite{MR2065168,MR2353387} for
some results in the time-independent case, as well as \cite{MR2450159} for some previously
known results that are very general (the authors allow non-constant diffusion coefficients
and higher space dimensions
for example), but do not appear to cover the situation at hand.

To conclude this introductory section, let us mention a few more byproducts of our 
construction that are of independent mathematical interest:
\begin{claim}
\item We provide an example of a two-dimensional ``geometric rough path'' which is obtained in a natural way by 
approximations by smooth paths but where, in order to obtain a well-defined limit, a logarithmically divergent ``area term''
needs to be subtracted, see Section~\ref{sec:controlProcess} below.
\item Since the map $\CS_{\mathrm{R}}$ is continuous, it does not depend on any choice of measure on $\CX$.
This allows us to use it for other type of convergence results, even in the case of deterministic drivers. 
As an example, we show in Section~\ref{sec:homogen} how to obtain a new periodic homogenisation result for
the heat equation with a very strong space-time periodic potential. 
\item It transpires that, besides the renormalisation $C_\eps \approx {1\over \eps}$ observed in \eref{e:KPZeps},
two further renormalisations, this time with logarithmically divergent constants, are lurking underneath. The reason
why this doesn't seem to have been observed before (and the reason why only $C_\eps$ appears when performing the
Cole-Hopf transform of the multiplicative stochastic heat equation)
is that these two logarithmically diverging constants cancel each other 
 \textit{exactly}, see Theorem~\ref{theo:convYeps} below. This appears to be due to a 
certain symmetry of the equation \eref{e:KPZ}
which may not hold in general for other equations in the same class.
\end{claim}
The remainder of this article is organised in the following way. 
In Section~\ref{sec:IdeasProofs}, we provide a more detailed mathematical formulation of the main results
of this article and we explain the main ideas arising in the proof. In particular, we provide an explicit description
of all the objects appearing in Proposition~\ref{prop:mainCont}.
In Section~\ref{sec:roughpaths}, we then introduce some the elements of the theory of (controlled) rough paths
that are essential to our proof. This section also contains some of the regularising bounds on the heat kernel that we
need in the sequel. This is followed in Section~\ref{sec:FP} by a solution theory for a class of rough stochastic PDEs
that includes the type of equation arising when taking the difference between two solutions to \eref{e:KPZ}.

In Section~\ref{sec:constructX}, we then build a ``universal process'' which provides a very good approximation
to the stationary solution to \eref{e:KPZ}, lying in the fourth Wiener chaos with respect to the driving noise $\xi$.
This process is centred by construction (so it really approximates the corresponding Burgers equation), so in
Section~\ref{sec:constantMode} we also construct its constant Fourier mode, in order to obtain an approximation to KPZ. 
Finally, in Section~\ref{sec:controlProcess} we provide a more detailed control of the local fluctuations of one of the
building blocks of the process built in Section~\ref{sec:constructX}.

\subsection{Notation}

We will often work with Fourier components. We adopt the usual convention
$X(x) = \sum_{k \in \Z} X_k e^{ikx}$,
so that one has the identity $(XY)_k = \sum_{\ell \in \Z} X_{\ell} Y_{k-\ell}$.
One feature of this normalisation is that the average of a function $X$ is equal to $X_0$
and the average of $|X|^2$ is given by $\sum_{k\in \Z}|X_k|^2$.

Throughout this article, we will consistently make use of H\"older seminorms, so that, for $X\colon S^1 \to \R$ and 
$\alpha \in (0,1]$, we set
\begin{equ}
\|X\|_\alpha \eqdef \sup_{x\neq y} {|\delta X(x,y)|\over |x-y|^\alpha}\;,
\end{equ}
where $\delta X(x,y) = X(y) - X(x)$. 
We will also extend this to negative values of $\alpha$. For $\alpha \in (-1,0)$, we set
\begin{equ}
\|X\|_\alpha \eqdef \sup_{x\neq y} {|\int_x^y X(z)\,dz|\over |x-y|^{1+\alpha}}\;,
\end{equ}
and we denote by $\CC^{\alpha}$ the space of distributions obtained by closing $\CC^\infty$ under the above norm.
We make a slight abuse of notation for the supremum norm by also writing
\begin{equ}
\|X\|_\infty \eqdef \sup_{x} {|X(x)|}\;.
\end{equ}
We are sometimes lead to consider H\"older norms instead of seminorms, so we set
\begin{equ}
\|X\|_{\CC^\alpha} \eqdef \|X\|_\alpha + \|X\|_\infty\;.
\end{equ}

A crucial ingredient in the theory of (controlled) rough paths used in this article is played by ``area processes'' 
and ``remainder terms'', both of which are
functions of \textit{two} spatial variables. For such functions, we also set
\begin{equ}[e:Holder]
\|\XX\|_{\alpha} \eqdef  \sup_{x \neq y} {|\XX(x,y)|\over |x-y|^{\alpha}}\;,
\end{equ}
which is a kind of H\"older seminorm ``on the diagonal'' and we denote by $\CC_2^\alpha$ the closure of
the space of smooth functions of two variables under the norm $\|\XX\|_{\CC_2^\alpha} = \|\XX\|_{\alpha} + \|\XX\|_{\infty}$.
The advantage of only ever considering the closures of $\CC^\infty$ under the above norms has the advantage that
all the spaces appearing in this article are separable, so that no problem of measurability arises.

\subsection*{Acknowledgements}

I would like to thank Sigurd Assing, Jan Maas, Neil O'Connell, Jon Warren, and 
Jeremy Quastel for numerous discussions on this and
related problems that helped deepen and 
clarifying the arguments presented here. Special thanks are due to Hendrik Weber for
numerous suggestions and his careful reading of the draft manuscript, 
as well as to G\'erard Ben Arous who suggested that the techniques developed
in \cite{SemiPert,BurgersRough} might prove 
useful for analysing the KPZ equation. 

Financial support was kindly provided by EPSRC grant EP/D071593/1, by the 
Royal Society through a Wolfson Research Merit Award, and by the Leverhulme Trust through a Philip Leverhulme Prize.

\section{Main results and ideas of proof}
\label{sec:IdeasProofs}

The idea pursued in this article is to solve \eref{e:KPZ} by performing a Wild expansion \cite{MR0042999} of 
the solution in powers of $\lambda$
but, instead of deriving an infinite series that may
be extremely difficult to sum, we truncate it at a fixed level (after exactly $4$ terms to be precise) and then use
completely different techniques to treat the remainder.
In order to appreciate how the techniques explained in this section can also apply to a concrete deterministic 
example, it may be helpful
to simultaneously follow the calculations in Section~\ref{sec:homogen} below.

Recall that 
$X^\1_\eps$ is the stationary mean zero solution to the linearised equation
\begin{equ}
\d_t X^\1_\eps = \d_x^2 X^\1_\eps + \Pi_0^\perp \xi_\eps\;.
\end{equ}
Here, the noise process $\xi_\eps$ is a mollified version of $\xi$, obtained by choosing a function $\phi \colon \R \to \R_+$ that is even, smooth, compactly supported, decreasing on $\R_+$, and
such that $\phi(0) = 1$, and then setting
\begin{equ}
\xi_{\eps,k} = \phi(\eps k)\,\xi_k\;.
\end{equ}
The $\xi_k$ are the Fourier components of $\xi$, which are complex-valued white noises with $\xi_{-k} = \bar \xi_k$ and
$\E \xi_k(s)\xi_\ell(t) = 2\delta_{k,-\ell}\delta(t-s)$.
The above properties of the mollifier $\phi$ will be assumed throughout the whole article without further mention.

A crucial ingredient of the construction performed in this article is a 
family $X^\tau_\eps$ of processes indexed by binary trees $\tau$, where ``$\bullet$'' denotes 
the ``trivial'' tree consisting of only its root. The process $X^\1_\eps$ associated to the trivial tree has already been
defined, and we define the remaining processes recursively as follows.
Denoting by $\CT_2$ the set of all binary trees, any binary tree 
$\tau \in \CT_2$ with $\tau \neq \bullet$ can be written as $\tau = [\tau_1,\tau_2]$, i.e. $\tau$ consists of 
its root, with trees $\tau_1,\tau_2 \in \CT_2$ attached. 
For any such tree $\tau$, we then define $X^\tau_\eps$ as the stationary solution to
\begin{equ}[e:defXtau]
\d_t X^\tau_\eps = \d_x^2 X^\tau_\eps + \Pi_0^\perp \bigl(\d_x X^{\tau_1}_\eps\, \d_x X^{\tau_2}_\eps\bigr)\;.
\end{equ}
\begin{remark}
As before, the reason why we introduce the projection $\Pi_0^\perp$ is so that we can consider stationary solutions. 
Another possibility 
would have been to slightly modify the equation to replace $\d_x^2$ by $\d_x^2 -1$ for example, but it turns out that the
current choice leads to simpler expressions. Since we only have derivatives of $X^\tau_\eps$ appearing in \eref{e:defXtau} anyway,
the effect of $\Pi_0^\perp$ turns out to be rather harmless, see Remark~\ref{rem:XvsY}.
\end{remark}

We now add the constant terms back in. Set $Y_\eps^\1(t) = X_\eps^\1(t) + \sqrt2 B(t)$, where $B$ is a standard Brownian motion.
One of the main results of this article is that one can then find constants $C_\eps^\tau$ for $\tau \neq \bullet$ such that  
the solutions $Y^\tau_\eps$ to
\begin{equ}[e:defYtau]
\d_t Y^\tau_\eps = \d_x^2 Y^\tau_\eps + \d_x Y^{\tau_1}_\eps\, \d_x Y^{\tau_2}_\eps - C^\tau_\eps\;,
\end{equ}
with initial condition $Y^\tau_\eps(0) = X^\tau_\eps(0)$, have a limit as $\eps \to 0$ that is independent of the choice of mollifier $\phi$.

\begin{remark}\label{rem:XvsY}
Since only derivatives of $Y_\eps^\tau$ appear on the right hand side of \eref{e:defYtau}, it follows that
$X^\tau_\eps = \Pi_0^\perp Y^\tau_\eps$. The reason for introducing the processes $X^\tau_\eps$ is that it
is easier, as a first step, to show that they converge to a limit. The constant Fourier mode will then be treated 
separately.
\end{remark}

The reason for the definition of the processes $Y^\tau_\eps$ is that, at least at a formal level, if one defines 
a process $h_\eps(t)$ by
\begin{equ}[e:sumTrees]
h_\eps(t) = \sum_{\tau} \lambda^{|\tau|} Y^\tau_\eps(t)\;,
\end{equ}
where $|\tau|$ denotes the number of inner nodes of $\tau$ (i.e.\ the number of nodes that are not leaves, with $|\bullet| = 0$
by convention),
then $h_\eps$ solves the equation
\begin{equ}
\d_t h_\eps = \d_x^2 h_\eps + \lambda (\d_x h_\eps)^2 + \xi_\eps - \sum_{\tau} \lambda^{|\tau|} C_\eps^\tau \;,
\end{equ}
which is precisely \eref{e:KPZeps}.
The problem with such an approach is twofold: first, we have no guarantee that the sum \eref{e:sumTrees} actually converges.
Then, even if it did converge for fixed $\eps > 0$, we would have no guarantee that the sequence of processes $h_\eps$ constructed
in this way converges to a limit, even if we knew that each of the $Y_\eps^\tau$ converges.
See however \cite{MR0042999,MR0224348,MR1725612} for an analysis of the corresponding expansion in the context
of the Boltzmann equation, where the sum over all binary trees can actually be shown to converge.

The strategy pursued in this work is to truncate the expansion \eref{e:sumTrees} at a fixed level and
to then derive an equation for the remainder that can be solved by using techniques inspired from \cite{BurgersRough}.

\subsection{Convergence of the processes $Y^\tau$}

In this section, we state the precise convergence result that we obtain for the processes $Y^\tau_\eps$.
The choice for $C_\eps^\tau$  that we retain 
is $C_\eps^\tau = 0$ for $\tau \not \in \{\2, \4, \5, \6, \7, \8\}$, and
\begin{equs}[e:defCtau]
C_\eps^\2 &= {1\over \eps} \int_{\R} \phi^2(x)\,dx\;,\\
C_\eps^\4 &= {4\pi \over \sqrt 3}|\log \eps|- 8\int_{\R_+} \int_\R {x \phi'(y)\phi(y)\phi^2(y-x) \log y \over x^2 -xy +y^2} \,dx\,dy\;,\\
 C_\eps^\5 &= -{C_\eps^\4 \over 4}\;.
\end{equs}
The remaining trees are of course all equivalent to the tree $\bigtree$ and therefore are associated with
the same constant. It is remarkable that $C_\eps^\4$ and $C_\eps^\5$ turn out to exhibit the 
exact same logarithmic divergence but with opposite signs, 
save for the factor $4$ that takes into account
the difference in multiplicities between the two terms.
It is also remarkable that, even though the constants $C_\eps^\4$ and $C_\eps^\5$ do depend on the
choice of mollifier $\phi$, the resulting processes $Y^\tau$ do not.

As a consequence of this choice,  note also that we have the identity
\begin{equ}[e:sumconstants]
\sum_{\tau} \lambda^{|\tau|} C_\eps^\tau = \lambda C_\eps^\2\;,
\end{equ}
so that, at least formally, the process $h_\eps$ solves \eref{e:KPZeps} for the ``correct'' constant $C_\eps$.
Before we state our convergence result, we introduce some more notation. For every $\tau$, we define
an exponent $\alpha_\tau$ by
\begin{equ}
\alpha_\1 = {1\over 2}\;,\qquad \alpha_\2 = 1\;,
\end{equ}
and then, recursively, by
\begin{equ}
\alpha_{[\tau_1,\tau_2]} = (\alpha_{\tau_1}\wedge\alpha_{\tau_2}) + 1\;.
\end{equ}
(So we have for example $\alpha_\3 = {3\over 2}$ and $\alpha_\4 = 2$.)
For $\tau \neq \bullet$, we then 
define the separable 
Fr\'echet space $\CX_\tau$ as the closure of smooth functions under the system of seminorms
\begin{equ}[e:defXtauSpace]
\|X\|_{\tau,\delta,T} = \sup_{s,t \in [-T,T]} \Bigl(\|X(t)\|_{\CC^{\alpha_\tau-\delta}} + {\|X(t)-X(s)\|_\infty \over |t-s|^{{1\over 2}-\delta}}\Bigr)\;,
\end{equ}
where $T \in [1,\infty)$ and $\delta \in (0,{1\over 4})$.
Similarly, we define $\CX_\1$ as the closure of smooth functions under the system of seminorms
\begin{equ}
\|X\|_{\1,\delta} = \sup_{|t-s| \in (0,1]}\Bigl( {\|X(t)-X(s)\|_{\infty}\over |t-s|^{{1\over 4}-\delta} (1+|t|)} + {\|X(t)\|_{\CC^{{1\over 2}-\delta}} \over 1+|t|}\Bigr)\;,
\end{equ}
for $\delta \in (0,{1\over 4})$.

With these definitions, our precise convergence result for the processes $Y^\tau_\eps$ is the following, which we will prove at the end of Section~\ref{sec:controlProcess}.

\begin{theorem}\label{theo:convYeps}
Let $Y_\eps^\tau$ be as in \eref{e:defYtau} and let $\CX_\tau$ be as above. Then, for every binary tree $\tau$, there 
exists a process $Y^\tau$ such that $Y^\tau_\eps \to Y^\tau$ in probability in $\CX_\tau$.
\end{theorem}

\begin{remark}
We believe that in the definition \eref{e:defXtauSpace}, we could actually have imposed time regularity of
order $\alpha_\tau/2$ instead of $1/2$.
\end{remark}

\subsection{Treatment of the remainder}

The truncation of \eref{e:sumTrees} that turns out to be the shortest ``viable'' one is as follows. 
Setting $\bar \CT =  \{\bullet,\2, \3, \9, \4, \5, \6, \7, \8\}$,
we look for solutions to \eref{e:KPZeps} of the form
\begin{equ}[e:defhstar]
h_\eps(t) = \sum_{\tau \in \bar \CT} \lambda^{|\tau|} Y^\tau_\eps(t) + u_\eps(t) \eqdef h^\star_\eps(t) + u_\eps(t)\;,
\end{equ}
for a remainder $u_\eps$. In the sequel, since the processes $Y^\tau_\eps$ mostly appear via their
spatial derivatives, we set
\begin{equ}[e:defYbar]
\bar Y^\tau_\eps \eqdef \d_x Y^\tau_\eps\;,
\end{equ}
as a shorthand. With this notation, we have the following result for $h^\star_\eps$:

\begin{proposition}\label{prop:sol}
The process $h^\star_\eps$ defined above is the stationary solution to
\begin{equ}
\d_t h^\star_\eps = \d_x^2 h^\star_\eps + \lambda\bigl(\d_x h^\star_\eps\bigr)^2 + \xi_\eps - \lambda C_\eps^\2 - \CR_\eps^\star\;,
\end{equ}
where the remainder term $\CR_\eps^\star$ is given by
\begin{equ}
\CR_\eps^\star = \sum_{\tau,\kappa \in \bar \CT \atop [\tau,\kappa] \not \in \bar \CT}\lambda^{|\tau|+|\kappa|+1} \bar Y_\eps^\tau\, \bar Y_\eps^\kappa\;.
\end{equ}
\end{proposition}

\begin{proof}
It follows from the definition of $Y_\eps^\tau$ and from the fact that $\bullet \in \bar \CT$ that
\begin{equ}
\d_t h^\star_\eps =  \d_x^2 h^\star_\eps + \sum_{\tau \in \bar \CT\setminus \{\bullet\} \atop \tau = [\tau_1,\tau_2]}\lambda^{|\tau|} \bar Y_\eps^{\tau_1}\, \bar Y_\eps^{\tau_2}+ \xi_\eps - \sum_{\tau \in \bar \CT} C_\eps^\tau\;.
\end{equ}
The claim now follows at once  from the identity 
\begin{equ}
\lambda \bigl(\d_x h_\eps^\star \bigr)^2  = \sum_{\tau, \kappa \in \bar \CT} \lambda^{|\tau|+|\kappa|+1}\bar Y_\eps^{\tau}\, \bar Y_\eps^{\kappa}\;,
\end{equ}
noting that $|[\tau,\kappa]| = |\tau|+|\kappa|+1$ and that
$\bar \CT \setminus \{\bullet\} \subset \{[\tau,\kappa]\,:\, \tau, \kappa\in \bar \CT\}$ by inspection.
\end{proof}

As a consequence of Proposition~\ref{prop:sol}, if we want $h_\eps$ to satisfy \eref{e:KPZeps}, we should take $u_\eps$ to be the solution to
\begin{equ}[e:equationuepsorig]
\d_t u_\eps = \d_x^2 u_\eps + \lambda (\d_x u_\eps)^2 + 2\lambda \d_x u_\eps\, \d_x h_\eps^\star  +  \CR_\eps^\star\;.
\end{equ}
Actually, it turns out to be advantageous to regroup the terms on the right hand side of this equation in a slightly different
way, by isolating those terms that contain an occurrence of $Y_\eps^\1$. We thus write $h_\eps^\star = Y_\eps^\1 + \bar h_\eps^\star$,
as well as
\begin{equ}
\CR_\eps^\star = 2\lambda^4 \bar Y_\eps^\1 \bigl(\bar Y_\eps^\4 + 4\bar Y_\eps^\5\bigr) + \bar \CR_\eps^\star\;,
\end{equ}
with
\begin{equs}
\bar \CR_\eps^\star &= \lambda^5 \bigl(2\,\bar Y_\eps^\2\, \bar Y_\eps^\4 + 8\,\bar Y_\eps^\2\, \bar Y_\eps^\5+\bar Y_\eps^\3\, \bar Y_\eps^\3 \bigr)
 + \lambda^6 \bigl(2\, \bar Y_\eps^\3\, \bar Y_\eps^\4+ 8\, \bar Y_\eps^\3\, \bar Y_\eps^\5\bigr) \\
&\quad  + \lambda^7 \bigl(\bar Y_\eps^\4\, \bar Y_\eps^\4 + 8\, \bar Y_\eps^\4\, \bar Y_\eps^\5 + 16\, \bar Y_\eps^\5\, \bar Y_\eps^\5\bigr)\;.\label{e:defRbar}
\end{equs}
The precise form of $\bar \CR_\eps^\star$ is actually irrelevant. The important fact is that one should retain 
from this expression is that, by 
combining Theorem~\ref{theo:convYeps} with Proposition~\ref{prop:Holder}, there is a limiting process $\bar \CR^\star$
such that $\bar \CR_\eps^\star \to \bar \CR^\star$ in probability in $\CC(\R,\CC^{-\beta})$ for every $\beta > 0$.

With these notations, \eref{e:equationuepsorig} can be rewritten as
\begin{equs}
\d_t u_\eps &= \d_x^2 u_\eps + 2 \lambda \bar Y_\eps^\1 \bigl(\d_x u_\eps + \lambda^3 \bar Y_\eps^\4 + 4\lambda^3\bar Y_\eps^\5\bigr) \\
&\qquad + \lambda (\d_x u_\eps)^2 + 2\lambda \d_x u_\eps \, \d_x \bar h_\eps^\star  + \bar \CR_\eps^\star\;.\label{e:equationu}
\end{equs}
Since, by Theorem~\ref{theo:convYeps}, $\bar h_\eps^\star$ is continuous with values in $\CC^{1-\delta}$ for every 
$\delta > 0$, it follows that if we are able to find a solution $u_\eps$ taking values in $\CC^\alpha$ for some $\alpha > 1$ with
a uniform bound as $\eps \to 0$, there is no problem in making
sense of the terms on the second line of this equation in the limit $\eps \to 0$.

The problem of course is the second term. Indeed, since $\bar Y^\1 \in \CC^{\gamma}$ only for $\gamma < -{1\over 2}$, we
would need $u_\eps(t)$ to converge in $\CC^\alpha$ for $\alpha > {3\over 2}$ for this term to make sense in the limit (see Remark~\ref{rem:Holder} below). This however is hopeless since, by the usual maximal regularity results, the 
action of the heat semigroup allows us to gain only two spatial derivatives so that the best we can hope for is that 
$u_\eps(t)$ converges in $\CC^\alpha$ precisely for
every $\alpha < {3\over 2}$ only!

This is where the theory of rough paths comes into play. Denote by $v_\eps$ the derivative of 
$u_\eps$, so that \eref{e:equationu} becomes
\begin{equ}[e:equationv]
\d_t v_{\eps} = \d_x^2 v_{\eps} + 2\lambda \d_x \bigl(\bar Y_\eps^\1 \bigl(v_{\eps} +  \lambda^3 \bar Y_\eps^\4 + 4\lambda^3\bar Y_\eps^\5\bigr)\bigr) + \d_x F_\eps(v_{\eps},t)\;,
\end{equ}
where the nonlinearity $F_\eps$ is given by
\begin{equ}
F_\eps = \lambda v^2_\eps + 2\lambda v_\eps \, \d_x \bar h_\eps^\star + \bar \CR_\eps^\star\;.
\end{equ}
As already mentioned, this nonlinearity is expected to be ``nice'', in the sense that we can use classical functional analysis
to make sense of it as $\eps \to 0$, so that we do not consider it for the moment and will treat it as
a perturbation later on.

If the right hand side of \eref{e:equationv} were well-posed in
the limit $\eps \to 0$, we would expect the solution $v_\eps$ to look at small scales like the solution $\Phi_\eps$ to
\begin{equ}[e:defPhieps]
\d_t \Phi_\eps = \d_x^2 \Phi_\eps + \d_x^2 Y_\eps^\1\;,
\end{equ}
so we define $\Phi_\eps$ by
\begin{equ}[e:realdefPhi]
\Phi_{\eps,t} = \int_{-\infty}^t P_{t-s}\, \d_x^2 Y_{\eps,s}^\1\,ds\;,
\end{equ}
where $P_t$ is the heat semigroup.
Since $\d_x^2 Y_{\eps,s}^\1$ has zero average, this is well-defined as long as $Y_{\eps,s}^\1$ does not grow
too fast for large times.

The idea now is to try to solve \eref{e:equationv} in a space of functions that are ``controlled by $\Phi$'' in the sense that 
there exists a function $v'$ such that the ``remainder term''
\begin{equ}[e:remainder]
R^v_{\eps,t}(x,y) = \delta v_{\eps,t}(x,y) - v'_{\eps,t}(x) \,\delta \Phi_{\eps,t}(x,y)\;,
\end{equ}
satisfies a bound of the type $\|R^v_{\eps,t}\|_\alpha < \infty$, uniformly as $\eps \to 0$, for some $\alpha > {1\over 2}$.
Here, we have made use of the shorthand notation $\delta v(x,y) = v(y)-v(x)$ and similarly for $\delta \Phi$. This notation will be used repeatedly in the sequel.
What a bound like \eref{e:remainder} tells us is that, at very small scales, $v$ looks like some multiple of $\Phi$, modulo a remainder
term that behaves as if it was $\alpha$-H\"older for some $\alpha > {1\over 2}$. Note that this is a purely local property of the increments. 

This suggests that, if we were able to show ``by hand'' that  $\Phi_{\eps,t} \, \d_x Y_{\eps,t}^\1$ 
converges to a limiting distribution as $\eps \to 0$, then one may be
able to use this knowledge to give a meaning to the expression
$v\, \d_x Y^\1$ for those functions $v$ admitting a ``derivative process'' $v'$
such that the remainder $R^v_{t}(x,y)$ defined as in \eref{e:remainder} satisfies $\|R^v_{t}\|_\alpha < \infty$
for some $\alpha > {1\over 2}$. This is precisely what the theory of controlled rough paths \cite{Max} allows
us to do. For any fixed $t$, let  $\YY_t$ be the function of two variables defined by
\begin{equ}[e:defYY]
\YY_{\eps,t}(x,y) = \int_x^y \delta \Phi_{\eps,t}(x,z)\,  dY_{\eps,t}^\1(z)\;.
\end{equ}
It is important to note that, for every $t$ and every $\eps$, $\YY_{\eps,t}$ satisfies the algebraic relation
\begin{equ}[e:relYY]
\YY_{\eps,t}(x,z) - \YY_{\eps,t}(x,y) - \YY_{\eps,t}(y,z) =  \delta \Phi_{\eps,t}(x,y)\,\delta Y_{\eps,t}^\1(y,z)\;,
\end{equ}
for every $x,y,z \in S^1$. One can then show, and this is the content of Proposition~\ref{prop:convYY}
below, that there exists a process $\YY$ with values in $\CC_2^{\gamma}$ such that 
$\YY_\eps \to \YY$ in probability in $\CC(\R, \CC_2^{\gamma})$ for every $\gamma < 1$.

We refer to Section~\ref{sec:roughpaths} below for more details, but the gist of the theory of controlled rough paths
is that
one can use the process $\YY$ in order to define a ``rough integral'' $\iint_x^yA_t(z)\, dY_{t}^\1(z)$ as a convergent limit of 
compensated Riemann sums for every smooth test function $\phi$
and for every function $A_t$ such that, for some $\delta > 0$, there exists $A'_t \in \CC^{\delta}$ and 
$R^A_t \in \CC_2^{{1/ 2}+\delta}$ with
\begin{equ}[e:remainderA]
R^A_{t}(x,y) = \delta A_t(x,y) - A'_t(x) \,\delta \Phi_{t}(x,y)\;.
\end{equ}
See Theorem~\ref{theo:integral} below for a precise formulation of this statement. 
It is important to note at this stage that the notation $\iint$ used for the rough integral is really an abuse of notation.
Indeed, it does in general depend not just on $A$ and $Y$, but also on a choice of $\YY$ satisfying 
\eref{e:relYY}, as well as on the choice of $A'$ in \eref{e:remainderA}. It is only when $\YY$ is actually given by
\eref{e:defYY} that it coincides with the Riemann integral, independently of the choice of $A'$.
See equation~\ref{e:corint} below for more details.

\begin{remark}
A number of recent results have made use of the theory of rough paths to treat classes of stochastic PDEs,
see for example \cite{PeterMich,Oberhaus,MaxSam,Josef}. In all of these cases, the theory of rough paths was used to deal with the lack
of \textit{temporal} regularity of the equations. In this article, as in \cite{BurgersRough,Hendrik}, we use it instead
in order to deal with the lack of \textit{spatial} regularity.
\end{remark}

In this way, we can indeed make sense of the product $v_t\, \d_x Y_{t}^\1$ as a distribution, provided that
$v_t$ admits a sufficiently regular decomposition as in \eref{e:remainderA} 
for some ``derivative process'' $v_t'$. In a way, this is reminiscent of the technique of ``two-scale convergence''
developed in \cite{MR990867,twoscale}. The main differences are that it does not require any periodicity at the small scale and
that it does not rely on any explicit small parameter $\eps$, both of which make it particularly adapted to situations where the
small-scale fluctuations are random. See however Section~\ref{sec:homogen} below for an example with deterministic periodic
data where the results of this article also apply.

The same theory can also be used in order to make sense of the term $\bar Y_t^\5 \,\bar Y_t^\1$ in \eref{e:equationv}. 
It is indeed possible to show that $\bar Y_t^\5$ is controlled by $\bar Y_t^\1$ in the sense that
the process $R_t^\5$ defined by
\begin{equ}[e:remainder5]
R^\5_{t}(x,y) = \delta \bar Y_t^\5(x,y) - \bar Y_t^\3(x) \,\delta \Phi_{t}(x,y)\;,
\end{equ}
takes values in $\CC_2^{{1\over 2}+\zeta}$ for some $\zeta > 0$. Furthermore, the corresponding processes for $\eps>0$ 
do converge
to $R^\5$ in that topology, which turns out to be surprisingly difficult to prove, see 
Theorem~\ref{theo:controlX5} below. In view of all of these convergence results, the space $\CX$ and the
map $\Psi\colon \Omega \to \CX$
appearing in Theorem~\ref{theo:mainCont} and Proposition~\ref{prop:mainCont} are then defined as follows:

\begin{definition}
Setting $\bar \CT_0 = \{\bullet,\2, \3, \4, \5\}$, the Fr\'echet space $\CW$ is given by
\begin{equ}
\CW = \Bigl(\bigoplus_{\tau \in \bar \CT_0} \CX_\tau \Bigr) \oplus \CC(\R, \CC_2^{3\over 4})  \oplus \CC(\R, \CC_2^{3\over 4})\;,
\end{equ}
and the map $\Psi\colon \Omega \to \CW$ is given by the random variable
\begin{equ}[e:notationPsi]
\Psi = \Bigl(\bigoplus_{\tau \in \bar \CT_0} Y^\tau \Bigr) \oplus \YY  \oplus R^\5\;.
\end{equ} 
The space $\CX\subset \CW$ is then defined as the algebraic variety determined by the relations \eref{e:relYY} and \eref{e:remainder5}. Since $\CX$ is closed (as a subset of $\CW$) and $\Psi$ is the limit in probability of 
maps $\Psi_\eps$ which map $\Omega$ into $\CX$, one automatically has $\Psi(\omega) \in \CX$ for almost every $\omega$.
\end{definition}

We now have all the ingredients necessary to reformulate \eref{e:equationv} as a fixed point map by considering its mild
formulation. We will then turn this into a fixed point argument for \eref{e:equationu}, which is equivalent save for the
constant Fourier mode.
Using the variation of constants formula, we can rewrite solutions to \eref{e:equationv} for every fixed realisation
of $\{Y^\tau\}_{\tau \in \bar \CT}$ and every fixed initial condition $v_0$ as
\begin{equ}
v_{\eps} = \CK_0(v_\eps)\;,
\end{equ}
where the map $\CK_0$ is given by
\begin{equs}
\bigl(\CK_0 v\bigr)_t &= P_t v_0 + 2\lambda\, \d_x \int_0^t P_{t-s} \Bigl(\bigl(v_s + 4\lambda^3 \bar Y_{\eps,s}^\5\bigr)\,\bar Y_{\eps,s}^\1\Bigr)\,ds \\
&\qquad + \d_x \int_0^t P_{t-s} \bigl(\lambda^4\bar Y_{\eps,s}^\4 \, \bar Y_{\eps,s}^\1  + F_\eps(v_\eps,s)\bigr)\,ds\;,
\end{equs}
were $P_t$ denotes the heat semigroup, the kernel of which we will denote by $p_t$. For any given \textit{smooth} data 
$\{Y^\tau\}_{\tau \in \bar \CT}$, the map $\CK_0$ is well-defined as a map from the set of smooth functions $v$ into itself.
The problem with $\CK_0$ is that it is not possible to extend it to sufficiently large functional spaces by performing
a classical completion procedure. 
The idea is therefore to first extend its definition to smooth input data $\Psi = (\{Y^\tau\}_{\tau \in \bar \CT}, \YY, R^\5) \in \CX$,
and to smooth triples $V = (v,v',R^v)$ such that the additional algebraic relations \eref{e:remainder} are satisfied, by setting
\begin{equs}
\CK (v_0,V,\Psi)_t &= P_t v_0 + 2\lambda \int_0^t\! \iint_{S^1} p'_{t-s}(\cdot -y) \bigl(v_s(y) + 4\lambda^3 \bar Y_{s}^\5(y)\bigr)\,dY_{s}^\1(y)\,ds \\
&\qquad +  \d_x \int_0^t P_{t-s} \bigl(\lambda^4 \bar Y_{s}^\4 \, \bar Y_{s}^\1  + F(v,\Psi,s)\bigr)\,ds\;,\label{e:defMstart}
\end{equs}
where the ``rough integral'' 
$\iint$ is defined as in \eref{e:defintfinal} below and where we set as before
\begin{equ}[e:defFs]
F(v,\Psi,s) = \lambda v^2_s + 2\lambda v_s \, \d_x \bar h^\star(\Psi)_s + \bar \CR^\star(\Psi)_s\;,
\end{equ}
with $h^\star$ and $\bar \CR^\star$ given by \eref{e:defhstar} and \eref{e:defRbar} respectively.
Actually, the precise definition of $\iint$ really does not matter at this stage. Indeed,
if we denote by $\CX_s \subset \CX$ the set of smooth elements in $\CX$ such that $\YY$ is given
by \eref{e:defYY} and $R^\5$ is given by \eref{e:remainder5}, then $\iint$ coincides with the usual
Riemann integral and therefore
$\CK$ coincides with $\CK_0$ on $\CX_s$. Furthermore, by Proposition~\ref{prop:continuous} below, 
$\CX_s$ is dense in $\CX$ and, as we will see in Theorem~\ref{theo:mainTechnical} below, $\CK$ is the 
unique continuous extension of $\CK_0$ to $\CX$.
In this sense, we have \textit{not} changed
the classical notion of a smooth solution to \eref{e:equationv} at all, but have simply extended it to a larger class
of input data.

\begin{remark}
If $\YY$ is defined differently from \eref{e:defYY}, \textit{even if it is smooth},
we obtain different solutions, see Section~\ref{sec:smooth} below.
While these different solutions may appear ``unphysical'' at first sight, they actually have a clear interpretation in terms
of limiting points of solutions to the KPZ equation with highly oscillatory data, see Section~\ref{sec:homogen}
for an explicit example. 
\end{remark}

For fixed $\kappa > 0$ (small enough as we will see shortly) and $T>0$, denote now by $\CB_{\star,T}$ the closure of the space of smooth quadruples 
$V=(m,v,v',R^v)$, where $m$ is a real-valued function of time only, $v$ and $v'$ are functions of time and space,
and $R^v$ is a function of time and two spatial variables, under the norm
\begin{equs}
\|m,v,v',R^v\|_{\star,T} &= \sup_{t\in (0,T]} t^{1-2\kappa} \bigl(\|v_t\|_{{1\over 2}-\kappa} + \|v'_t\|_{\CC^{3\kappa}} + 
\|R_t^v\|_{{1\over 2}+2\kappa} + t^{{3\kappa-1\over 2}}\|v_t\|_\infty\bigr) \\
&\qquad + \sup_{0<s< t\le T} {s^{1-2\kappa}\over |t-s|^{2\kappa}} \|v_t-v_s\|_\infty + \sup_{t \in (0,T]} |m_t|\;. 
\end{equs}
Here, $m_t$ is interpreted as the spatial mean of $u_t$, so that the natural projection map $\pi\colon \CB_{\star,T} \to \CC([0,T],\CC^{{3\over 2}-\kappa})$ recovering $u$ from $V$ is given by
\begin{equ}
(\pi V)_t = \CI v_t + m_t\;,
\end{equ}
where $\CI$ is the integration operator given by the Fourier multiplier $(1-\delta_{k,0})/(ik)$.

We furthermore denote by $\CY_{\star,T}$ the closed algebraic variety in $\CX \oplus \CB_{\star,T}$
determined by the additional relation \eref{e:remainder}. Note that $\CY_{\star,T}$ is again a Polish space
equipped with a natural metric given by the restriction of the product norm on $\CW \oplus \CB_{\star,T}$.
We now use $\CK$ as a building block for the map $\hat \CM$ appearing in Proposition~\ref{prop:mainCont} 
in the following way. 
For any smooth element $(h_0,\Psi,V) \in \CC^\infty \times \CY_{\star,T}$, and using furthermore
the shorthand notation $V = (m,v,v',R^v)$, we set
\begin{equ}
\hat \CM(h_0, \Psi,V) = \Bigl(\CJ(V,\Psi), \CK\bigl(\d_x(h_0 - h^\star_0(\Psi)),V,\Psi\bigr), \CK'(V,\Psi), R^\CM\Bigr)\;,
\end{equ}
where $\CK$ is as in \eref{e:defMstart}, $\CK'$ is given by
\begin{equ}[e:defKprime]
\CK'(V,\Psi) = 2\lambda \bigl(v + 4\lambda^3 \bar Y^\5+ \lambda^3 \bar Y^\4\bigr)\;,
\end{equ}
$R^\CM$ is defined by the relation \eref{e:remainder}, and
\begin{equs}
\CJ(V,\Psi)_t &= \Pi_0 \bigl(h_0 - h_0^\star(\Psi)\bigr) + \int_0^t \Pi_0 F(v,\Psi,s)\,ds \\
&\qquad + {\lambda \over \pi} \int_0^t \iint_{S^1} \Bigl(v_s(y) + 4\lambda^3 \bar Y_{s}^\5(y)+\lambda^3 \bar Y_{s}^\4(y)\Bigr)\,dY_{s}^\1(y)\,ds\;.
\end{equs}
(The latter expression is nothing but the constant mode of the right hand side of \eref{e:defMstart} before that expression was differentiated.)

At this stage of our construction, it seems that the choice \eref{e:defKprime} for $\CK'$ is somewhat
arbitrary. Intuitively, it should be the right choice though, since this is precisely the factor that appears
in the second term of \eref{e:equationv}, so that one does expect it to describe the amplitude of the small-scale
fluctuations of the solution. Mathematically, the fact that this is indeed the correct choice is seen by the fact that this
is the \textit{only} choice guaranteeing that the image of $\hat \CM$ lies again in $\CB_{\star,T}$, so that we
can set up a fixed point argument. 
This is the content of the following result which,
together with the convergence results already mentioned earlier in this section, forms the core of this article.
The space $\CB_{\Psi,T}$ appearing in the statement is defined as in Proposition~\ref{prop:mainCont}.

\begin{theorem}\label{theo:mainTechnical}
For every $\kappa < {1\over 12}$, every $T>0$, and every $\beta > 2\kappa$, 
the map $\hat \CM$ extends uniquely to 
a locally uniformly continuous map from $\CC^\beta \times \CY_{\star,T}$ to $\CB_{\star,T}$.

Furthermore, for every $\Psi \in \CX$ and every $h_0 \in \CC^{\beta}$, there exists $T>0$ depending 
only on the norms of $\Psi$ and $h_0$ such that
the map $V \mapsto \hat \CM(h_0,\Psi,V)$ is a strict contraction in a sufficiently small ball of $\CB_{\Psi,T}$.
Furthermore, the equation $V = \hat \CM(h_0,\Psi,V)$ admits a unique solution in all of $\CB_{\Psi,T}$.
\end{theorem}

\begin{proof}
First, note that, since we defined $\CR^\CM$ such that \eref{e:remainder} holds, we ensure that, 
at least for smooth data, $\hat \CM(h_0,\Psi,V) \in \CB_{\Psi,T}$
for every $(h_0,\Psi,V) \in \CC^\beta\times \CY_{\star,T}$.
The local uniform continuity of $\hat \CM$ is the hard part of this result, and this is
obtained in Proposition~\ref{prop:contraction} below.

The contraction properties and the existence of a unique fixed point for $\hat \CM$ with its first two arguments fixed 
then follows from Theorem~\ref{theo:main}, noting that its
assumptions are satisfied for \textit{every} $\Psi$ by the definition of the space $\CX$ in which the input $\Psi$ lies.
\end{proof}

At this point, it is legitimate to question whether such a complicated nonlinear
construction is really necessary, or whether one could instead find fixed Banach
spaces $\hat \CB_T$ and $\hat \CX$ such that 
$\hat \CM$ extends to a continuous map $\hat \CX \times \hat \CB_T \to \hat \CB_T$ and
has a fixed point for small enough time horizon $T$. 

While it doesn't seem easy to disprove such a statement at this level of generality, the results in 
\cite{TerryPath} strongly suggest that it is not possible to find any such spaces.
Indeed, the following is a straightforward extension of \cite{TerryPath}:

\begin{theorem}\label{theo:negativeBM}
There exists no separable Banach space $\CB$ supporting 
Wiener measure and such that the bilinear functional
\begin{equ}
\CI\colon (u,v) \mapsto \int_0^1 u(t)\,dv(t)\;,
\end{equ}
defined on $\CH = H^1([0,1])$, extends to a continuous function on $\CB\times \CB$.
\end{theorem}

\begin{proof}
Note first that we can assume without loss of generality that $\CB \subset \CC([0,1])$ since 
larger spaces make it only harder for $\CI$ to be continuous.
Also, by assumption, $\CB$ is the completion of $\CH$ under some norm $\|\cdot\|_\CB$. 
Assuming by contradiction that $\CI$ is continuous on $\CB \times \CB$, it follows from Fernique's theorem
that $\int \CI(u,v)\,\mu(du,dv) < \infty$, for every measure $\mu$ on $\CB \times \CB$ such that both of its marginals
are given by Wiener measure.

Let $\Pi_N \colon \CB \to \CH$ be the projection onto the first $N$ Fourier modes which, since
$\CB \subset \CC([0,1])$, is a bounded operator for every $N$.
The construction in \cite{TerryPath} then yields  a measure $\mu$ as above with the property that
\begin{equ}[e:boundLog]
\int \CI(\Pi_N u, \Pi_N v)\,\mu(du,dv) \sim \log N\;,
\end{equ}
for $N$ large. Since Fourier modes form an orthonormal basis of $\CH$, it follows from \cite[Thm~3.5.1]{Bogachev} 
that $(\Pi_N u, \Pi_N v) \to (u,v)$
$\mu$-almost surely as $N \to \infty$. Since weak convergence implies tightness in separable 
Banach spaces, we conclude from Fernique's theorem that 
\begin{equ}
\sup_N \int \|\Pi_N u\|_\CB \|\Pi_N v\|_\CB\,\mu(du,dv) < \infty\;,
\end{equ}
which is a contradiction to \eref{e:boundLog}. 
\end{proof}

Since, for any fixed $t$, both $Y^\1_t$ and $v_t$ in \eref{e:defMstart} have regularity properties
identical to those of Brownian motion (actually, both $Y^\1_t$ and $\Phi_t$ \textit{are} nothing but centred Brownian bridges),
Theorem~\ref{theo:negativeBM} leaves no doubt that the classical approach to making sense of \eref{e:equationv} 
in the limit $\eps \to 0$ is doomed to failure.

We are now able to provide a proof of the results stated in the introduction:

\begin{proof}[of Proposition~\ref{prop:mainCont}]
The spaces $\CB_{\star,T}$, $\CX$ and $\CY_{\star,T}$ as well as the maps $\hat \CM$ and $h^\star = \sum_{\tau \in \bar \CT} Y^\tau$  were already defined, so that it suffices to verify that they satisfy the required properties.

Given an initial condition $h_0 \in \CC^\beta$, we set $v_0 = \d_x \bigl(h_0 - h^\star_0\bigr)$, so that
$v_0 \in \CC^{\beta-1}$. We know from Theorem~\ref{theo:mainTechnical} 
that one can choose $T>0$ depending only on  
$\|v_0\|_{\beta-1}$ and $\|\Psi\|_{\CW}$ such that the map $\hat \CM$ is a contraction in its last argument
and we denote its fixed point by $\hat \CS_{\mathrm{R}}^T(h_0, \Psi) \in \CY_{\star,T}$. By performing the  same continuation
procedure as in the proof of the existence of a unique maximal solution for ordinary differential equations, 
we obtain an explosion time $T_\star(h_0, \Psi)$, which is the supremum over all times $T$ 
such that the fixed point problem in $\CB_{\star,T}$ has a solution. The fact that all H\"older norms of the solution explode
as $t \to T_\star$ is an immediate consequence of the fact that the local existence time can be controlled in terms 
of the H\"older norm of the initial condition.
Furthermore, these solutions are all unique by the same
argument as in the proof of Theorem~\ref{theo:main} below, and they agree on their common domains of definition.

The third property, namely the continuity of $\hat \CS_{\mathrm{R}}^T$ in a neighbourhood of $(h_0, \Psi_0)$ whenever
$T < T_\star(h_0,\Psi_0)$ also follows in the same way as in the classical theory of ODEs. This then immediately implies the lower semicontinuity of $T_\star$, since its definition implies that one has $T_\star(h,\Psi) > T$ for every $(h,\Psi)$ in
such a neighbourhood. 
Finally, if we define $\Theta_t$ to be the canonical time-shift on $\CX$ (which is a continuous map for 
every $t \in \R$), then the cocycle property follows immediately from the elementary properties of the integral
and the heat semigroup.
\end{proof}

\begin{proof}[of Theorem~\ref{theo:mainCont}]
We now define the map $\CS_{\mathrm{R}}$ by setting
\begin{equ}
\CS_{\mathrm{R}}(h_0, \Psi)_t = h_\star(\Psi)_t + \bigl(\pi \hat \CS_{\mathrm{R}}(h_0, \Psi)\bigr)_t\;,
\end{equ}
for $t < T_\star(h_0, \Psi)$, and $\CS_{\mathrm{R}}(h_0, \Psi)_t = \infty$ for $t > T_\star(h_0, \Psi)$.
Since one necessarily has $\lim_{t \to T_\star} \|\CS_{\mathrm{R}}(h_0, \Psi)_t\|_{{1\over 2}-\beta} = +\infty$,
the definition of the topology on $\bar \CC^{{1\over 2}-\beta}$ implies that the map $\CS_{\mathrm{R}}$ constructed
in this way does indeed take values in $\CC(\R_+, \bar \CC^{{1\over 2}-\beta})$.

If we furthermore denote by $\CS_{\mathrm{R}}^T$ the restriction of $\CS_{\mathrm{R}}$ to the interval $[0,T]$,
then it follows from Proposition~\ref{prop:mainCont} that $\CS_{\mathrm{R}}^T$ is continuous on the set $\{(h,\Psi)\,:\, T_\star(h,\Psi) > T\}$. In particular, this is stronger than the claimed continuity property.

It remains to show that, for every fixed initial condition $h_0$, one has 
$\CS_{\mathrm{CH}}(h_0,\omega) = \CS_{\mathrm{R}}(h_0,\Psi(\omega))$ almost surely. 
Fix $T>0$, and let $S_T(h_0) = \{\Psi \in \CX\,:\, T_\star(h_0,\Psi) \le T\}$, which is the set of 
possible discontinuities of $\CS_{\mathrm{R}}(h_0,\cdot)$. By construction, for every $\eps > 0$, 
$\CS_{\mathrm{R}}^T(h_0,\Psi_\eps(\omega))$ almost surely agrees with the solution $h_\eps$ to
\eref{e:KPZeps} up to time $T$. Since we know on the one hand that $\Psi_\eps \to \Psi$ in probability,
and on the other hand that $h_\eps$ converges in probability to $\CS_{\mathrm{CH}}(h_0,\omega)$,
the stated claim follows if we can show that $\P(\Psi(\omega) \in S_T(h_0)) = 0$ for every $T>0$ and every 
initial condition $h_0 \in \CC^{\beta}$.

Assume by contradiction that there exists $h_0$ and $\kappa > 0$ such that 
$\P(\Psi(\omega) \in S_T(h_0)) \ge \kappa$.
It follows from our construction that, for every
$\Psi_0 \in S_T(h_0)$ and every $K>0$, there exists a neighbourhood $V$ of $\Psi_0$ in $\CX$ such that
\begin{equ}
\sup_{t \le T} \|\CS_{\mathrm{R}}(h_0,\Psi)\|_{\beta} \ge K\;,\qquad \forall \Psi \in V\;.
\end{equ}
Since $\Psi_\eps \to \Psi$ in probability, we conclude that there exists $\eps_0 > 0$ such that 
\begin{equ}
\P \bigl(\sup_{t \le T} \|h_{t,\eps}\|_{\beta} \ge K\bigr) \ge {\kappa \over 2}\;,
\end{equ}
uniformly over all $\eps < \eps_0$. This on the other hand is ruled out by the fact that $h_\eps \to h$ in
probability in $\CC([0,T], \CC^\beta)$.
\end{proof}

To conclude this section, we give an explicit interpretation of the solution map $\CS_{\mathrm{R}}$ for
arbitrary smooth data $\Psi$ and we use the continuity of the solution map to provide a novel homogenisation 
result.

\subsection{Smooth solutions}
\label{sec:smooth}

It is instructive to see what is the meaning of $\CS_{\mathrm{R}}(h_0,\Psi)$ for general smooth data $\Psi = (\{Y^\tau\}_{\tau \in \bar \CT}, \YY, R^\5) \in \CX$.
Given such smooth data, we set
\begin{equ}
\xi(x,t) \eqdef \d_t Y^\1_t(x) - \d_x^2 Y^\1_t(x)\;,
\end{equ}
as well as
\begin{equ}[e:defH]
H_t(x) = \sum_{\tau=[\tau_1,\tau_2] \in \bar \CT}\lambda^{|\tau|} \Bigl(\d_t Y^\tau_t(x) -\d_x^2 Y^\tau_t(x)- \bar Y^{\tau_1}_t(x) \,\bar Y^{\tau_2}_t(x) \Bigr)\;,
\end{equ}
which is some kind of ``defect'' by which the $Y^\tau$ may fail to satisfy their constituent equations.
Defining $\Phi$ as in \eref{e:realdefPhi} (but with $Y^\1_\eps$ replaced by $Y^\1$), we also define
$G$ to be the smooth function such that
\begin{equ}
\YY_t(x,y) = \int_x^y \Phi_t(z)\,dY_t^\1(z) + \int_x^y G_t(z)\,dz\;.
\end{equ}
Such a function always exists since $\YY$ satisfies \eref{e:relYY} by definition of $\CX$ and since any two
functions satisfying these relations always differ by an increment of a function of one variable.

With this notation, we then have the following result:

\begin{theorem}\label{theo:smooth}
Let $\Psi \in \CX$ be a smooth element, let $h_0 \in \CC^\infty$, and let $H$, $G$ and $\xi$ be as above. 
Then, $T_\star(h_0,\Psi) = +\infty$ and $\CS_{\mathrm{R}}(h_0,\Psi)$ is the unique global solution to
\begin{equ}[e:correctH]
\d_t h_t = \d_x^2 h_t + \lambda \bigl(\d_x h_t\bigr)^2 + 4\lambda^2 G_t\, \d_x \bigl(h_t - J_t(\Psi)\bigr) + H + \xi\;,
\end{equ}
where $J_t(\Psi)$ is the function given by
\begin{equs}
J_t(\Psi) = Y_t^\1 + \lambda Y_t^\2 - \lambda^2 Y_t^\3\;,
\end{equs}
with initial condition $h_0$.
\end{theorem}

\begin{proof}
By construction, we have
\begin{equ}
h = u + \sum_{\tau\in \bar \CT} \lambda^{|\tau|} Y^\tau\;,
\end{equ} 
where $u$ solves the fixed point equation
\begin{equs}
u &= \bigl(P_t u_0\bigr)(x) + 2\lambda \int_0^t \iint_{S^1} p_{t-s}(x-y) \bigl(\d_x u_s(y) + 4 \lambda^3\bar Y_{s}^\5(y)\bigr)\,dY_{t}^\1(y)\,ds \\
&\qquad + \int_0^t P_{t-s} \bigl(\lambda^4 \bar Y_{s}^\4(y) \, \bar Y_{s}^\1(y)  + F(u,\Psi,s)\bigr)\,ds\;, \label{e:equationfixu}
\end{equs} 
where $u_0 = h_0 - \sum_{\tau\in \bar \CT} \lambda^{|\tau|} Y^\tau_0$ and $F$ is as in \eref{e:defFs}. It then 
follows from \eref{e:corint} below and the fact that, by construction, $\d_x u_s + 4 \lambda^3 \bar Y_{s}^\5$ is 
a rough path controlled by $\Phi$ with derivative process $2\lambda\bigl(\d_x u_s + 4\lambda^3\bar Y_s^\5 + \lambda^3\bar Y_s^\4 + 2\lambda^2\bar Y_s^\3\bigr)$, that one has the identity
\begin{equs}
\iint_{S^1} &p_{t-s}(x-y) \bigl(\d_x u_s(y) + 4\lambda^3 \bar Y_{s}^\5(y)\bigr)\,dY_{s}^\1(y)\\
&= \int_{S^1} p_{t-s}(x-y) \bigl(\d_x u_s(y) + 4 \lambda^3\bar Y_{s}^\5(y)\bigr)\,\d_xY_{s}^\1(y)\,dy\\
&\quad + 2\lambda \int_{S^1} p_{t-s}(x-y) \bigl(\d_x u_s + 4\lambda^3\bar Y_s^\5 + \lambda^3\bar Y_s^\4 + 2\lambda^2\bar Y_s^\3\bigr)\,G_{s}(y)\,dy\\
&= \int_{S^1} p_{t-s}(x-y) \bigl(\d_x u_s(y) + 4 \lambda^3\bar Y_{s}^\5(y)\bigr)\,\d_xY_{t}^\1(y)\,dy\\
&\quad + 2\lambda \int_{S^1} p_{t-s}(x-y) \d_x\bigl(h_s - J_s(\Psi)\bigr)\,G_{s}(y)\,dy\;.
\end{equs}
Similarly, it follows from \eref{e:defH} that
\begin{equs}
\sum_{\tau\in \bar \CT} \lambda^{|\tau|} Y^\tau_t &= \sum_{\tau\in \bar \CT} \lambda^{|\tau|} P_t Y^\tau_0 + \int_0^t P_{t-s} \Bigl(\Bigl(\sum_{\tau\in \bar \CT} \lambda^{|\tau|} Y^\tau_s\Bigr)^2- \bar \CR^\star(\Psi)_s \Bigr)\,ds \\
&\quad +  \int_0^t P_{t-s}H_s\,ds\;.
\end{equs}
We can now ``undo'' the construction and recover a fixed point equation for $h$. Since the fixed point
map for $u$ was built precisely in such a way that $h$ solves the KPZ equation, provided that the $Y^\tau$ solve
their constituent equations and that the rough integral is replaced by a usual Riemann integral, we recover
the KPZ equation, except for the two correction terms involving $H$ and $G$, thus yielding \eref{e:correctH}.
\end{proof}

\subsection{A new homogenisation result}
\label{sec:homogen}

To conclude this section, we present a new periodic homogenisation result for the 
heat equation with a strong time-varying potential, which illustrates the power of the techniques presented in this article. 
This equation has been studied extensively recently and several homogenisation results have
been obtained for both the stochastic and the deterministic case \cite{MR2718269,Guillaume,Etienne}, 
see also the monograph \cite{MR1185878}.

In this section, we show how to obtain a periodic homogenisation result in the situation where, in \eref{e:KPZ},
the driving noise $\xi$ is replaced by a space-time periodic function that is rescaled with the same exponents 
as space-time
white noise. More precisely, we fix a periodic function $\phi\colon S^1 \to \R$ with $\int \phi(x)\,dx = 0$
and we consider the equation
\begin{equ}[e:homogen]
\d_t h^{(n)} = \d_x^2 h^{(n)} + (\d_x h^{(n)})^2 + n^{3/2} \phi(nx + cn^2 t) - C_n\;,
\end{equ}
for $n$ large, where $C_n$ a sequence of constants to be determined so that the solutions to \eref{e:homogen}
converge to a non-trivial limit. Of course, as in \eref{e:linear}, this is equivalent to solving the heat equation
with the potential $n^{3/2} \phi(nx + cn^2 t)$.

Similarly to what we did before, we write $C_n = \sum_{\tau \in \bar \CT} C_n^\tau$ and we define
$Y_n^\tau$ as the stationary (modulo constant Fourier mode) solutions to
\begin{equ}[e:equationYn]
\d_t Y_n^\tau = \d_x^2 Y_n^\tau + \d_x Y_n^{\tau_1}\,\d_x Y_n^{\tau_2} - C_n^\tau\;,
\end{equ}
where we want to specify the constants $C_n^\tau$ in such a way that the resulting expressions all converge to
finite limits. It turns out to be straightforward to solve these equations in the following way. Set $\gamma_\1 = {1\over 2}$ and
then define recursively a family of exponents $\gamma_\tau$ by
\begin{equ}
\gamma_{[\tau_1,\tau_2]} = \gamma_{\tau_1} + \gamma_{\tau_2}\;.
\end{equ}
With this notation, we then make the ansatz
\begin{equ}[e:ansatzY]
Y_n^\tau(t,x) =  n^{-\gamma_{\tau}} \phi^\tau(nx + cn^2 t) + n^{2-\gamma_\tau} K^\tau\, t - C_n^\tau\,t\;,
\end{equ}
for some periodic centred functions $\phi^\tau$ and constants $K^\tau$. We furthermore 
introduce the operator $G = (c - \d_x)^{-1}$, where $c$ is as in \eref{e:homogen}.
With this ansatz, we then immediately obtain the identity 
\begin{equ}
\phi^\1 = G\,\d_x \phi\;.
\end{equ}
Further inserting \eref{e:ansatzY} into \eref{e:equationYn}, we obtain for the remaining functions $\phi^\tau$
and constants $K^\tau$ the recursion relations
\begin{equ}
\d_x \phi^{[\tau_1,\tau_2]} = G\,\Pi_0^\perp \bigl(\d_x \phi^{\tau_1}\,\d_x \phi^{\tau_1}\bigr)\;,\qquad
K^{[\tau_1,\tau_2]} = \Pi_0 \bigl(\d_x \phi^{\tau_1}\,\d_x \phi^{\tau_1}\bigr)\;.
\end{equ}
It is now very easy to apply the results exposed in this section to obtain the following homogenisation
result:

\begin{theorem}
With the same notations as above, set $C_n = n K^{\2} + 2n^{1/2} K^\3$. Then, for every H\"older continuous 
initial condition $h_0$, the solution to \eref{e:homogen} converges locally uniformly as $n \to \infty$ to the solution
$h$ to
\begin{equ}
\d_t h = \d_x^2 h + (\d_x h)^2 + 4\bar K \d_x h + K^\4 + 4K^\5\;,
\end{equ}
where the constant $\bar K$ is given by $\bar K = \Pi_0 \bigl(\d_x \phi^\1\, G\d_x \phi^\1\bigr)$.
If furthermore $\phi$ is non-constant, then $\bar K \neq 0$ if and only if $c \neq 0$.
\end{theorem}

\begin{proof}
The claim follows immediately from Theorem~\ref{theo:smooth}, as well as the continuity of $\CS_{\mathrm{R}}$
established in Proposition~\ref{prop:mainCont}, provided that we can show that
\begin{equ}
(\{Y_n^\tau\}_{\tau\in \bar \CT}, \YY_n, R^\5_n) \to (\{Y^\tau\}_{\tau\in \bar \CT}, \YY, R^\5)\;,
\end{equ}
in $\CX$, where $Y^\1 = Y^\2 = Y^\3 = 0$, $Y^\4 = K^\4$, $Y^\5 = K^\5$, $\YY_t(x,y) = \bar K(y-x)$, and $R^\5 = 0$.
By choosing $C_n^\2 = n K^\2$ and $C_n^\3 = n^{1/2} K^\3$, the convergence of the processes $Y^\tau$ to the correct constants follows immediately from \eref{e:ansatzY}. Note that the scaling is precisely such that the convergence does indeed take place
in $\CX_\tau$ for each of the $Y^\tau_n$, but it would \textit{not} take place in any stronger H\"older-type norm.
The reason why $K^\3$ appears with a prefactor $2$ in the statement of the theorem is that there are
two trees isometric to $\smalltree$ in $\CT$.

It remains to consider $\YY_n$ and $R^\5_n$, which are both related to the process $\Phi_n$ given
as in \eref{e:realdefPhi}. A straightforward calculation shows that
\begin{equ}
\Phi_n(t) = n^{-1/2} \bar \phi(nx + cn^2 t)\;,\qquad \bar \phi = G\,\d_x \phi^\1\;.
\end{equ}
Let now $\psi = \bar \phi\, \d_x \phi^\1$ so that, at time $t=0$, one has
\begin{equs}
\YY_n(x,y) &= \int_x^y \psi(nz) \,dz - {\bar \phi(nx) \over n}\bigl(\phi^\1(ny)-\phi^\1(nx)\bigr)\\
&= \bar K (y-x) + \CO(|y-x| \wedge n^{-1})\;.
\end{equs}
Since the situation at time $t \neq 0$ is the same, modulo a spatial translation, it shows that
$\YY_n$ does indeed converge to $\YY$ in $\CC(\R, \CC_2^{3/4})$.
A similar calculation shows that $R^\5_n(x,y) = \CO(|y-x| \wedge n^{-1})$, so that it does indeed converge to $0$
in the same space.

For the last statement, an explicit calculation yields the identity
\begin{equ}
\bar K = \sum_{k \in \Z} {ck^2 \over (c^2 + k^2)^2}|\phi_k|^2\;,
\end{equ}
from which the claim follows at once.
\end{proof}

\section{Elements of rough path theory}
\label{sec:roughpaths}

In this section, we give a very short introduction to some of the elements of rough path theory needed for this work.
For more details, see the original article \cite{Lyons} and the monographs \cite{MR2036784,MR2314753,PeterBook} or, for a simplified exposition covering most of the notions required for this
work, see \cite{BurgersRough}. 
We will mostly make use of the notations and terminology
introduced by Gubinelli in \cite{Max} since the estimates given in that work 
seem to be the ones that are most 
suitable for the present undertaking.

We denote by $\CC_2(S^1, \R^n)$ the space of continuous functions
from $\R^2$ into $\R^n$ that vanish on the diagonal and such that, for $f \in \CC_2(S^1,\R^n)$,
there exists $c \in \R^n$ such that the relations
\begin{equ}
f(x+2\pi, y+2\pi) = f(x,y)\;,\qquad f(x,y+2\pi) = f(x,y) + c\;,
\end{equ}
hold for every $x,y \in \R^n$. We will often make an abuse of notation and write $f(x,y)$ for
$x, y \in S^1$. Our convention in this case is that we take for $x$ the unique representative in 
$[0,2\pi)$ for $y$ the unique representative in $[x,x+2\pi)$. The same convention is enforced whenever 
we write $\int_x^y$ for $x,y \in S^1$.

Usually, we will omit the base space $S^1$ and 
the target space $\R^n$ in our notations  for the
sake of simplicity. We also define a difference operator $\delta\colon \CC \to \CC_2$ by
\begin{equ}
\delta X(x,y) = X(y) - X(x)\;.
\end{equ}

A \textit{rough path} on $S^1$ then consists of two parts: a continuous function
$X \in \CC(S^1,\R^n)$, as well as a continuous ``area process'' $\XX \in \CC_2(S^1, \R^{n \times n})$ such that
the algebraic relations
\begin{equ}[e:constr]
\XX^{ij}(x,z) - \XX^{ij}(x,y) -\XX^{ij}(y,z) =  \delta X^i(x,y)\delta X^j(y,z)\;,
\end{equ}
hold for every triple of points $(x,y,z)$ and every pair of indices $(i,j)$.
One should think of $\XX$ as postulating the value of the quantity
\begin{equ}[e:intX]
\int_x^y \delta X^i(x,z)\,dX^j(z) \eqdef \XX^{ij}(x,y) \;,
\end{equ}
where we take the right hand side as a \textit{definition} for the left hand side. (And not the other way around!)
The aim of imposing \eref{e:constr} is to ensure that \eref{e:intX} does indeed behave
like an integral when considering it over two adjacent intervals.

\begin{remark}
We see from \eref{e:intX} why $\XX$ can not in general be a continuous function on $S^1\times S^1$,
since there is no a priori reason to impose that $\int_{S^1} \delta X^i(x,z)\,dX^j(z) = 0$.
\end{remark}

Another important notion taken from \cite{Max} is that of a path $Y$ \textit{controlled} by a rough path $X$.
Given a rough path $X$, we say that a \textit{pair} of functions $(Y,Y')$ is a rough path \textit{controlled by $X$} 
if  the ``remainder term'' $R$ given by
\begin{equ}[e:defYR]
R(x,y) \eqdef \delta Y(x,y) - Y'(x)\, \delta X(x,y)\;,
\end{equ}
has better regularity properties than $Y$. Typically, we will assume that $\|Y\|_\alpha < \infty$ and $\|Y'\|_\alpha < \infty$
for some H\"older exponent $\alpha$, but that $\|R\|_\beta < \infty$ for some $\beta > \alpha$.
Here, $R_{s,t} \in \R^m$ and the second term is  a matrix-vector multiplication. 

Note that, a priori, there could be many distinct ``derivative processes'' $Y'$ associated to a given path $Y$. 
However, if $X$ is a typical sample path of Brownian motion and if we impose the bound $\|R\|_\beta < \infty$
for some $\beta > {1\over 2}$, then it was shown in \cite{NateshRough} that there can be at most one derivative process $Y'$ 
associated to every $Y$.

\subsection{Integration of controlled rough paths.}

It turns out that if $(X,\XX)$ is a rough path taking values in $\R^n$ and $Y$ is a path controlled by $X$ that also takes values in $\R^n$, 
then one can give a natural 
meaning to the expression $\int \scal{Y_t , dX_t}$, provided that $X$ and $Y$ are sufficiently regular. 
The approximation $Y_t \approx Y_s + Y'_s \,\delta X_{s,t}$ suggested  by
\eref{e:defYR} shows  that it is reasonable to define the integral as the following limit of ``second-order Riemann sums'':
\begin{equ}[e:defintfinal]
\iint \scal{Y(x) , dX(x)} = \lim_{|\CP| \to 0} \sum_{[x,y] \in \CP} \bigl(\scal{Y(x), \delta X(x,y)} + \tr Y'(x) \,\XX(x,y)\bigr)\;,
\end{equ}
where $\CP$ denotes a partition of the integration interval, and $|\CP|$ is the length of its longest element.

With these notations at hand, we quote the following result, which is a slight reformulation of \cite[Prop~1]{Max}:

\begin{theorem}\label{theo:integral}
Let $(X,\XX)$ satisfy \eref{e:constr} and let $(Y,Y')$ be a rough path controlled by $X$ with a remainder $R$
given by \eref{e:defYR}.
Assume furthermore that
\begin{equ}[e:boundHolderX]
\|X\|_\alpha + \|\XX\|_\beta + \|Y'\|_{\bar \alpha} + \|R\|_{\bar \beta} < \infty\;,
\end{equ}
for some exponents $\alpha, \bar \alpha, \beta, \bar \beta > 0$. Then, provided that 
$\alpha + \bar \beta > 1$ and $\bar \alpha +  \beta > 1$, the compensated 
Riemann sum in \eref{e:defintfinal} converges. Furthermore, one has the bound
\begin{equ}[e:boundInt]
\Bigl|\iint_x^y \scal{\delta Y(x,z), dX(z)} - \tr Y'(x)\, \XX(x,y)\Bigr| \lesssim |y-x|^{\gamma}
\bigl(\|X\|_\alpha \|R\|_{\bar \beta} + \|\XX\|_{\beta}\|Y'\|_{\bar \alpha}\bigr)
\end{equ}
with $\gamma = (\alpha + \bar \beta)\wedge (\bar \alpha +  \beta)$, 
for some proportionality constant depending only on the dimensions of the quantities involved
and the values of the exponents.
\end{theorem}

Actually, one has an even stronger statement. Let $\CC^{\alpha,\beta} = \CC^\alpha \oplus \CC_2^\beta$ be the
space of integrators $(X,\XX)$ and let $\CY$ be the closed subset of $\CC^{\alpha,\beta} \oplus \CC^{\bar \alpha,\bar \beta}$
(with elements of $\CY$ written as $(X,\XX,Y',R)$) 
defined by the algebraic relations \eref{e:constr} and \eref{e:defYR}, where \eref{e:defYR} is interpreted as stating that
there exists a function of \textit{one} variable $Y$ such that \eref{e:defYR} holds for all pairs $(x,y)$.
Let furthermore $\CY_g \subset \CY$ be the set defined by the additional constraint
\begin{equ}[e:geometric]
\XX^{ij}(x,y) + \XX^{ji}(x,y) = \delta X^i(x,y)\,\delta X^j(x,y)\;.
\end{equ}
(Note that this constraint is automatically satisfied if $\XX$ is given by the left hand side of \eref{e:intX} for some smooth $X$.)
Then, one has:

\begin{proposition}\label{prop:continuous}
The set $\CY_g$ is dense in $\CY$.
Furthermore, provided that $\bar \alpha \le \alpha$, $\bar \beta \le \beta$, and $(\alpha + \bar \beta)\wedge (\bar \alpha +  \beta) > 1$,
the map 
\begin{equ}[e:integrationMap]
\CI \colon (X,\XX,Y, Y') \mapsto  \Bigl(X, \XX, \int_0^\cdot \scal{\delta Y(0,z), dX(z)}, \delta Y(0,\cdot)\Bigr)\;,
\end{equ}
defined on smooth elements of $\CY_g$, extends uniquely to the continuous map 
$\hat \CI\colon \CY_g \to \CY_g$ obtained by replacing the Riemann integral by $\iint$ in the above expression. 

Furthermore, $\hat \CI$ is
uniformly Lipschitz continuous on bounded sets of $\CY_g$ under the natural norm
\begin{equ}
\|X,\XX,Y', R^Y\| = \|Y'\|_{\bar\alpha} + \|R^Y\|_{\bar\beta} + \|X\|_\alpha + \|\XX\|_\beta\;,
\end{equ}
and it is given by replacing $\int$ by $\iint$ in \eref{e:integrationMap}.
\end{proposition}

\begin{proof}
The density of $\CY_g$ in $\CY$ was shown for example in \cite{Geometric}.
For the uniform Lipschitz continuity of $\hat \CI$ on bounded sets, it  
suffices to retrace the proof of \cite[Theorem~1]{Max}. Since $\CY_g$ is dense in $\CY$, the uniqueness
of the extension follows.

Note that this is \textit{not} a corollary of Theorem~\ref{theo:integral}. Indeed, the bound \eref{e:boundInt} only
holds on the nonlinear space $\CY$ so that it is not possible to simply exploit the bilinearity of the integral,
even though the bound obtained in \cite{Max} shows that it behaves ``as if'' the bound \eref{e:boundInt} was valid on
all of $\CC^{\alpha,\beta} \oplus \CC^{\bar \alpha,\bar \beta}$.
\end{proof}

\begin{remark}
We made a slight abuse of notation in \eref{e:integrationMap} in order to improve the legibility of the expressions,
by identifying on both sides of the equation elements $(X,\XX,Y',R)$ with the corresponding element 
$(X,\XX,Y,Y')$, where $Y$ is the
(unique up to constants) function such that \eref{e:defYR} holds. We also slightly jumbled the dimensions of the spaces
(if $X$ is $n$-dimensional then $Y$ should also be so, but the integral is only one-dimensional), but the meaning
should be obvious.
\end{remark}

\begin{remark}
The bound \eref{e:boundInt} does behave in a very natural way under dilatations.
Indeed, the integral is invariant under the transformation 
\begin{equ}[e:dilat]
(Y,X,\XX) \mapsto (\lambda^{-1}Y, \lambda X, \lambda^2 \XX)\;.
\end{equ}
The same is true for right hand side of \eref{e:boundInt}, since under this dilatation, we also have $(Y', R) \mapsto (\lambda^{-2} Y', \lambda^{-1}R)$. 
\end{remark}

\begin{remark}
It is straightforward to check that, if $(Y,Y')$ is a rough path controlled by $X$, then so is $(fY, fY')$, for any 
smooth function $f$. As a consequence, if $(X,\XX)$ and $(Y,Y')$ satisfy the bounds \eref{e:boundHolderX},
then Theorem~\ref{theo:integral} allows to make sense of the product $Y(x) {dX\over dx}$ as a distribution,
even in situations when $\alpha < {1\over 2}$, where such a product would not be well-defined in the classical sense.
\end{remark}

\begin{remark}
One could argue that it would have been natural to impose the condition \eref{e:geometric} from the beginning.
The reasons for not doing so are that the integral $\iint$ is well-defined without it and that non-geometric
situations can arise naturally in the context of numerical approximations, see for example \cite{Jan,JanHendrik}.
\end{remark}

It is clear from the definition \eref{e:defintfinal} that if $X$ is smooth and $\XX$ is given by \eref{e:intX}
(reading the definition from right to left), then $\iint$ coincides with the usual Riemann integral.
It is therefore instructive to see what happens if $X$ is a smooth function but one sets
\begin{equ}
\XX^{ij}(x,y) = \int_x^y \delta X^i(x,z)\,dX^j(z) + \int_x^y F^{ij}(z)\,dz \;,
\end{equ}
for some continuous function $F$. 
It is then clear that \eref{e:constr} is still satisfied and that $\|\XX\|_{\beta} < \infty$
provided that $\beta \le 1$. Even the additional ``geometric'' constraint \eref{e:geometric} is satisfied if $F$ is
antisymmetric. 
Given a smooth function $Y$, we can then choose for $Y'$ an arbitrary smooth function
and the remainder term $R$ given by \eref{e:defYR} will still satisfy $\|R\|_{\bar\beta} < \infty$ for $\bar\beta \le 1$.
It is now straightforward to verify that the rough integral is well-posed and
equals
\begin{equ}[e:corint]
\iint_x^y \scal{Y(z)\,dX(z)} = \int_x^y \scal[B]{Y(z),{dX \over dz}(z)}\, dz + \int_x^y \tr Y'(z) F(z)\,dz\;,
\end{equ}
where the right hand side is a usual Riemann integral. See for example the original article 
\cite[Example~1.1.1]{Lyons} for a more detailed explanation on
how to interpret this apparent discrepancy.

\subsection{Heat kernel bounds}

In this section, we obtain a number of sharp bounds on the interplay between the heat kernel on $S^1$ and
rough path valued functions. The reader who is interested in getting quickly to the heart of the matter can easily
skip the proofs of these results, since they are not particularly formative and mostly consist of relatively
straightforward estimates. However, Proposition~\ref{prop:intScale} is one of the most important ingredients
of the next section, so we prefer not to relegate these bounds to a mere appendix.
Several of these bounds are close in spirit to those obtained in \cite{Hendrik,BurgersRough}, but both the norms employed here 
and the precise form of the bounds required for our arguments are quite different.

The following quantity will be very often used in the sequel, so we give it a name. Given
a path $Y$ controlled by a rough path $(X,\XX)$ and given $\kappa > 0$, we define the quantity
\begin{equ}
\KK^\kappa(Y,X) \eqdef \|Y\|_{{1\over 2}-\kappa} \|X\|_{{1\over 2}-\kappa}  + \|\XX\|_{1-2\kappa}\|Y'\|_{\CC^{3\kappa}} + \|R^Y\|_{{1\over 2}+2\kappa}\|X\|_{{1\over 2}-\kappa}\;.
\end{equ}
We also define, for $\CC^1$ functions $f\colon \R \to \R$, the norm
\begin{equ}[e:normf]
\$f\$ := \sum_{n \in \Z} \sqrt{1+|n|}\sup_{0 \le t \le 1} \bigl(|f(n+t)| + |f'(n+t)|\bigr) < \infty\;.
\end{equ}

We then have the following bound, which can be viewed as 
a refinement of \cite[Prop.~2.5]{BurgersRough}:

\begin{proposition}\label{prop:intScale}
Let $f \in \CC^1(\R,\R)$ be such that $\$f\$ < \infty$, let $\lambda \ge 1$, and let $\kappa \in (0,{1\over 2})$. Then, the bound
\begin{equs}
\Bigl|\iint_{S^1} f(\lambda x) Y(x)\,dX(x)\Bigr| &\lesssim \lambda^{\kappa - {1\over 2}} |Y(0)| \|X\|_{{1\over 2}-\kappa} + \lambda^{2\kappa-1} \KK^\kappa(Y,X)\;,
\end{equs}
holds uniformly for all $\lambda > 1$, with a proportionality  constant depending only on $\$f\$$.
\end{proposition}

\begin{remark}
One very important feature of this bound is that the first term on the right hand side only depends on $|Y(0)|$ and not on 
$\|Y\|_\infty$ as in \cite{BurgersRough}. This is achieved thanks to the control provided by the norm $\$\cdot\$$, which ensures that
$f$ decays sufficiently fast at infinity. One place where this plays a crucial role is the proof of Corollary~\ref{cor:heatDiff} below.
\end{remark}

\begin{proof}
We use the same technique of proof as in \cite[Prop.~2.5]{BurgersRough}, but we are more careful with our bounds
and exploit the knowledge from \eref{e:normf} that $f$ decays relatively fast at infinity.
To shorten our notations, we set $Y_f(x) = f(\lambda x) Y(x)$ and $Y_f'(x) = f(\lambda x)Y'(x)$, and we also set
\begin{equ}
a_k = \sup_{0 \le t \le 1} \bigl(|f(k+t)| + |f'(k+t)|\bigr)\;.
\end{equ}
Setting $N = \lfloor 2\pi\lambda\rfloor$, $\delta x = 2\pi/N$ and writing $x_k = k\,\delta x$, we have
\begin{equ}
\Bigl|\iint_{S^1} f(\lambda x) Y(x)\,dX(x)\Bigr| \le \sum_{k=0}^{N-1}\Bigl|\iint_{x_k}^{x_{k+1}} Y_f(x)\,dX(x)\Bigr|
\eqdef \sum_{k=0}^{N-1} T_k\;.
\end{equ}
Note furthermore that, for $x \in [x_k, x_{k+1}]$, one has $\lambda x \in [k,k+2]$ so that, for every $\alpha \in (0,1]$,
one has the bounds
\minilab{e:bYf}
\begin{equs}
\|Y_f\|_{\alpha,k} &\lesssim (a_k+a_{k+1}) \bigl( \|Y\|_\alpha + \lambda^\alpha \|Y\|_\infty\bigr) \label{e:bYfa}\;,\\
\|Y_f'\|_{\alpha,k} &\lesssim (a_k+a_{k+1}) \bigl( \|Y'\|_\alpha + \lambda^\alpha \|Y'\|_\infty\bigr)\label{e:bYfb}\;, \\
\|R^{Y_f}\|_{\alpha,k} &\lesssim (a_k+a_{k+1}) \bigl( \|R^Y\|_{\alpha} + \lambda^{\alpha} \|Y\|_\infty\bigr)\;,\label{e:bYfc}
\end{equs}
where we denoted by $\|\cdot\|_{\alpha,k}$ the corresponding H\"older seminorm restricted to the interval $[x_k,x_{k+1}]$.

It then follows from Theorem~\ref{theo:integral} that
\begin{equ}[e:exprTk]
T_k = f(\lambda x_k) Y(x_k)\,\delta X_{k} + f(\lambda x_k) Y'(x_k)\,\XX_{k} + R_k\;,
\end{equ}
where the remainder term $R_k$ is bounded by
\begin{equ}[e:boundRk]
|R_k| \lesssim \lambda^{-\kappa-1} \bigl(\|Y_f'\|_{3\kappa,k} \|\XX\|_{1-2\kappa} + \|R^{Y_f}\|_{{1\over 2}+2\kappa,k}\|X\|_{{1\over 2}-\kappa}\bigr)\;.
\end{equ}
At this point, we note that the supremum norm of $Y$ over the interval $[x_k,x_{k+1}]$ is bounded by
\begin{equ}[e:boundYinf]
\|Y\|_{\infty,k} \lesssim |Y(0)| + \|Y\|_\alpha  \lambda^{-\alpha} (1+|k|)^\alpha\;,
\end{equ}
for any $\alpha \in (0,1]$. Using this identity with $\alpha = {1\over 2}-\kappa$ and 
the fact that $(1+|k|)^{1/2} a_k$ is summable by assumption, we can
combine \eref{e:boundRk} with \eref{e:bYf}, so that
\begin{equs}
\sum_{k=0}^{N-1} |R_k| &\lesssim \lambda^{\kappa - {1\over 2}} |Y(0)|\,\|X\|_{{1\over 2}-\kappa} + \lambda^{2\kappa - 1}\bigl( \|Y'\|_\infty \|\XX\|_{1-2\kappa}+ \|Y\|_{{1\over 2}-\kappa} \|X\|_{{1\over 2}-\kappa} \bigr)\\
&\quad + \lambda^{-\kappa-1} \bigl(\|\XX\|_{1-2\kappa}\|Y'\|_{3\kappa} + \|R^Y\|_{{1\over 2}+2\kappa}\|X\|_{{1\over 2}-\kappa}\bigr)\;,
\end{equs}
which is actually slightly better than the desired bound. In order to conclude, it remains to bound the other two terms appearing in
the right hand side of \eref{e:exprTk}. To do so, we use again \eref{e:boundYinf} to obtain
\begin{equs}
|f(\lambda x_k) &Y(x_k)\,\delta X_{k} + f(\lambda x_k) Y'(x_k)\,\XX_{k}| \lesssim (a_k+a_{k+1}) \lambda^{\kappa - {1\over 2}} |Y(0)|\,\|X\|_{{1\over 2}-\kappa} \\
& +  (a_k+a_{k+1}) \lambda^{2\kappa - 1} \bigl(|k|^{{1\over 2}-\kappa}\|Y\|_{{1\over 2}-\kappa} \|X\|_{{1\over 2}-\kappa} + 
\|Y'\|_\infty \|\XX\|_{1-2\kappa}\bigr)\;,
\end{equs}
and the claim follows at once.
\end{proof}

\begin{remark}
We think of $\kappa$ as being a small parameter. As a consequence, this bound is especially strong in the case
$Y(0)=0$ (or small), which will play a crucial role in the sequel.
\end{remark}

\begin{corollary}\label{cor:intKernel}
Let $p_t$ denote the heat kernel on $S^1$ and let $p_t^{(k)}$ be its $k$th (spatial)
derivative. Then, the bound 
\begin{equ}
\Bigl|\int_{S^1} p_t^{(k)}(x-y)\,dX(y)\Bigr| \lesssim t^{-{1\over 4} - {k + \kappa\over 2}}\|X\|_{{1\over 2}-\kappa}\;,
\end{equ}
holds uniformly over all $x$.
\end{corollary}

\begin{proof}
Setting $Y(x) = 1$, this is an immediate consequence of Proposition~\ref{prop:intScale},
using the fact that there exist functions $f_t$ such that, for every $k \ge 0$, 
$\$f_t^{(k)}\$$ is uniformly
bounded for $t \in (0,1]$, and such that 
\begin{equ}
p_t^{(k)}(x) = t^{-{1+k\over 2}} f_t^k(t^{-1/2} x)\;.
\end{equ}
The claim then follows by setting $\lambda = t^{-1/2}$.
\end{proof}

\begin{corollary}\label{cor:heatBasic}
Let $p_t$ denote the heat kernel on $S^1$ and let $p_t^{(k)}$ be its $k$th (spatial)
derivative. Then, the bound 
\begin{equs}
\Bigl|\iint_{S^1} p_t^{(k)}(z-y) Y(y)\,dX(y)\Bigr| \lesssim t^{-{1\over 4} - {k + \kappa\over 2}} |Y(z)| \|X\|_{{1\over 2}-\kappa} + t^{-{k\over 2}-\kappa} \KK^\kappa(Y,X)\;,
\end{equs}
holds uniformly over all $z$.
\end{corollary}

\begin{proof}
This follows from Proposition~\ref{prop:intScale} and the scaling properties of the heat kernel in the same way as Corollary~\ref{cor:intKernel}.
It furthermore suffices to translate the origin to $y = z$.
\end{proof}

\begin{corollary}\label{cor:heatDiff}
Let $p_t$ denote the heat kernel on $S^1$ and let $p_t^{(k)}$ be its $k$th (spatial)
derivative. Then, the bound 
\begin{equs}
\Bigl|\iint_{S^1} p_t^{(k)}(z-y) \bigl(Y(y) &- Y(x)\bigr)\,dX(y)\Bigr| \lesssim t^{-{1\over 4} - {k + \kappa\over 2}} |z-x|^{{1\over 2}-\kappa} \|Y\|_{{1\over 2}-\kappa}\|X\|_{{1\over 2}-\kappa} \\
&\quad + t^{-{k\over 2}-\kappa} \KK^\kappa(Y,X)\;,
\end{equs}
holds uniformly over all $x$ and $z$.
\end{corollary}

\begin{proof}
This is a particular case of Corollary~\ref{cor:heatBasic}, using the fact that $|Y(z) - Y(x)| \le |z-x|^{{1\over 2}-\kappa} \|Y\|_{\kappa-{1\over 2}}$.
\end{proof}

Actually, a similar bound also holds if we replace $p_t^{(k)}$ by a kind of ``fractional derivative'' as follows:

\begin{proposition}\label{prop:fracDer}
Let $p_t$ denote the heat kernel on $S^1$, let $p_t^{(k)}$ be its $k$th (spatial)
derivative, let $\kappa \in (0,{1\over 2})$, and let $\alpha \in [{1\over 2}-\kappa,1]$. 
Then, the bound 
\begin{equs}
\Bigl|\iint_{S^1} &{p_t^{(k)}(z-y) - p_t^{(k)}(z'-y) \over |z-z'|^\alpha} \bigl(Y(y) - Y(x)\bigr)\,dX(y)\Bigr| \\
&\qquad \lesssim t^{-\kappa - {k +\alpha\over 2}} \|Y\|_{{1\over 2}-\kappa}\|X\|_{{1\over 2}-\kappa}   + t^{-{\kappa}-{k+\alpha\over 2}} \KK^\kappa(Y,X)\;,\label{e:boundFrac} 
\end{equs}
holds uniformly over all $z$, $z'$ and $x$, such that $|x-z| \vee |x-z'| \le |z-z'|$, and over all $t \le 1$.
\end{proposition}

\begin{proof}
Denote the first term on the right hand side of \eref{e:boundFrac} by $T_1$ and the second term by $T_2$. As a shorthand, we also write
\begin{equ}
\CI \eqdef \iint_{S^1} \bigl(p_t^{(k)}(z-y) - p_t^{(k)}(z'-y)\bigr) \bigl(Y(y) - Y(x)\bigr)\,dX(y)\;,
\end{equ}
so that we aim to show that 
\begin{equ}[e:wantedCI]
|\CI| \lesssim |z-z'|^\alpha \bigl(T_1 +T_2\bigr)\;.
\end{equ}
With these notations, it follows immediately from Corollary~\ref{cor:heatDiff} that
\begin{equ}[e:bound1]
|\CI|  \lesssim t^{{\alpha +\kappa\over 2}-{1\over 4}} |z-z'|^{{1\over 2}-\kappa}T_1 + t^{\alpha \over 2} T_2\;.
\end{equ}
This shows that \eref{e:wantedCI} holds on the set $\{|t| \le |z-z'|^2\}$.
On the other hand, we can write 
\begin{equ}
\CI = \int_z^{z'} \iint_{S^1} p_t^{(k+1)}(z''-y) \bigl(Y(y) - Y(x)\bigr)\,dX(y)\,dz''
\end{equ}
Applying again Corollary~\ref{cor:heatDiff} (this time with $k+1$ instead of $k$) for the integrand and integrating over $z''$, we conclude that 
the bound
\begin{equ}[e:bound2]
|\CI|  \lesssim t^{{\alpha +\kappa-1\over 2}-{1\over 4}} |z-z'|^{{3\over 2}-\kappa}T_1 + t^{\alpha - 1 \over 2} |z-z'| T_2\;,
\end{equ}
holds. This in turn shows that \eref{e:wantedCI} holds on the set $\{|t| \ge |z-z'|^2\}$, so that the proof is complete.
\end{proof}

\begin{corollary}\label{cor:fracDer}
Let $p_t^{(k)}$ be as above, let $\kappa \in (0,{1\over 2})$, and let $\alpha \in [{1\over 2}-\kappa,1]$. Then, the bound
\begin{equs}
\Bigl|\iint_{S^1} &{p_t^{(k)}(z-y) - p_t^{(k)}(z'-y) \over |z-z'|^\alpha} Y(y)\,dX(y)\Bigr| \\
&\qquad \lesssim t^{-{1\over 4} - {k +\alpha+\kappa\over 2}} \|Y\|_\infty \|X\|_{{1\over 2}-\kappa}   + t^{-{\kappa}-{k+\alpha\over 2}} \KK^\kappa(Y,X)\;, 
\end{equs}
holds uniformly over all $z$, $z'$, and over all $t \le 1$.
\end{corollary}

\begin{proof}
The proof is the same as that of Proposition~\ref{prop:fracDer}, but using Corollary~\ref{cor:heatBasic} instead of Corollary~\ref{cor:heatDiff}.
\end{proof}

Combining both results, we also obtain

\begin{corollary}\label{cor:fracDer2}
Let $p_t^{(k)}$ be as above, let $\kappa, \delta \in (0,{1\over 2})$, and let $\alpha \in [{1\over 2}-\kappa,1]$. Then, the bound
\begin{equs}
\Bigl|\iint_{S^1} &{p_t^{(k)}(z-y) - p_t^{(k)}(z'-y) \over |z-z'|^\alpha} Y(y)\,dX(y)\Bigr|  \lesssim t^{-\kappa - {k +\alpha\over 2}} \|Y\|_{{1\over 2}-\kappa}\|X\|_{{1\over 2}-\kappa}   \\
&\qquad\qquad + t^{-{\kappa}-{k+\alpha\over 2}} \KK^\kappa(Y,X)
+  t^{-{1\over 4} - {k +\alpha+\delta\over 2}} \|Y\|_\infty \|X\|_{{1\over 2}-\delta} \;,
\end{equs}
holds uniformly over all $z$, $z'$, and over all $t \le 1$.
\end{corollary}

\begin{proof}
It suffices to write $Y(y)$ as $(Y(y) - Y(x)) + Y(x)$ for $x$ between $z$ and $z'$. One then applies Proposition~\ref{prop:fracDer} to the first
term and Corollary~\ref{cor:fracDer} to the second term.
\end{proof}

\section{Fixed point argument}
\label{sec:FP}

With these bounds at hand, we can now set up the spaces for our fixed point argument.
Our aim is to provide a rigorous meaning for local solutions to equations of the type
\begin{equ}[e:defv]
\d_t v_t = \d_x^2 v_t + \d_x \bigl(G(v_t,t)\,\d_x Y_t\bigr) + \d_x F(v_t,t)\;,
\end{equ}
where $Y$ is a fixed process taking values in 
$\CC^{{1\over 2} - \bar \kappa}$ for some $\bar \kappa > 0$, and $F$ and $G$ are sufficiently ``nice'' nonlinearities.
The precise conditions on $F$ and $G$ will be spelled out in Section~\ref{sec:solutions} below. For the moment, a typical example 
to keep in mind is
\begin{equ}[e:exampleFG]
G(v_t,t) = v_t + w_t\;,\qquad F(v_t, t) = v_t^2 + \bar w_t\;,
\end{equ} 
for some fixed processes $w$ and $\bar w$.

In full generality, such an equation simply does not make
sense in the regularity class that we are interested in. 
However, it turns out that it is well-posed if we are able to find a sufficiently regular 
``cross-area'' $\YY$ between $Y$ and $\Phi$, where $\Phi$ is given by the centred stationary solution to
\begin{equ}[e:defPhi]
\d_t \Phi_t = \d_x^2 \Phi_t + \d_x^2 Y_t\;,
\end{equ}
and if, in the example \eref{e:exampleFG}, we assume that for every fixed $t>0$, $w_t$ is controlled by $(\Phi_t, Y_t)$.
Indeed, if this is the case, then we can ``guess'' that the solution $v$ to \eref{e:defv} will locally
``look like'' $\Phi$, so that we will search for solutions belonging to a space of paths 
controlled by $\Phi$.

\subsection{Preliminary computations}
\label{sec:boundRemM}

In this subsection, we consider the following setting. We assume that we are given processes
$Y$ and $Z$ taking values in $\CC^{{1\over 2} - \bar \kappa}$ for some $\bar \kappa > 0$, and
we define a process $\Phi$ by setting
\begin{equ}
\Phi_t = P_t \Phi_0 +  \int_0^t \d_x^2 P_{t-s}\,Y_s\,ds\;.
\end{equ}
We also assume that we are given a process $\YY$ such that, for every $t > 0$ and every $x,y,z \in S^1$,
\begin{equ}[e:consistencyXX]
\YY_t(x,y) + \YY_t(y,z) - \YY_t(x,z) = \delta Y_t(x,y)\,\delta Z_t(y,z)\;,
\end{equ}
and such that $\sup_{t \le 1} \|\YY_t\|_{1-2 \bar \kappa} < \infty$. This allows to construct a rough path-valued process $\hat Y$
with components $\hat Y_t = (Y_t, Z_t)$, and with the antisymmetric part of its area process given by $\YY$.
(Its symmetric part is canonically given by half of the increment squared, as in \eref{e:geometric}.)
In the sequel, we will mostly use the case where $Z_t = \Phi_t$ for $\Phi$ given by \eref{e:defPhi}, but this
is not essential, and it will be useful in Section~\ref{sec:controlProcess} below to have the freedom to consider 
different choices of $Z$ and $\YY$.

We assume that, for almost every $t>0$, $v_t$ is controlled by $Z_t$.
With this notation fixed, we can then define a map $\CM$ by
\begin{equ}
\bigl(\CM v\bigr)_t(x) = \int_0^t \iint_{S^1} p'_{t-s} (x-y) \,v_s(y)\,dY_s(y)\,ds\;.
\end{equ}
Here, the inner integral is to be interpreted in the sense of Theorem~\ref{theo:integral}.
The map $\CM$ will be our main building block for providing a rigorous way of interpreting \eref{e:defv} in a ``mild formulation''.
However, it is important to remember that, as already noted in \cite{BurgersRough}, the notion of solution
obtained in this way does depend on the choice of $\YY$, which is not unique.

Our aim is to show that, provided that $\hat Y$ and $v$ are regular enough, $\bigl(\CM v\bigr)_t$ is 
controlled by $\Phi_t$. In the light of Corollary~\ref{cor:heatDiff} and Proposition~\ref{prop:fracDer},
we set as a shorthand
\begin{equ}
\KK_{s}^\kappa \eqdef \KK^\kappa(v_s, \hat Y_s)\;,
\end{equ}
and we define $R^\CM_t$ to be the ``remainder term'' given by
\begin{equ}
R^{\CM}_t(x,y) \eqdef \bigl(\CM v\bigr)_t(y) - \bigl(\CM v\bigr)_t(x) - v_t(x) \bigl(\Phi_t(y) - \Phi_t(x)\bigr)\;.
\end{equ}
With these notations at hand, we obtain the following bound as a straightforward corollary of the previous section:

\begin{proposition}\label{prop:boundRemainder}
For every $\kappa \in (0,{1\over 4})$ and every $\bar \kappa \in (0,{1\over 2})$, the bound
\begin{equs}
\|R^{\CM}_t\|_{{1\over 2}+2\kappa} &\lesssim t^{-{3\kappa\over 2}} \|v_t\|_\infty \|\Phi_0\|_{{1\over2}-\kappa} +\int_0^t (t-s)^{-1-\kappa-{\bar \kappa\over2}} \|Y_s\|_{{1\over2}-\bar \kappa}\,\|v_s - v_t\|_\infty \,ds \\
&\quad\!\! + \int_0^t \Bigl((t-s)^{-{3\over 4}-\kappa-\bar \kappa} \|v_s\|_{{1\over 2}-\bar \kappa}\|Y_s\|_{{1\over 2}-\bar \kappa} + (t-s)^{- {3 \over 4}-\kappa-\bar \kappa} \KK_{s}^{\bar \kappa}\Bigr)\,ds\;,
\end{equs}
holds uniformly over $t \in (0,T]$ for every $T > 0$.
\end{proposition}

\begin{proof}
We have the identity
\begin{equs}
R^{\CM}_t(x,y) &= \int_0^t \iint_{S^1} \bigl(p'_{t-s} (y-z) - p'_{t-s} (x-z)\bigr) \bigl(v_s(z) - v_t(x)\bigr)\,dY_s(z)\,ds\\
&\quad + v_t(x) \bigl(P_t \Phi_0(y) - P_t \Phi_0(x)\bigr)\;,
\end{equs}
where $P_t$ denotes the heat semigroup. Here, we have made use of the fact that $Y$ solves \eref{e:defPhi}.
We can rewrite this as
\begin{equ}
R^{\CM}_t(x,y) = T_t^1(x,y) + T_t^2(x,y) + T_t^3(x,y)\;,
\end{equ}
with 
\begin{equs}
T_t^1(x,y) &= \int_0^t \iint_{S^1} \bigl(p'_{t-s} (y-z) - p'_{t-s} (x-z)\bigr) \bigl(v_s(z) - v_s(x)\bigr)\,dY_s(z)\,ds\;,\\
T_t^2(x,y) &=  \int_0^t \bigl(v_s(x) - v_t(x)\bigr)\iint_{S^1} \bigl(p'_{t-s} (y-z) - p'_{t-s} (x-z)\bigr)\,dY_s(z)\,ds\;,\\
T_t^3(x,y) &=  v_t(x) \bigl(P_t \Phi_0(y) - P_t \Phi_0(x)\bigr)\;.
\end{equs}
As a shorthand, we furthermore rewrite $T_t^i$ as
\begin{equ}
T_t^i(x,y) = \int_0^t T_{t,s}^i(x,y)\,ds\;,\qquad i=1,2\;.
\end{equ}
Setting $\alpha = {1\over 2}+2\kappa$, it then follows from Proposition~\ref{prop:fracDer} that one has the inequality
\begin{equ}[e:boundR1]
\|T_{t,s}^1\|_{{1\over2} + 2\kappa} \lesssim (t-s)^{-{3\over 4}-2\kappa} \|v_s\|_{{1\over 2}-\kappa}\|Y_s\|_{{1\over 2}-\kappa} + (t-s)^{- {3 \over 4}-\kappa} \KK_{s}\;.
\end{equ}
On the other hand, it follows from Corollary~\ref{cor:fracDer} that
\begin{equ}[e:boundR2]
\|T_{t,s}^2\|_{{1\over2} + 2\kappa} \lesssim (t-s)^{-1-{3\kappa\over2}} \|Y_s\|_{{1\over2}-\kappa}\,\|v_s - v_t\|_\infty\;.
\end{equ}
Finally, we have
\begin{equ}[e:boundR3]
\|T_{t}^3\|_{{1\over2} + 2\kappa} \lesssim t^{-{3\kappa\over 2}} \|v_t\|_\infty \|\Phi_0\|_{{1\over2}-\kappa}\;,
\end{equ}
as a consequence of the regularising properties of the heat equation. Collecting all of these bounds concludes the proof.
\end{proof}

In order to make the bound \eref{e:boundR2} integrable in $s$, we see that if we want to be 
able to set up a fixed point argument, we also need to obtain some time regularity estimates on $\CM v$.
We achieve this with the following bound:

\begin{proposition}\label{prop:diffMv}
Let $v$ be a smooth function and let $\tilde v \eqdef \CM v$. Then, the bound
\begin{equs}
\|\tilde v_t &- \tilde v_s\|_\infty \lesssim \int_0^s \int_s^t \bigl((q-r)^{-{7\over 4}-{\kappa \over 2}} \|v_r\|_\infty \|Y_r\|_{{1\over 2}-\kappa}
+ (q-r)^{-{3\over 2}-{\kappa}} \KK_{r}^\kappa\bigr)\,dq\,dr\\
&\quad + \int_s^t (t-r)^{-{3\over 4}-{\kappa \over 2}} \|v_r\|_\infty \|Y_r\|_{{1\over2}-\kappa}\,dr+ \int_s^t (t-r)^{-{1\over 2}-\kappa} \KK_{r}^\kappa\,dr\;.\label{e:bounddiffv}
\end{equs}
holds.
\end{proposition}

\begin{proof}
In order to achieve such a bound, we write for $0 < s \le t$
\begin{equs}
(\CM v)_t(x) - (\CM v)_s(x) &= \int_0^s \iint_{S^1} \bigl(p'_{t-r} (y-z) - p'_{s-r} (y-z)\bigr) v_r(z)\,dY_r(z)\,dr\\
&\quad + \int_s^t \iint_{S^1} p'_{t-r} (y-z)v_r(z)\,dY_r(z)\,dr \\
&= \int_0^s \int_s^t \iint_{S^1} p'''_{q-r} (y-z) v_r(z)\,dY_r(z)\,dq\,dr\\
&\quad + \int_s^t \iint_{S^1} p'_{t-r} (y-z)v_r(z)\,dY_r(z)\,dr\;,
\end{equs}
where we used the identity $\d_t p_t(x) = p_t''(x)$ to obtain the second identity.
The claimed bound then follows in  a straightforward way from Corollary~\ref{cor:heatBasic}.
\end{proof}

We can also obtain a bound on the H\"older norm of $\CM v$ that is slightly better than the one that can be deduced from
the bound on $R_t^{\CM}$. 
It follows indeed  from Corollary~\ref{cor:fracDer2} that, for every $\bar \kappa \in(0, \kappa)$ and every $\kappa < {1\over 2}$, 
one has the bound
\begin{equs}[e:boundHolv]
\|(\CM v)_t\|_{{1\over2}-\kappa} &\lesssim \int_0^t(t-s)^{- {3\over 4}-{\kappa \over 2}} \bigl(\|v_s\|_{{1\over 2}-\kappa} \|Y_s\|_{{1\over 2}-\kappa} + \KK_s^\kappa\bigr)\,ds\\
&\qquad  + \int_0^t (t-s)^{{\kappa-\bar \kappa \over 2}-1} \|v_s\|_\infty \|Y_s\|_{{1\over 2}-\bar \kappa} \,ds \;.
\end{equs}

Finally, we obtain from Corollary~\ref{cor:heatBasic} the following bound on the supremum norm of $\CM v$:
\begin{equ}[e:boundSupv]
\|(\CM v)_t\|_\infty  \lesssim \int_0^t\Bigl( (t-s)^{ - {3\over 4}-{\kappa \over 2}} \|v_s\|_\infty \|Y_s\|_{{1\over 2}-\kappa} + (t-s)^{-{1\over 2}-\kappa}\KK_s^\kappa\Bigr) \,ds \;.
\end{equ}
With these calculations at hand, we are now ready to build a norm in which we can solve \eref{e:defv}
by a standard Banach fixed point argument.

\subsection{Bounds on the fixed point map}

We are now almost ready to tackle the problem of constructing local solutions to \eref{e:defv}.
In the remainder of this section, we will apply the results from the previous subsection with the special
case $Z = \Phi$. We furthermore assume that there exists a process $\YY$ such that \eref{e:consistencyXX} holds,
again with the choice $Z = \Phi$.

The above calculations suggest the introduction of a collection of space-time norms
controlling the various quantities appearing there
for functions taking values in spaces of rough paths controlled by $\Phi$.
Given a pair of functions $v$ and $v'$ in $\CC([0,T]\times S^1)$, 
we define the corresponding ``remainder'' process $R_t$ as before by
\begin{equ}[e:defRt]
R_t^v(x,y) \eqdef v_t(y) - v_t(x) - v_t'(x) \bigl(\Phi_t(y) - \Phi_t(x)\bigr)\;,
\end{equ}
where the process $\Phi$ is as in \eref{e:defPhi}. We also \textit{define} the derivative process of
$\CM v$ to be given by $ \bigl(\CM v\bigr)' = v$. 

Withe these notations at hand, we fix a (small) value $\kappa > 0$ and we define the norms
\begin{equs}[2]
\|v\|_{1,T} &\eqdef \sup_{0 < t \le T} t^\alpha \|v_t\|_{{1\over 2}-\kappa}\;, &\quad
\|v\|_{2,T} &\eqdef \sup_{0 < t \le T} t^\alpha \|v_t'\|_{\CC^{3\kappa}}\;, \\
\|v\|_{3,T} &\eqdef \sup_{0 < t \le T} t^\alpha \|R_t^v\|_{{1\over 2}+2\kappa}\;, &\quad
\|v\|_{4,T} &\eqdef \sup_{0 < t \le T} t^\beta \|v_t\|_{\infty}\;, \\
\|v\|_{5,T} &\eqdef \sup_{0 < s < t \le T} {s^\gamma \over |t-s|^{\delta}} \|v_t - v_s\|_{\infty}\;,&\quad
\|v\|_{\star,T} &\eqdef \sum_{j=1}^5 \|v\|_{j,T} \;,
\end{equs}
where $\alpha$, $\beta$, $\gamma$, and $\delta$ are exponents in $(0,1)$ that are at this stage still to be determined.
We furthermore denote by $\CB_{\star,T}$ the closure of $\CC^\infty([0,T]\times S^1)$ under $\|\cdot\|_{\star,T}$.
Here, we made an abuse of notation, since these (semi-)norms really are norms on the pair of processes $(v,v')$ and
not just on $v$. However, it will always be clear from the context what $v'$ is, so we will usually omit it from
our notations. 

Our main result in this section is the following:

\begin{proposition}\label{prop:contraction}
Assume that $Y$, $\Phi$ and $\YY$ are as in \eref{e:defPhi} and \eref{e:consistencyXX} and that, for some $\bar \kappa < \kappa$,
\begin{equ}[e:normXY]
\sup_{t\le 1} \bigl(\|\Phi_t\|_{{1\over 2} - \bar \kappa} + \|Y_t\|_{{1\over 2} - \bar\kappa} + \|\YY_t\|_{1-2\bar \kappa}\bigr) < \infty\;.
\end{equ}
Then, for every $\kappa < {1\over 10}$, there exist choices of $\alpha$, $\beta$, $\gamma$, and $\delta$ in $(0,1)$ such that
\begin{equ}
\|\CM v\|_{\star,T} \lesssim  T^\theta \|v\|_{\star, T}\;,
\end{equ}
for some $\theta > 0$, and all $T \le 1$. Here, the proportionality constant only depends on the 
quantity appearing in \eref{e:normXY}.
\end{proposition}

\begin{proof}
In the sequel, we always take for granted that $\alpha, \beta, \gamma, \delta \in (0,1)$.
We will bound the various norms appearing in $\|\cdot\|_{\star, T}$ separately, using the results from
the previous subsection. Noting that, by the definitions of $\|\cdot\|_{\star, T}$ and $\KK$, we have the bound
\begin{equ}
\KK_t \lesssim t^{-\alpha} \|v\|_{\star,T}\;.
\end{equ}
As a consequence, we have from \eref{e:boundHolv} that
\begin{equ}
\|(\CM v)_t\|_{{1\over 2}-3\kappa} \lesssim \|v\|_{\star,T}\int_0^t \bigl((t-s)^{{\kappa- \bar \kappa\over 2}-1} s^{-\beta} + (t-s)^{-{3\over 4}- {\kappa\over 2}} s^{-\alpha}\bigr)\,ds \;,
\end{equ}
so that, provided that 
\minilab{e:conditions}
\begin{equ}[e:cond1]
\kappa < {1\over 8}\;,
\end{equ}
we obtain the bound
\begin{equ}
\|\CM v\|_{1,T} \lesssim \bigl(T^{\alpha +{\kappa-\bar \kappa\over 2}-\beta} + T^{{1\over 4}-{\kappa\over 2}}\bigr)\,\|v\|_{\star,T}\;.
\end{equ}
In order for this to be bounded by a positive power of $T$, we impose the additional condition
\minilab{e:conditions}
\begin{equ}[e:cond1b]
\alpha > \beta\;.
\end{equ}

Since $\bigl(\CM v\bigr)_t' = v_t$ by definition, the bound on $\|\CM v\|_{2,T}$ is somewhat trivial. 
Using the simple interpolation bound $\|u\|_{\alpha} \lesssim \|u\|_\infty^{(\bar \alpha - \alpha)/ \bar \alpha} \|u\|_{\bar \alpha}^{\alpha / \bar \alpha}$, which holds for
$0 < \alpha < \bar \alpha < 1$, one has indeed the bound
\begin{equs}
\|\CM v\|_{2,T} &= \sup_{t \le T} t^\alpha \|v_s\|_{\CC^{3\kappa}} \lesssim \sup_{t \le T} t^\alpha \bigl(\|v_t\|_{\infty}^{1-8 \kappa \over 1-2\kappa} \|v_t\|_{{1\over 2}-\kappa}^{6\kappa \over 1-2\kappa}+ \|v_t\|_\infty\bigr) \\
&\lesssim \bigl(T^{(\alpha -\beta){1-8 \kappa \over 1-2\kappa}} + T^{\alpha - \beta} \bigr)\|v\|_{\star,T}\;,
\end{equs}
which is bounded by a positive power of $T$, since we 
assumed that $\alpha > \beta$.

For the bound on $R_t$, we make use of Proposition~\ref{prop:boundRemainder}, which yields the bound
\begin{equs}
\|\CM v\|_{3,T} &\lesssim \|v\|_{\star,T} \sup_{t \le T} t^\alpha \int_0^t \bigl((t-s)^{-{3\over4}-2\kappa} s^{-\alpha} + (t-s)^{\delta-1-{3\kappa\over 2}} s^{-\gamma} \bigr)\,ds \\
&\qquad + \|v\|_{\star,T} T^{\alpha -{3\kappa\over 2} - \beta}\;.
\end{equs}
Provided that the additional condition
\minilab{e:conditions}
\begin{equ}[e:cond3]
\delta > {3\kappa\over 2}
\end{equ}
holds, we conclude that 
\begin{equ}
\|\CM v\|_{3,T} \lesssim  \bigl(T^{{1\over 4}-2\kappa} + T^{\alpha + \delta -\gamma - {3\kappa\over 2}}+ T^{\alpha - \beta - {3\kappa\over 2}}\bigr)\|v\|_{\star,T}\;,
\end{equ}
yielding the additional conditions
\minilab{e:conditions}
\begin{equ}[e:cond4]
\alpha + \delta > \gamma + {3\kappa \over 2}\;,\quad \alpha > \beta + {3\kappa\over 2}\;.
\end{equ}

We now turn to the bound on $\|v_t\|_\infty$. It follows from \eref{e:boundSupv} that
\begin{equ}
\|(\CM v)_t\|_{\infty} \lesssim \|v\|_{\star,T} \int_0^t \bigl((t-s)^{-{\kappa\over 2}-{3\over 4}}s^{-\beta} + (t-s)^{-\kappa-{1\over 2}}s^{-\alpha}\bigr)\,ds\;,
\end{equ}
which yields the bound
\begin{equ}
\|\CM v\|_{4,T} \lesssim  \bigl(T^{{1\over 4}-{\kappa\over 2}} + T^{\beta-\alpha + {1\over 2} - \kappa}\bigr)\|v\|_{\star,T}\;,
\end{equ}
so that we have the additional condition
\minilab{e:conditions}
\begin{equ}[e:cond5]
\beta > \alpha - {1\over 2} + \kappa\;.
\end{equ}

The last bound turn out to be slightly less straightforward. Indeed, we obtain 
from Proposition~\ref{prop:diffMv} the bound
\begin{equs}
\|(\CM v)_t &- (\CM v)_s\|_{\infty} \lesssim  \|v\|_{\star,T} \int_0^s\int_s^t \Bigl({(q-r)^{-{7\over 4}-{\kappa\over 2}} \over r^{\beta}} + {(q-r)^{-{3\over 2}-\kappa}\over  r^{\alpha}}\Bigr)\,dq\,dr \\
&\quad + \|v\|_{\star,T} \int_s^t {dr\over r^\beta (t-r)^{{3\over 4}+{\kappa\over 2}}}+ \|v\|_{\star,T} \int_s^t {dr \over  r^{\alpha}(t-r)^{{1\over 2}+\kappa}}\;.\label{e:bounddiffInf}
\end{equs}
In order to bound the first term, we make use of the inequality
\begin{equ}
\int_s^t {dq \over (q-r)^{\zeta}} \lesssim {|t-s| \over  |s-r|^{\zeta}} \wedge |s-r|^{1-\zeta}
\le |t-s|^\delta |s-r|^{1-\delta-\zeta}\;,
\end{equ}
which is valid for every $\zeta > 1$, $\delta \in [0,1]$, and $r < s < t$. In particular, this implies that
\begin{equ}
\int_0^s\int_s^t \Bigl({(q-r)^{-{7\over 4}-{\kappa\over 2}} \over r^{\beta}} + {(q-r)^{-{3\over 2}-\kappa} \over r^{\alpha}}\Bigr)\,dq\,dr \lesssim |t-s|^\delta \bigl(s^{{1\over 4} - {\kappa\over 2} -\beta-\delta} + s^{{1\over 2}-\kappa - \alpha-\delta}\bigr)\;.
\end{equ}

A similar calculation allows to bound the terms on the second line of \eref{e:bounddiffInf}.
Indeed, for $\zeta, \eta \in (0,1)$, $\delta \in [0, 1-\zeta]$, and
$s < t$, one has the bound
\begin{equ}[e:boundTD1]
\int_s^t {dr\over r^\eta (t-r)^{\zeta}} \lesssim |t-s|^{1-\zeta}\bigl(s^{-\eta} \wedge |t-s|^{-\eta}\bigr) 
\lesssim |t-s|^\delta \bigl(1 \vee s^{1-\zeta-\eta-\delta}\bigr)\;.
\end{equ}

It follows from all of these considerations that, provided that
\minilab{e:conditions}
\begin{equ}[e:cond6]
 \delta  \le {1\over 4}-{\kappa\over 2}\;,
\end{equ}
one obtains the bound
\begin{equ}
\|\CM v\|_{5,T} \lesssim  \bigl(T^{\gamma -\beta-\delta + {1\over 4} - {\kappa\over 2} }
 + T^{\gamma - \alpha -\delta + {1\over 2} - \kappa}
\bigr)\|v\|_{\star,T}\;.
\end{equ}
As a consequence, we impose the condition
\minilab{e:conditions}
\begin{equ}[e:cond7]
\gamma > \Bigl(\beta + \delta - {1\over 4} + {\kappa\over 2}\Bigr) \vee \Bigl(\alpha + \delta - {1\over 2} + \kappa\Bigr)\;.
\end{equ}

It now remains to check that the conditions \eref{e:cond1}--\eref{e:cond7} can be satisfied
simultaneously for $\kappa$ small enough. For example, we can set
\begin{equ}[e:choiceExp]
\alpha = 1-2\kappa\;,\quad \beta = {1-\kappa \over 2}\;,\quad \gamma = 1-2\kappa\;,\quad \delta = 2\kappa\;.
\end{equ}
With these definitions, it is straightforward to check that the conditions \eref{e:cond1}--\eref{e:cond7}
are indeed satisfied, provided that one chooses $\kappa < {1\over 10}$.
\end{proof}

\begin{remark}
It follows from the proof of Proposition~\ref{prop:contraction} and from Proposition~\ref{prop:continuous} 
that the map $\CM$ is actually uniformly continuous on bounded sets.
\end{remark}

\subsection{Construction of solutions}
\label{sec:solutions}

We now have all the ingredients in place for the proof of our main uniqueness result.
We \textit{define} solutions to \eref{e:defv}
as solutions to the fixed point problem
\begin{equ}[e:fixedPointEquation]
v = \hat \CM(v)\;,
\end{equ}
where $\hat \CM$ is the nonlinear operator given by
\begin{equ}[e:fixedPoint]
\bigl(\hat \CM(v)\bigr)_t = P_t v_0 + \bigl(\CM G(v_{\cdot},\cdot)\bigr)_t + \d_x\int_0^t P_{t-s}F(v_s,s)\,ds \;,
\end{equ}
where $P_t$ denotes the heat semigroup. For fixed $t>0$, we will consider $ \bigl(\hat \CM(v)\bigr)_t$ as a path controlled
by $\Phi_t$ and we \textit{define} its derivative process as
\begin{equ}
\bigl(\hat \CM'(v)\bigr)_t = G(v_t,t)\;.
\end{equ}
We will assume in this section that the nonlinearity $F$ can be split into 
two parts $F = F_1 + F_2$, with different regularity properties. Our precise assumptions on $F_1$, $F_2$ and $G$ are 
summarised in the following three assumptions:

\begin{assumption}\label{ass:F1}
For every $t>0$, the map $F_1(\cdot,t)$ maps $\CC(S^1)$ into itself. Furthermore, it satisfies the bounds
\begin{equ}
\|F_1(v,t)\|_\infty \lesssim 1 + \|v\|_\infty^2\;,\;\; \|F_1(u,t)-F_1(v,t)\|_\infty \lesssim \|u-v\|_\infty \bigl(1 + \|u\|_\infty+ \|v\|_\infty\bigr)\;,
\end{equ}
for all $u$ and $v$ in $\CC(S^1)$, with a proportionality constant that is uniform over bounded time intervals.
\end{assumption}

\begin{assumption}\label{ass:F2}
There exists $\eta < {1\over 2}$ such that, for every $t>0$, the map $F_2(\cdot,t)$ maps 
$\CC(S^1)$ into $\CC^{-\eta}$. Furthermore, it satisfies the bounds
\begin{equ}
\|F_2(v,t)\|_{-\eta} \lesssim 1 + \|v\|_\infty\;,\quad \|F_2(u,t)-F_2(v,t)\|_{-\eta} \lesssim \|u-v\|_\infty \;,
\end{equ}
for all $u$ and $v$ in $\CC(S^1)$, with a proportionality constant that is uniform over bounded time intervals.
\end{assumption}

\begin{assumption}\label{ass:G}
For every $t>0$, the map $G(\cdot,t)$ maps $\CC(S^1)$ into itself. Furthermore, if $(v,v')$ is controlled by $\Phi_t$, then
this is also the case for $G(v,t)$, for some ``derivative process'' $G'(v,v',t)$. Denote by $R^v_t$ the remainder
for $(v,v')$  and by $R^G_t$ the remainder for $(G(v,t), G'(v,v',t))$. Then, there exists $\kappa \in (0,{1\over 4})$ such that, 
for every $\zeta \in (0,{1\over 2}-\kappa)$, one has the bounds
\begin{equ}
\|G(v,t)\|_{\zeta} \lesssim 1 + \|v\|_{\zeta}\;,\qquad \|G(u,t)-G(v,t)\|_{\zeta} \lesssim \|u-v\|_{\zeta} \;.
\end{equ}
Furthermore, for the same $\kappa > 0$, one has the bounds
\begin{equs}
\|G(v, t) - G(v,s)\|_\infty &\lesssim |t-s|^{2\kappa} \;,\\
\|G(v, t) - G(u,t) &- G(v,s) + G(u,s)\|_\infty \lesssim \|u-v\|_\infty\;,\\
\|G'(v,t)\|_{\CC^{3\kappa}} &\lesssim 1 + \|v'\|_{\CC^{3\kappa}}+ \|v\|_{\CC^{{1\over 2}-\kappa}} + \|R^v_t\|_{{1\over 2}+2\kappa}\;,\\
\|G'(u,u',t)-G'(v,v',t)\|_{\CC^{3\kappa}} &\lesssim  \|u'-v'\|_{\CC^{3\kappa}}+ \|u-v\|_{\CC^{{1\over 2}-\kappa}} + \|R^u_t-R^v_t\|_{{1\over 2}+2\kappa}\;,\\
\|R^{G(v)}_t\|_{{1\over 2}+2\kappa} &\lesssim 1 + \|v'\|_{\CC^{3\kappa}}+ \|v\|_{\CC^{{1\over 2}-\kappa}} + \|R^v_t\|_{{1\over 2}+2\kappa}\;,\\
\|R^{G(u)}_t(u)-R^{G(v)}_t(v)\|_{{1\over 2}+2\kappa} &\lesssim  \|u'-v'\|_{\CC^{3\kappa}}+ \|u-v\|_{\CC^{{1\over 2}-\kappa}} + \|R^u_t-R^v_t\|_{{1\over 2}+2\kappa}\;,
\end{equs}
with a proportionality constant that is uniform over bounded time intervals.
\end{assumption}

We now have all the necessary ingredients to solve \eref{e:fixedPoint} by a fixed point argument.

\begin{theorem}\label{theo:main}
Assume that there exist $\kappa < {1\over 12}$ and $\eta < {1\over 2} - 2\kappa$ such that Assumptions~\ref{ass:F1}--\ref{ass:G} hold.
Assume furthermore that $Y$, $\Phi$ and $\YY$ are as in \eref{e:defPhi} and \eref{e:consistencyXX} and that the bound 
\eref{e:normXY} holds for
some $\bar \kappa \in (0,\kappa)$.

Then,  for every initial condition $v_0 \in \CC^{\zeta -1}$ with $\zeta > 2\kappa$, there exists a choice of 
exponents $\alpha$, $\beta$, $\gamma$ and $\delta$ such that 
the nonlinear operator $\hat \CM$ maps $\CB_{\star,T}$ into itself for every $T>0$. Furthermore, there exists $T_\star > 0$ such 
that the fixed point equation \eref{e:fixedPointEquation} admits a solution in $\CB_{\star,T_\star}$, and this solution is unique.
\end{theorem}

\begin{proof}
We choose $\alpha$, $\beta$, $\gamma$ and $\delta$ as in \eref{e:choiceExp}.
With this choice, it suffices to show that there exists $T>0$ such that $\hat \CM$
maps some ball of $\CB_{\star, T}$ into itself and is a contraction there.

We first consider the first term in $\hat \CM$, namely $P_t v_0$. It follows from Proposition~\ref{prop:interpolation} that
one has the bounds
\begin{equ}
\|P_t v_0\|_{{1\over 2} + 2\kappa} \lesssim t^{-{3\over 4}} \|v_0\|_{\zeta-1} \;,\qquad
\|P_t v_0\|_{\infty} \lesssim t^{\kappa - {1\over 2}} \|v_0\|_{\zeta-1} \;,
\end{equ}
as well as
\begin{equs}
\|P_t v_0 - P_s v_0\|_{\infty} &= \bigl\|\bigl(P_{t-s} - 1\bigr) P_s v_0\bigr\|_{\infty} \lesssim |t-s|^{\delta} \|P_s v_0\|_{2\delta } \\
&\lesssim  |t-s|^\delta s^{\kappa - \delta - {1 \over 2} } \|v_0\|_{\zeta-1} \;. \label{e:boundTimeReg}
\end{equs}
Here, we made use of the fact that $\zeta > 2\kappa$ by assumption.
Since, by our assumptions, we have $\alpha > {3\over 4}$, $\beta > {1\over 2}-\kappa$, and $\gamma > {1\over 2}+\delta-\kappa$,
it follows that we have the bound
\begin{equ}
\|P_\cdot v_0\|_{\star, T} \lesssim T^\theta \|v_0\|_{\zeta - 1}\;,
\end{equ}
for some $\theta > 0$. (Note that we consider the derivative process of $P_t v_0$ to be simply $0$.)

In the next step, define a nonlinear map $\CN$ by 
\begin{equ}
\bigl(\CN v\bigr)_t = \d_x \int_0^t P_{t-s} F(v_s, s)\,ds\;.
\end{equ}
It then follows from Proposition~\ref{prop:interpolation} and the assumptions on $F$ that
\minilab{e:boundF}
\begin{equs}
\|(\CN v)_t\|_{{1\over 2} + 2\kappa} &\lesssim \int_0^t \Bigl((t-s)^{-\kappa - {3\over 4}} \|F_1(v_s,s)\|_\infty + (t-s)^{-{\eta \over 2} - {3\over 4} - \kappa}\|F_2(v_s,s)\|_{-\eta}\Bigr)\,ds \\
&\lesssim \bigl(1+\|v\|_{\star,T}\bigr)^2 \int_0^t \Bigl((t-s)^{-\kappa - {3\over 4}}s^{-2\beta} + (t-s)^{-{\eta \over 2} - {3\over 4} - \kappa}s^{-\beta}\Bigr)\,ds\\
&\lesssim \bigl(1+\|v\|_{\star,T}\bigr)^2 \Bigl(T^{{1\over 4}-\kappa-2\beta} + T^{{1\over 4} -{\eta \over 2} - \kappa-\beta}\Bigr)\;.\label{e:boundF1}
\end{equs}
Similarly, the supremum norm is bounded by
\minilab{e:boundF}
\begin{equ}[e:boundF2]
\|(\CN v)_t\|_{\infty} \lesssim \bigl(1+\|v\|_{\star,T}\bigr)^2 \Bigl(T^{{1\over 2}-2\beta} + T^{{1\over 2} -{\eta \over 2}-\beta}\Bigr)\;.
\end{equ}
Regarding the time regularity bound, we have as in \eref{e:boundTimeReg} the bound
\begin{equs}
\|(\CN v)_t - (\CN v)_s\|_{\infty} &\lesssim \int_s^t  \Bigl((t-r)^{- {1\over 2}} \|F_1(v_r,r)\|_\infty + (t-r)^{-{\eta \over 2} - {1\over 2}}\|F_2(v_r,r)\|_{-\eta}\Bigr)\,dr \\
&\quad +  |t-s|^\delta \int_0^s  (s-r)^{- {1\over 2}-\delta} \|F_1(v_r,r)\|_\infty\,dr\\
&\quad +  |t-s|^\delta \int_0^s  (s-r)^{-{\eta \over 2} - {1\over 2}-\delta}\|F_2(v_r,r)\|_{-\eta}\,dr\;. \label{e:boundN}
\end{equs}
Making use of the bound \eref{e:boundTD1} and otherwise proceeding as before, we conclude that
\minilab{e:boundF}
\begin{equ}[e:boundF3]
\|(\CN v)_t - (\CN v)_s\|_{\infty} \lesssim |t-s|^\delta \bigl(1+\|v\|_{\star,T}\bigr)^2 \Bigl(T^{{1\over 2}-2\beta-\delta} + T^{{1\over 2} -{\eta \over 2}-\beta-\delta}\Bigr)\;.
\end{equ}
Note here that, since $\kappa < {1\over 12}$ by assumption, we have $\eta < 1-8\kappa$. This ensures that the bound $\delta < {1\over 2}-{\eta \over 2}-\delta$,
which is required in \eref{e:boundTD1}, does indeed hold.

Collecting the bounds from \eref{e:boundF}, it is lengthy but straightforward to check that, thanks to our assumptions on $\kappa$, $\eta$, and $\delta$,
there exists $\theta>0$ such that one does have the bound
\begin{equ}
\|\CN v\|_{\star,T} \lesssim T^\theta \bigl(1+\|v\|_{\star,T}\bigr)^2\;.
\end{equ}
Similarly, one can verify in exactly the same way that one also has the bound
\begin{equ}
\|\CN u - \CN v\|_{\star,T} \lesssim T^\theta \|u-v\|_{\star,T} \bigl(1+\|u\|_{\star,T}+\|v\|_{\star,T}\bigr)\;.
\end{equ}
Furthermore, the assumptions on the map $G$ are set up precisely in such a way that one has
\begin{equ}
\|G(v_\cdot,\cdot)\|_{\star,T} \lesssim 1 + \|v\|_{\star,T}\;,\quad
\|G(u_\cdot,\cdot) - G(v_\cdot,\cdot)\|_{\star,T} \lesssim \|u-v\|_{\star,T}\;.
\end{equ}
Combining this with Proposition~\ref{prop:contraction}, as well as the bounds on $\CN$ and $P_t v_0$ that we just obtained,
we conclude that 
\begin{equs}[e:boundMhat]
\|\hat \CM v\|_{\star,T} &\lesssim T^\theta \bigl(1+\|v\|_{\star,T}\bigr)^2\;,\\ 
\|\hat \CM u-\hat \CM v\|_{\star,T} &\lesssim T^\theta \bigl(1+\|u\|_{\star,T}+\|v\|_{\star,T}\bigr)\|u-v\|_{\star,T}\;. 
\end{equs}
It follows immediately that, for $T>0$ small enough, there exists a ball around the origin in $\CB_{\star,T}$ which
is left invariant by $\hat \CM$ and such that $\hat \CM$ admits a unique fixed point in this ball.

The uniqueness of this fixed point in all of $\CB_{\star,T}$ now follows from the following argument. 
Denote by $T_\star$ and $v_\star$ the time horizon and fixed point that were just constructed and assume that there
exists a fixed point $v \neq v_\star$ for $\hat \CM$. Note now that, by the definition of the norm $\|\cdot\|_{\star,T}$, the natural
restriction operator from $\CB_{\star,T_\star}$ to $\CB_{\star,T}$ is a contraction for every $T<T_\star$.
Since it follows from \eref{e:boundMhat} that there exists some $T \in (0, T_\star)$ such that $\hat \CM$ is a contraction in
the ball of radius $\|v\|_{\star, T_\star}$ in $\CB_{\star,T}$, this shows that on the interval $[0,T]$, $v$ must agree with $v_\star$.
The uniqueness claim then follows by iterating this argument.
\end{proof}

%
%

Once we do have a unique solution to a PDE, we can perform the usual kind of bootstrapping argument to improve the
regularity estimates provided ``for free'' by the fixed point argument. In our case, we can certainly not expect the solution $v$ to be
more regular than the process $\Phi$, which in turn cannot be expected to be more regular than $Y$. However, it is possible to
slightly improve the regularity estimates for the remainder term $R^v_t(x,y)$ defined in \eref{e:defRt}. Our current bounds
show that $\|R^v_t\|_{{1\over 2}+2\kappa} < \infty$, which is not a very good bound in general.

Given the (lack of) regularity of $F_1$, we certainly do not expect $\|R^v_t\|_1$ to be finite, but
it turns out that this can be approached arbitrarily close:

\begin{proposition}\label{prop:maxReg}
Let the assumptions of Theorem~\ref{theo:main} hold, let $v$ be the unique maximal solution to \eref{e:fixedPoint} with lifetime $T_\star$. Then, for $0 < s < t < T_\star$, one has
\begin{equ}[e:boundtimereg]
\|v_t - v_s\|_\infty \lesssim |t-s|^{\gamma}\;,
\end{equ}
for every $\gamma < {1\over 4} - {\bar \kappa \over 2}$. The proportionality constant is uniform over every
compact time interval in $(0,T_\star)$. 
Furthermore, one has
\begin{equ}
\|R_t^v\|_{\bar \gamma} < \infty\;,
\end{equ}
for every $\bar \gamma < 1 - (2\bar \kappa \vee \eta)$ and every $t \in (0,T_\star)$.
\end{proposition}

\begin{proof}
Since it is possible to concatenate solutions to \eref{e:fixedPoint}, we can restart the solution at some
positive time. As a consequence, since we know that the solution belongs to $\CB_{\star,T}$, we can assume that
$\|v_t\|_{{1\over 2} - \kappa}$ and $\|R_t^v\|_{{1\over 2} + 2\kappa}$ are bounded uniformly in time.
Since we furthermore know that $v_t$ is controlled by $\Phi_t$ with remainder $R_t^v$,
we obtain that actually $\|v_t\|_{{1\over 2} - \bar \kappa}$ is bounded.
Furthermore, since $v_t' = G(v_t,t)$ by construction, we also have $\|v_t'\|_{{1\over 2} - \bar \kappa}$ uniformly bounded.

It then follows form \eref{e:boundN} that
\begin{equ}
\|\CN v_t - \CN v_s\| \lesssim |t-s|^\gamma\;,
\end{equ}
provided that $\gamma < {1\over 2} - {\eta\over 2}$. It follows from Proposition~\ref{prop:diffMv} that a similar bound 
holds for the term $\CM G(v_t,t)$, provided that $\gamma < {1\over 4} - {\bar \kappa \over 2}$, so that the first bound follows.

For the second bound, it follows from Proposition~\ref{prop:interpolation} that the bound holds for $\CN v_t$.
To show that it also holds for $\CM G(v_t,t)$, it suffices to apply Proposition~\ref{prop:boundRemainder} 
by noting that the right hand side of that bound is integrable
as soon as $\kappa <  {1\over 4} - \bar \kappa$, thanks to the bound \eref{e:boundtimereg}.
\end{proof}

An important special case is given by the case when 
\begin{equ}
G(v,t) = v + w_t\;,
\end{equ}
for some fixed process $w$ such that $w_t$ is controlled by $\Phi_t$ for every $t$ with
\begin{equ}[e:asswt]
\sup_{t \le T} \|w'_t\|_{\CC^{3\kappa}} < \infty\;,\quad \sup_{t \le T} \|R_t^w\|_{{1\over 2}+2\kappa} < \infty\;,\quad
\sup_{s,t \le T} {\|w_t-w_s\|_{\infty}\over |t-s|^\kappa} < \infty\;,
\end{equ}
One then has:

\begin{proposition}\label{prop:expPhi}
Let the assumptions of Theorem~\ref{theo:main} hold and let $G(v,t) = cv+w_t$ with $w$ as above and $c \in \R$. 
Let  $v$ be the unique maximal solution to \eref{e:fixedPoint} with lifetime $T_\star$. Then, for every $t \in (0,T_\star)$, one has the decomposition
\begin{equ}[e:defRtvv]
e^{-c\Phi_t(x)} v_t(x) = \iint_0^x e^{-c\Phi_t(z)} w_t(z) \,d\Phi_t(z) + R_t(x)\;,
\end{equ}
with $R_t \in \CC^{\bar \gamma}$ for every $\bar \gamma < 1 - (2\bar \kappa \vee \eta)$.
\end{proposition}

\begin{remark}
The rough integral appearing on the right hand side is well-posed since, by assumption, 
 $w_t$ is controlled by $\Phi_t$, so that the same is true for the integrand in \eref{e:defRtvv}.
\end{remark}

\begin{remark}
It is not guaranteed that $\iint_0^{2\pi} e^{-c\Phi_t(z)} w_t(z) \,d\Phi_t(z) = 0$, so the two functions
appearing in the right hand side of \eref{e:defRtvv} are not necessarily periodic. This is irrelevant however,
since one can easily rectify this by adding to each of them a suitable multiple of $x$.
\end{remark}

\begin{proof}[of Proposition~\ref{prop:expPhi}]
Setting $\tilde v_t(x) = \iint_0^x e^{-c\Phi_t(z)} w_t(z) \,d\Phi_t(z)$,  it follows from \eref{e:asswt} and 
Theorem~\ref{theo:integral} that
\begin{equ}[e:boundvtilde]
\delta \tilde v_t(x,y) = e^{-c\Phi_t(x)} w_t(x)  \delta \Phi_t(x,y) + R_t^{\tilde v}(x,y)\;,
\end{equ}
with $\|R_t^{\tilde v}\|_{1-2\bar \kappa} < \infty$.

On the other hand, we know from Proposition~\ref{prop:maxReg} that 
\begin{equ}
\delta v_t(x,y) = \bigl(cv_t(x) + w_t(x)\bigr) \delta \Phi_t(x,y) + R_t^{v}(x,y)\;,
\end{equ}
with $\|R_t^v\|_{\bar \gamma} < \infty$. In particular, this implies that
\begin{equs}
v_t(y) &= v_t(x) \bigl(1 +  c\delta \Phi_t(x,y)\bigr) + w_t(x)\, \delta \Phi_t(x,y) + R_t^{v}(x,y) \\
&= v_t(x) e^{c\delta \Phi_t(x,y)} + w_t(x)\, \delta \Phi_t(x,y) + \tilde R_t^{v}(x,y) \;,
\end{equs}
where we also have $\|\tilde R_t^v\|_{\bar \gamma} < \infty$. Multiplying both sides by $e^{-c\Phi_t(y)}$ and
subtracting \eref{e:boundvtilde} from the resulting expression, 
we obtain the identity
\begin{equs}
e^{-c\Phi_t(y)} v_t(y) - e^{-c\Phi_t(x)} v_t(x) &= \iint_x^y e^{-c\Phi_t(z)} w_t(z) \,d\Phi_t(z) \\
&\qquad +  e^{-c\Phi_t(y)}\tilde R_t^{v}(x,y) - R_t^{\tilde v}(x,y)\;.
\end{equs}
It follows that the function $R_t$ defined in \eref{e:defRtvv} satisfies the identity $\delta R_t(x,y) = e^{-c\Phi_t(y)}\tilde R_t^{v}(x,y) - R_t^{\tilde v}(x,y)$,
so that  the claim follows at once.
\end{proof}

\section{Construction of the universal process}
\label{sec:constructX}

The aim of this section is to prove the convergence of the processes $Y^\tau_\eps$ to some
limiting processes $Y^\tau$. Actually, it turns out that the constant Fourier mode requires a separate treatment, 
so we only consider the centred processes $X^\tau_\eps$ here, which were defined in \eref{e:defXtau}.
The aim of this section is to show that, for every binary tree $\tau$, there exists a process $X^\tau$ such that
\begin{equ}
X^\tau = \lim_{\eps \to 0} X^\tau_\eps\;,
\end{equ}
in a suitable sense, and to obtain quantitative estimates on $X^\tau$.

\subsection[Construction of first level]{Construction of $X^\2$.}

A crucial observation for the sequel is that, for $k \neq 0$, the covariance of the Fourier modes of $X_\eps^\1$ is  given by
\begin{equ}
\E X^\1_{\eps,k}(s) X^\1_{\eps,\ell}(t) = \delta_{k,-\ell} {\phi^2(\eps k)\over k^2} \exp(-k^2 |t-s|)\;.
\end{equ}
Since $X_\eps$ and $X_\eps^\1$ will virtually always arise via their spatial derivatives, it will be convenient
to introduce a notation for this. We therefore define $\bar X^\tau \eqdef \d_x X^\tau$ as in \eref{e:defYbar}, and similarly for $\bar X^\tau_\eps$,
so that one has the identity
\begin{equ}[e:covZ1]
\E \bar X^\1_{\eps,k}(s) \bar X^\1_{\eps,\ell}(t) = \delta_{k,-\ell} \phi^2(\eps k) \exp(-k^2 |t-s|)\;,
\end{equ}
provided that $k \neq 0$.
This is where our convention \eref{e:convention} shows its advantage: this choice of normalisation for the driving noise
ensures that we do not have any
constant prefactor appearing in \eref{e:covZ1} so that, except for the constant mode, the space-time 
correlation function of $\bar X^\1$ is precisely equal
to the heat kernel.

With these notations at hand, the process $X^\2_\eps$ is given as the stationary solution to
\begin{equ}[e:equX2eps]
\d_t X^\2_\eps = \d_x^2 X^\2_\eps + \Pi_0^\perp |\bar X^\1_\eps|^2\;,
\end{equ}
so that its Fourier modes are given for $k\neq 0$ by the identity
\begin{equ}[e:X2eps]
X^\2_{\eps,k}(t) = \int_{-\infty}^t e^{-k^2(t-s)} \sum_{\ell\in \Z} \bar X^\1_{\eps,\ell}(s) \bar X^\1_{\eps,{k-\ell}}(s)\,ds\;.
\end{equ}

%
We now show that $X^\2_\eps$ converges to a limiting process $X^\2$
in the following sense:
\begin{proposition}\label{prop:procX1}
There exists a process $X^\2$ such that the weak convergence $X^\2_\eps \to X^\2$ takes place in $\CC([-T,T],\CC^\alpha)\cap \CC^\beta([-T,T],\CC)$
for every $\alpha < 1$, every $\beta < {1\over 2}$, and every $T>0$.
\end{proposition}

Before we proceed to the proof, we recall  
Wick's theorem (sometimes also called Isserlis's theorem) on the higher order moments of Gaussian random variables:
\begin{proposition}\label{prop:momentGaussian}
Let $T$ be a finite index set and let $\{X_\alpha\}_{\alpha \in T}$ be a collection of real or complex-valued 
centred jointly Gaussian random variables. Then,
\begin{equ}
\E \prod_{\alpha \in T} X_\alpha = \sum_{P \in \CP(T)} \prod_{\{\alpha,\beta\} \in P} \E X_\alpha X_\beta\;.
\end{equ}
\end{proposition}

\begin{proof}[of Proposition~\ref{prop:procX1}]
The proof is an almost direct application of Proposition~\ref{prop:generalConvergence} below.
Indeed, writing $\Z_\star = \Z\setminus \{0\}$ as a shorthand,
we can set $\JJ = \Z_\star^2$ and write elements in $\JJ$ as $\kappa = (k,\ell) \in \JJ$.
For $\kappa = (k,\ell)$, we furthermore set $g_\kappa(x) = \exp(ik x)$ and
$C_\eps(\kappa) = \phi(\eps \ell)\phi(\eps (k-\ell))$. With this notation, the involution $\iota$ appearing in the assumptions
is given by $(k,\ell) \leftrightarrow (-k,\ell)$.

Since it follows from \eref{e:X2eps} and Proposition~\ref{prop:momentGaussian} that
\begin{equ}
X^\2_{\eps}(x,t) =  \sum_{\kappa = (k,\ell) \in \JJ}{g_\kappa(x) C_\eps(\kappa)} \int_{-\infty}^te^{-k^2(t-s)} \bar X^\1_\ell(s)\bar X^\1_{k-\ell}(s)\,ds\;,
\end{equ}
we set
\begin{equ}
f_\kappa(t) = \int_{-\infty}^te^{-k^2(t-s)} \bar X^\1_\ell(s)\bar X^\1_{k-\ell}(s)\,ds\;,
\end{equ}
in order to be in the framework of Proposition~\ref{prop:generalConvergence}. It then follows from \eref{e:covZ1} that
\begin{equ}
\E f_\kappa(t) f_{\bar \kappa}(s) = \delta_{k,-\bar k} \bigl(\delta_{\ell,-\bar \ell} + \delta_{\ell-k,\bar \ell}\bigr) K_\kappa(s,t)\;, 
\end{equ}
where the kernels $K_\kappa(s,t)$ are given for $\kappa=(k,\ell)$ by
\begin{equs}
K_\kappa(s,t) &= \int_{-\infty}^t\int_{-\infty}^se^{-k^2(t+s-r-r') -\ell^2 |r-r'| -(k-\ell)^2 |r-r'|}\,dr\,dr' \\
&= \int_{-\infty}^{t-s}\int_{-\infty}^0e^{-k^2(t-s-r-r') -\ell^2 |r-r'| -(k-\ell)^2 |r-r'|}\,dr\,dr'\;.
\end{equs}
Here we assumed $t > s$ for simplicity, but the kernels are of course symmetric in $s$ and $t$. 
A lengthy but straightforward calculation then shows that one has the identity
\begin{equ}[e:exprKm]
K_\kappa(s,t) = {k^2 e^{-(\ell^2 + (k-\ell)^2)|t-s|}
- (\ell^2 + (k-\ell)^2) e^{-k^2|t-s|} \over k^2 \bigl(k^2 - \ell^2 - (k-\ell)^2\bigr) \bigl(k^2 + \ell^2 + (k-\ell)^2\bigr)} \;.
\end{equ}
It will be convenient in the sequel to introduce the shorthand notation
\begin{equ}
\Delta_{\kappa,\bar \kappa} = \delta_{k,-\bar k} \bigl(\delta_{\ell,-\bar \ell} + \delta_{\ell-k,\bar \ell}\bigr)\;,\qquad \kappa = (k,\ell)\;,\quad \bar \kappa = (\bar k, \bar \ell)\;.
\end{equ}
With this notation, we then have, for $F_{\kappa\eta}$ as in Proposition~\ref{prop:generalConvergence}, the identity
\begin{equ}
F_{\kappa\eta}(t) \propto \Delta_{\kappa,\eta}K_\kappa(t,t) = {\Delta_{\kappa,\eta} \over k^2 \bigl(k^2 + \ell^2 + (k-\ell)^2\bigr)} = {\Delta_{\kappa,\eta} \over 2k^2 \bigl(k^2 -k\ell+ \ell^2\bigr)}\;.
\end{equ}
Using Proposition~\ref{prop:bounddiff} below, we furthermore obtain from \eref{e:exprKm} the bound
\begin{equ}
\hat F_{\kappa\eta}(s,t) \propto  \Delta_{\kappa,\eta} |K_\kappa(s,t) - K_\kappa(0,0)| \le F_{\kappa\eta} \wedge {\ell^2 + (k-\ell)^2 \over k^2 -k\ell+ \ell^2}|t-s|^2\;.
\end{equ}
In particular, for every $\beta \le 1$, one has
\begin{equ}
\hat F_{\kappa\eta}(s,t) = \Delta_{\kappa,\eta} |t-s|^{2\beta} {|\ell^2 + (k-\ell)^2|^\beta |k|^{2\beta-2} \over k^2 -k\ell + \ell^2}\;.
\end{equ}
Since furthermore the Lipschitz constant of $g_\kappa$ is given by $G_\kappa = |k|$, the conditions \eref{e:assumFalpha}
boil down to
\begin{equs}
\sum_{k \neq 0} |k|^{2\alpha} \sum_{\ell \neq 0} {1\over k^2 \bigl(k^2 -k\ell+ \ell^2\bigr)} &< \infty\;, \\
\sum_{k \neq 0} |k|^{2\beta-2} \sum_{\ell \neq 0} {|\ell^2 + (k-\ell)^2|^\beta \over k^2 -k\ell+ \ell^2}&< \infty\;.
\end{equs}
Approximating the sum by an integral, one can check that
\begin{equ}
\sum_{\ell \neq 0} {1\over k^2 \bigl(k^2 -k\ell+ \ell^2\bigr)} \lesssim {1\over k^3}\;,
\end{equ}
so that the first condition is indeed satisfied as soon as $\alpha < 1$.

Regarding the second condition, one
similarly obtains
\begin{equ}
 \sum_{\ell \in \Z}{|\ell^2 + (k-\ell)^2|^\beta \over k^2 -k\ell+ \ell^2} \lesssim {1\over k^{1-2\beta}}\;,
\end{equ}
provided that $\beta < {1\over 2}$ (otherwise, the expression is not summable).
As a consequence, the second condition reduces to $k^{4\beta-3}$ being summable, which 
is again the case if and only if $\beta < {1\over 2}$. This concludes the proof.
\end{proof}

\begin{remark}
It also follows from the proof that, for any fixed $t$, $X^\2(t) \not \in H^1$ almost surely, since one has
$\E \bigl(X^\2_k(t)\bigr)^2 \sim {1\over k^3}$.
\end{remark}

\begin{remark}\label{rem:renorm}
In light of the construction just explained, we can understand how the limit $\alpha = {1\over 8}$ arises in
\eref{e:KPZalpha}. Indeed, it turns out that $\alpha > {1\over 8}$ is precisely the borderline for which the right hand side
in \eref{e:equX2eps} converges to a limit for every \textit{fixed} value of $t$. The reason why we can break through
this barrier is that, instead of making sense of the right hand side for fixed $t$, we only need to make sense of
its time integral. (This was already remarked in \cite{Milton,SigurdRecent}.) 

If we  use this trick and then continued with the classical tools as in \cite{MR2365646}, we
would however hit another barrier at $\alpha = {1\over 20}$ when the product $\bar X^\1\, \bar X^\2$ ceases
to make sense classically (i.e.\ in the sense of Proposition~\ref{prop:Holder}). 
Treating this term also ``by hand'' in order to overcome that barrier, 
it would not be too difficult to make sense of \eref{e:KPZalpha} for every $\alpha > 0$. The most difficult barrier
to break is the passage from $\alpha > 0$ to $\alpha = 0$ since there are then \textit{infinitely many} products
that cease to make sense classically. More precisely, it will be clear from the remainder of this section that
if $\tau$ is any tree of the form $\tau = [\bullet,\bar \tau]$, then the product $\bar X^\1\,\bar X^\tau$ does not
make sense classically. 
\end{remark}

\subsection{A more systematic approach}
\label{sec:systematic}

We would now like to similarly construct a process $X^\3$ that is the limit of $X^\3_\eps$
as $\eps \to 0$. For $k \neq 0$, it follows from the definitions that
\begin{equs}
X^\3_{\eps,k}(t) &= i\sum_{\ell \in \Z} \int_{-\infty}^t e^{-k^2(t-s)} (k-\ell) \bar X^\1_{\eps,\ell}(s)X^\2_{\eps,k-\ell}(s)\,ds\\
&= i\sum_{\ell + m + p = k} \int_{-\infty}^t \int_{-\infty}^s e^{-k^2(t-s) - (k-\ell)^2(s-r)} (k-\ell) \\
&\qquad\qquad\qquad\qquad\qquad\qquad \times \bar X^\1_{\eps,\ell}(s)\bar X^\1_{\eps,m}(r)\bar X^\1_{\eps,p}(r)\,dr\,ds\;.
\end{equs}
At this stage, it becomes clear that a somewhat more systematic approach to the estimation of the 
correlations is needed. In principle, one could try the same ``brute force'' approach as in Proposition~\ref{prop:procX1} 
and obtain exact expressions for the
correlations of $X^\3_\eps$, but it rapidly becomes clear that bounding the resulting 
expressions is a rather boring and not very instructive task. By the time we want to construct $X^\5$,
a brute force approach is definitely out of the question.
Instead, we will now provide a more systematic approach to estimating the correlations of $X^\tau_\eps$ for
more complicated trees $\tau$. 

\begin{table*}\centering
\renewcommand{\arraystretch}{1.1}
\begin{tabular}{ll}\toprule
Symbol & Meaning\\
\midrule
$\CE(\tau)$ & Edges of $\tau$\\
$\CV(\tau)$ & Vertices of $\tau$ (including leaves)\\
$\ell(\tau)$ & Leaves of $\tau$\\
$i(\tau)$ & Inner vertices of $\tau$\\
$\LL^\tau$ & Proper integer labelling of edges of the tree $\tau$\\
$\TT^\tau$ & Ordered real labelling of vertices of the tree $\tau$\\
$\CP^\tau$ & Pairings of two copies of the leaves of $\tau$\\
$\LL_P^\tau$ & Elements in $\LL^\tau \times \LL^\tau$ respecting the pairing $P$\\
$S_\tau$& Group of isometries of $\tau$\\
\bottomrule
\end{tabular}
\caption{Notations for various objects associated to a given tree $\tau$.}
\end{table*}

Although the setting is quite different, our approach is inspired by 
the classical construction of Feynman diagrams in perturbative quantum field theory 
(see for example \cite{MR2131010} for an introduction), with the heat kernel playing the role
of the propagator.
We will associate to any given process $X^\tau$ a number of ``Feynman diagrams'', that turn out 
in our case to be graphs with certain properties. Each of these graphs encodes a multiple sum of a multiple
integral that needs to be bounded in order to ascertain the convergence of the corresponding process 
$X^\tau_\eps$ to a limit. The main achievement of this section is to describe a very simple ``graphical'' 
algorithm that provides a sufficient condition for this convergence which is not too difficult to check in
practice.

First, for a given binary tree $\tau \neq \bullet$, we introduce the set $\LL^\tau$ of ``proper'' labellings of $\tau$
which consists of all possible ways of associating to each edge $e$ of $\tau$ a non-zero integer $L_e \in \Z_\star$,
with the additional constraints that Kirchhoff's law should be satisfied. In other words, for every node $v$ that is neither the
root nor a leaf, the sum of the labels of the two edges connecting $v$ to its children should be equal to the label of 
the edge connecting $v$ to its parent.
For example, we have
\begin{equ}
\mhpastefig{tree12label1} \in \LL^\3\;,\qquad \mhpastefig{tree12label2} \not \in \LL^\3\;.
\end{equ}
Given a labelling $L \in \LL^\tau$, we also denote by $\rho$ the root vertex and by $\rho(L)$ the sum of the labels of the edges attached to $\rho$.
(In the first example above, we would have $\rho(L) = 7$.) Each label of a proper labelling should be
thought of as a ``Fourier mode'' and the reason behind Kirchhoff's law is the way exponentials behave under multiplication. 
The precise meaning of this will soon become clear.

For a given binary tree $\tau$, we denote by $\ell(\tau)$ the set of leaves and by $i(\tau)$
the set of inner vertices (the complement of $\ell(\tau)$ in the set $\CV(\tau)$ of all vertices of $\tau$).
It will then be useful to introduce a labelling of the interior vertices of a binary tree 
by real numbers, which should this time be thought of as ``times''
instead of ``Fourier modes''. Denoting by ``$\le$'' the canonical partial order of a rooted tree
(i.e. $u \le v$ if $u$ lies on the path from $v$ to the root), 
we denote by $\TT^\tau$ the set of all labellings that associate to each vertex
$v \in i(\tau)$ a real number $T_v \in \R$ with the constraints that $T_v \le T_{\bar v}$ if $\bar v \le v$.
In our example, we have

\begin{equ}[e:exampletreelabel]
\mhpastefig{tree12label3} \in \TT^\3\;,
\end{equ}
provided that $r \le s$. We furthermore denote by $\mu_t^\tau$ the restriction of Lebesgue measure to the
subset $\TT^\tau_t$ of $\TT^\tau$ given by $\{T_\rho \le t\}$. With this notation, the example shown in 
\eref{e:exampletreelabel}
belongs to $\TT^\3_t$, provided that $s \le t$. A special case is given by $\tau = \bullet$ in which case
we set $\TT^\1_t = \TT^\1 = \{0\}$ and $\mu_t^\1 = \delta_0$. 

Denoting by $\hat \iota(\tau) = i(\tau) \setminus \{\rho\}$ the set of those interior vertices of $\tau$
that are not the root vertex, we define, for each $L \in \LL^\tau$, a stochastic process $Z^\tau_L$ 
on $\{(t,T) \in \R \times \TT^\tau\,:\, T \in \TT^\tau_t\}$ by
\begin{equ}[e:defZ]
Z^\tau_L(t,T) = e^{-\rho(L)^2 (t - T_{\rho})} \Bigl(\prod_{v \in \hat\iota(\tau)}i L_{e(v)} e^{-L_{e(v)}^2 \,|\delta T_{e(v)}|} \Bigr)\Bigl(\prod_{v \in \ell(\tau)} \bar X^\1_{L_{e(v)}}(T_{v_\downarrow})\Bigr)\;,
\end{equ}
where $v_\downarrow$ denotes the parent of $v$, $e(v)$ denotes the edge $(v, v_\downarrow)$, and 
$\delta T_e = T_v - T_u$ for an edge $e = (u,v)$.

Even though the tree $\bullet$ has an empty edge set, we set $\LL^\1 \sim \Z_\star$ by convention, by specifying $\rho(L)$
as an arbitrary value in $\Z_\star$.
With this convention in place, we also set 
\begin{equ}
Z^\1_L(t,T) =  X^\1_{\rho(L)}(t)\;.
\end{equ}
Finally, for $L \in \LL^\tau$, we write
\begin{equ}[e:defCeps]
C_\eps(L) \eqdef \prod_{v \in \ell(\tau)} \phi(\eps L_{e(v)})\;,
\end{equ}
with the additional convention $C_\eps(L) = \phi(\eps \rho(L))$ for $L \in \LL^\1$.

With all of these notations at hand, we then have the following identity:

\begin{proposition}\label{prop:decompX}
For every binary tree $\tau$ and every index $k \neq 0$, one has the identity
\begin{equ}[e:wanted]
 X^\tau_{\eps,k}(t) = \sum_{L \in \LL^\tau\atop \rho(L) = k} C_\eps(L) \int_{\TT^\tau_t} Z^\tau_L(t,T)\,\mu^\tau_t(dT)\;.
\end{equ}
\end{proposition}

\begin{proof}
We proceed by induction over the set of all binary trees, taking as induction parameter the number of leaves of $\tau$.
The identity is true by definition if $\tau = \bullet$. If $\tau \neq \bullet$, 
we can always write $\tau = [\kappa,\bar \kappa]$, where $\kappa$ and $\bar \kappa$ are trees
that have less leaves than $\tau$, so that
we assume that the identity \eref{e:wanted} holds true when $\tau$ is replaced by either 
$\kappa$ or $\bar \kappa$.

We then have the identity
\begin{equs}
 X^\tau_{\eps,k}(t) &= \sum_{\ell + m = k} \int_{s \le t}e^{-k^2|t-s|} \bar X^{\kappa}_{\eps,\ell}(s)\,\bar X^{\bar \kappa}_{\eps,m}(s)\,ds\\
&= \sum_{\ell + m = k}\sum_{L \in \LL^{\kappa}\atop \rho(L) = \ell}\sum_{\bar L \in \LL^{\bar \kappa}\atop \rho(\bar L) = m} C_\eps(L)C_\eps(\bar L) \\
&\quad \times \int_{s \le t}(i\ell)(im) e^{-k^2|t-s|}Z^\kappa_L(s,T) Z^{\bar\kappa}_{\bar L}(s,\bar T)\,\mu^{\kappa}_s(dT)\,\mu^{\bar \kappa}_s(d\bar T)\,ds\;,
\end{equs}
where we used the induction hypothesis and the fact that $k = \ell + m$ to go from the first to the second line.
Note now that one has the following simple facts:
\begin{claim}
\item For $\tau = [\kappa,\bar\kappa]$, there is a natural bijection $K\colon \LL^{\kappa}\times \LL^{\bar\kappa} \to \LL^\tau$ as follows.
Given $L \in\LL^{\kappa}$ and $\bar L\in \LL^{\bar\kappa}$,  
 one identifies the edges of $\kappa$ and $\bar \kappa$ with the corresponding subset of the edges of $\tau$ and uses the 
 labels $L$ and $\bar L$ to label them. One then labels the two edges connecting the root of $\tau$ to $\kappa$ and $\bar \kappa$ by $\rho(L)$ and
 $\rho(\bar L)$ respectively. Note that thanks to our convention for $\LL^\1$, this recipe also yields a bijection when one of the trees is the trivial tree.
\item For $\tau = [\kappa,\bar\kappa]$ and $s \in \R$, there is a natural map $\bar K_s\colon \TT^{\kappa}_s\times \TT^{\bar\kappa}_s \to \TT^\tau_s$
 obtained by associating $s$ to the root vertex of $\tau$, but otherwise leaving the labels of the interior vertices of $\kappa$ and $\bar \kappa$ untouched. Furthermore, one has the disintegration 
 \begin{equ}
\int_{\TT_t^\tau} F(T) \mu_t^\tau(dT) = \int_{-\infty}^t \int_{\TT_s^\kappa}\int_{\TT_s^{\bar \kappa}} F(\bar K_s(T, \bar T))\,\mu_s^\kappa(dT)\, \mu_s^{\bar \kappa}(d\bar T)\,ds\;,
 \end{equ}
 for every integrable function $F\colon \TT^\tau_t \to \R$.
\end{claim}
As a consequence, we can rewrite the desired identity \eref{e:wanted} as
\begin{equs}
 X^\tau_{\eps,k}(t) &= \sum_{\ell + m = k}\sum_{L \in \LL^{\kappa}\atop \rho(L) = \ell}\sum_{\bar L \in \LL^{\bar \kappa}\atop \rho(\bar L) = m} C_\eps(K(L,\bar L))\\
&\quad\times \int_{s \le t}Z^\kappa_{K(L,\bar L)}(\bar K_s(T,\bar T))\,\mu^{\kappa}_s(dT)\,\mu^{\bar \kappa}_s(d\bar T)\,ds\;.
\end{equs}
However, it follows from the definition \eref{e:defZ} of $Z$ and from the definition of the isometry $K$ that one has
the identity 
\begin{equ}
Z^\tau_{K(L,\bar L)}(t, \bar K_s(T,\bar T)) = ike^{-k^2|t-s|}Z^\kappa_L(s,T) Z^{\bar\kappa}_{\bar L}(s,\bar T)\;,
\end{equ}
whenever $\tau = [\kappa,\bar \kappa]$ and $L_\rho + \bar L_\rho = k$. Our conventions are set up
in such a way that this is true even if some of the trees involved are the trivial tree. Since
one has furthermore the identity $C_\eps(K(L,\bar L)) = C_\eps(L)C_\eps(\bar L)$, the claim follows.
\end{proof}

The computation of the correlations of $X_\eps^\tau$ is thus 
reduced to the computation of the correlations of $Z_L^\tau$, even though these then have to be integrated
over $\TT_t^\tau$ and summed over $L$, which is potentially no easy task. 

In order to compute correlations of polynomials of Gaussian random variables,
a useful notion is that of a \textit{pairing} of a set $T$ with $|T| \in 2\N$. We first denote the set
of all possible pairs of $T$ by $\CS_2(T) \eqdef \{A\subset T\,:\, |A| = 2\}$. With this notation,
the set of all pairings of $T$ is given by 
\begin{equ}
\CP(T) \eqdef \Bigl\{P \subset \CS_2(T)\,:\, \bigcup P = T \quad\&\quad p\cap q = \emptyset\;\forall p\neq q \in P\Bigr\}\;.
\end{equ}
In other words, $\CP(T)$ consists of all partitions of $T$ that are made up of pairs. By definition, 
$\CP(T) = \emptyset$ whenever $|T|$ is odd.

Since we want to estimate second moments of the processes $Z_L^\tau$,
the relevant notion of pairing arising from Wick's theorem will be that of a pairing of 
two copies of the leaves of a binary tree $\tau$. We thus introduce the shorthand notation 
\begin{equ}
\CP^\tau = \CP(\ell(\tau)\sqcup \ell(\tau))\;.
\end{equ}
See \eref{e:somepairing} below for a graphical representation of an element of $\CP^\tau$ for $\tau = \bigtree$.

\begin{definition}
Given two
labellings $L, \bar L \in \LL^\tau$, we denote by $L \sqcup \bar L$ the map 
\begin{equ}
(L \sqcup \bar L)\colon \CE(\tau)\sqcup \CE(\tau) \to \Z\;,
\end{equ}
which restricts to $L$ (respectively $\bar L$) on the first (respectively second) copy of $\CE(\tau)$.
Here, $\CE(\tau)$ denotes the set of edges of the binary tree $\tau$.
\end{definition}

\begin{definition}
Given a pairing $P \in \CP^\tau$ and $L, \bar L \in \LL^\tau$, we say that $L \sqcup \bar L$ is adapted to 
the pairing $P$ if 
\begin{equ}[e:adapted]
(L \sqcup \bar L)_{e(u)} + (L \sqcup \bar L)_{e(v)} = 0\;,\qquad \forall \{u,v\} \in P\;.
\end{equ}
We denote by $\LL^\tau_P$ the set of all labellings of the form $L\sqcup \bar L$ 
that are adapted to $P$. 
\end{definition}

\begin{remark}\label{rem:derrhoL}
Since $\rho(L) = \sum_{v \in \ell(\tau)} L_{e(v)}$, it follows from the definition that one automatically has the identity
$\rho(L) + \rho(\bar L) = 0$ if $L \sqcup \bar L \in \LL^\tau_P$. As a consequence, for every $\hat L = L \sqcup \bar L \in \LL^\tau_P$,
the quantity $|\rho(\hat L)|$ is well-defined by $|\rho(\hat L)| = |\rho(L)| = |\rho(\bar L)|$.
\end{remark}

Given an inner node $u \in i(\tau)$, we denote by $D(u) = \{v \in \ell(\tau)\,:\, u \le v\}$ the set of its descendants.
In the case where two copies of a tree are considered, we extend this definition in the natural way.
A very important remark is the following:

\begin{lemma}\label{lem:allowedPairs}
One has $\LL^\tau_P \neq \emptyset$ if and only if, for every inner node
$u \in i(\tau) \sqcup i(\tau)$, there exists at least one pair $\{v,\bar v\} \in P$ such that $v \in D(u)$ and $\bar v \not \in D(u)$.
\end{lemma}

\begin{proof}
To see that the condition is necessary, we note that if it fails, there exists at least one inner node $u$ such that
all of its descendants are paired together. It then follows from \eref{e:adapted} and the definition of a proper
labelling that $L_{e(u)} = 0$, which is excluded.

To see sufficiency, we can construct $(L,\bar L)$ as follows. First, we order the pairs in $P$, so that each 
one is assigned a strictly positive integer $p$, and we label the $p$th pair by 
$(3^p, -3^p)$. The claim now follows form the fact that a sum of the form $\sum_{p=0}^n a_p 3^p$ with
$a_p \in \{-1,0,1\}$ vanishes if and only if all the $a_p$ vanish. (This can be seen  by
expressing the number $\sum_{p=0}^n (a_p+1) 3^p$ in basis $3$.)
\end{proof}

\begin{remark}\label{rem:equiv}
We can introduce an equivalence relation on nodes of $\tau$ by setting $u \sim v$ if and only if
$u_\downarrow = v_\downarrow$, i.e.\ if $u$ and $v$ share the same parent.
As a consequence of Lemma~\ref{lem:allowedPairs}, we then note that if $P\in \CP^\tau$ contains a pair $\{u,v\}$
with $u \sim v$, then $\LL^\tau_P=\emptyset$.
\end{remark}

The importance of knowing for which pairings $P$ one has $\LL^\tau_P \neq \emptyset$ is
illustrated by the following result:
\begin{lemma}\label{lem:pairCross}
Let $\tau$ be a binary tree and let $L, \bar L \in \LL^\tau$, $T\in \TT_t^\tau$ and $\bar T \in \TT_{\bar t}^\tau$. Then, $\E Z_L^\tau(t,T) Z_{\bar L}^\tau(\bar t,\bar T) \neq 0$ if and only if there exists at least one pairing $P \in \CP^\tau$ such that $L\sqcup \bar L \in \LL^\tau_P$.
\end{lemma}

\begin{proof}
It follows from \eref{e:defZ} that, up to a non-vanishing numerical factor (that still depends on $\tau$, $t$, $\bar t$, $T$ and $\bar T$ but is not random),
one has
\begin{equ}
Z_L^\tau(t,T) Z_{\bar L}^\tau(\bar t,\bar T) \propto \prod_{u,v \in \ell(\tau)} \bar X^\1_{L_{e(u)}}(T_{u_\downarrow})\bar X^\1_{\bar L_{e(v)}}(\bar T_{v_\downarrow})\;.
\end{equ}
Writing $\hat L = L \sqcup \bar L$ and similarly $\hat T = T \sqcup \bar T$,
it then follows from \eref{prop:momentGaussian} that
\begin{equ}[e:momentZ]
\E Z_L^\tau(t,T) Z_{\bar L}^\tau(\bar t, \bar T) \propto \sum_{P \in \CP^\tau} \prod_{\{u,v\} \in P} \E \bar X^\1_{\hat L_{e(u)}}(\hat T_{u_\downarrow})\bar X^\1_{\hat L_{e(v)}}(\hat T_{v_\downarrow})\;.
\end{equ}
This shows that the condition is necessary since, by \eref{e:covZ1}, the terms in this product are all non-vanishing
if and only if $\hat L \in \LL^\tau_P$.
Its sufficiency is then a consequence of the positivity
of \eref{e:covZ1}.
\end{proof}

We finally introduce a notion of isometry of a tree $\tau$ that will be useful to identify terms that yield
identical contributions. 

\begin{definition}
Denote by $\CV(\tau) = i(\tau)\sqcup \ell(\tau)$ the set of all vertices of $\tau$.
A bijection $\sigma\colon \CV(\tau) \to \CV(\tau)$ is called an \textit{isometry} of $\tau$ if $v \sim u$ if and only if 
$ \sigma(u)\sim  \sigma(v)$, with ``$\sim$'' as in Remark~\ref{rem:equiv}. In other words, it is an isometry if it preserves ``family relations''.
We denote by $S_\tau$ the group of all isometries of $\tau$.
\end{definition}

\begin{remark}
Any isometry $\sigma$ extends in a natural way to $\CE(\tau)$  by
$\sigma(u,v) = (\sigma(u),\sigma(v))$, where the definition of an isometry ensures that the object on the right
is again an edge of $\tau$.
\end{remark}

\begin{remark}
The tree $\bigtree$ contains only one non-trivial isometry, which is the one that exchanges the two top leaves.
The tree $\nicetree$ on the other hand contains many more isometries since one can also exchange the two 
branches attached to the root for example.
\end{remark}

As a consequence, there are natural actions of $S_\tau$ on $\LL^\tau$ and $\TT^\tau$ by
\begin{equ}
\bigl(\sigma L\bigr)_e = L_{\sigma^{-1} e}\;,\qquad \bigl(\sigma T\bigr)(v) = T( \sigma^{-1} v)\;,
\end{equ}
for $L \in \LL^\tau$ and $T \in \TT^\tau$. 
There is also a natural action of $S_\tau \times S_\tau$ on $\CP^\tau$ by
\begin{equ}
\hat \sigma P = \{\{ \hat\sigma u, \hat \sigma v\}\,:\, \{u,v\} \in P\}\;,
\end{equ}
where we interpret elements in $S_\tau \times S_\tau$ as bijections of $\ell(\tau)\sqcup \ell(\tau)$.
Note that $\LL^\tau_P$ is covariant under this action in the sense that, for $\sigma, \bar \sigma \in S_\tau$, one has
\begin{equ}
L \sqcup \bar L \in \LL^\tau_P \quad\Leftrightarrow\quad (\sigma L) \sqcup (\bar \sigma \bar L) \in \LL^\tau_{(\sigma \times \bar \sigma)P}\;.
\end{equ}

For every binary tree $\tau$, every $P \in \CP^\tau$, and every $\hat L \in \LL^\tau_P$, we now define a quantity
$\CK^\tau(P,\hat L;\delta)$, which will be the basic building block for computing the correlations
of $X^\tau_\eps$, by
\begin{equs}
\CK^\tau(P,\hat L;\delta) &\eqdef \int_{\TT_0^\tau}\int_{\TT_{\delta}^\tau} e^{-\rho(\hat L)^2(\delta - T_\rho - \bar T_{\bar \rho})} \Bigl(\prod_{v \in \hat \iota(\tau)\sqcup \hat \iota(\tau)}\hat L_{e(v)} e^{-\hat L_{(v)}^2 (\hat T_{v_\downarrow} - \hat T_{v})} \Bigr)\\
&\qquad \times \Bigl(\prod_{\{u,v\} \in P} e^{-\hat L_{e(v)}^2 |\hat T_{u_\downarrow} - \hat T_{v_\downarrow}|} \Bigr) \,\mu_\delta^\tau(dT)\,\mu_{0}^\tau(d\bar T)\;, \label{e:defKernel}
\end{equs}
where we used the shorthand notation $\hat T = T \sqcup \bar T$ and where we denoted by $\rho$ and $\bar \rho$ the two copies of the root of $\tau$. For $v$ belonging to one of the two copies of the original
tree $\tau$, we again denote by $v_\downarrow$ its parent and by $e(v)$ the edge that connects it to its parent.
On the second line, we could of course have written $\hat L_{e(u)}$ instead of $\hat L_{e(v)}$ since, by the definition
of $\LL^\tau_P$, they only differ by a sign.

One then has the following fact:

\begin{lemma}\label{lem:isom}
For every $\hat \sigma \in S_\tau \times S_\tau$, one has $\CK^\tau(\hat \sigma P,\hat \sigma \hat L;\;\cdot\;) = \CK^\tau(P,\hat L;\;\cdot\;)$.
\end{lemma}

\begin{proof}
It suffices to notice that the integrand is preserved under isometries, provided that one also applies it to $T\sqcup \bar T$.
The claim now follows from the fact that isometries leave $\mu_t^\tau$ invariant.
\end{proof}

We are now almost ready to state the main result in this section. Before we do so however, we still need to
introduce one final notation. Given a tree $\tau$ and a pairing $P \in \CP^\tau$, we denote by $\ell_\ell^P(\tau)$
the set of leaves in $\ell(\tau)\sqcup \ell(\tau)$ with the  property that, for every $v \in \ell_\ell^P(\tau)$, there
exists a pair $\{u,\bar u\} \in P$ such that $v \in \{u,\bar u\}$ and such that the parent of $u$ is equal to the grandparent of $\bar u$. (In terms of genealogy, $\ell_\ell^P(\tau)$ contains all pairings 
between a nephew and his uncle.)
For example, in the following pairing of the tree $\bigtree$, the set $\ell_\ell^P(\tau)$ consists of exactly two leaves 
that are distinguished by being filled with white:
\begin{equ}[e:somepairing]
\mhpastefig{tree112pair4b}\;.
\end{equ}
Given any two labellings $L, \bar L \in \LL_P^\tau$,
we then write $L \sims \bar L$ if $|L_{e(v)}| = |\bar L_{e(v)}|$ for all $v \in \ell(\tau)\sqcup \ell(\tau)$ and furthermore
$L_{e(v)} = \bar L_{e(v)}$ for all $v \in \bigl(\ell(\tau)\sqcup \ell(\tau)\bigr) \setminus \ell_\ell^P(\tau)$.
In other words, $L$ and $\bar L$ are only allowed to differ by changing the signs of the labels 
adjacent to $\ell_\ell^P(\tau)$.
This allows to define a ``symmetrised'' kernel $\CK^\tau_\sym$ by
\begin{equ}
\CK_\sym^\tau(P, L; \delta) \eqdef {1\over |[L]_P|}\sum_{\bar L \sims L} \CK^\tau(P, \bar L; \delta)\;,
\end{equ}
where $[L]_P$ denotes the equivalence class of $L$ under $\sims$ and $|\cdot|$ denotes its cardinality.

With this final notation at hand, the main result of this section is the following:

\begin{theorem}\label{theo:abstractConv}
For a given $\tau \in \CT_2 \setminus \{\bullet\}$, if there exist $\alpha > 0$ and $\beta \in (0,1)$ such that, for every $P \in \CP^\tau/(S_\tau\times S_\tau)$
\minilab{e:boundKernel}
\begin{equs}
\sum_{\hat L \in \LL^\tau_P} \sup_{\delta\in \R} |\rho(\hat L)|^{2\alpha} |\CK^\tau_\sym(P,\hat L;\delta)| &< \infty \;, \label{e:boundKernel1}\\ 
\sum_{\hat L \in \LL^\tau_P} \sup_{|\delta| \le 1} {|\CK_\sym^\tau(P,\hat L;0) - \CK_\sym^\tau(P,\hat L;\delta)| \over \delta^{2\beta}} &< \infty\;,\label{e:boundKernel2}
\end{equs}
then the sequence of processes $X^\tau_\eps$ converges to a limit $X^\tau$ in probability in
$\CC([0,T],\CC^\gamma) \cap \CC^\delta([0,T],\CC)$, provided that $\gamma < \alpha$ and $\delta < \beta$.
\end{theorem}

\begin{proof}
This is a rather straightforward application of Proposition~\ref{prop:generalConvergence}.
Without loss of generality, we can assume that  $\alpha \le 1$, since otherwise it suffices to 
consider the appropriate derivative of $X^\tau_\eps$.
We first introduce an equivalence relation $\sim$ on $\LL^\tau$ by stipulating that 
$L \sim \bar L$ if and only if $|L_{e(v)}| = |\bar L_{e(v)}|$
for every $v \in \ell(\tau)$ and $\rho(L) = \rho(\bar L)$. We then define a symetrised family of processes $F_L$ by
\begin{equ}
F_L(t) = {1\over |[L]|} \sum_{\bar L \in [L]} \int_{\TT^\tau_t} Z_{\bar L}^\tau(T)\,\mu^\tau_t(dT)\;,
\end{equ}
where we denote by $[L]$ the equivalence class of $L$ under $\sim$. 
Writing furthermore $g_L(x) = e^{i\rho(L) x}$
and noting that both $C_\eps(L) = C_\eps(\bar L)$ and $g_L = g_{\bar L}$ for $L \sim \bar L$ by definition, 
it then follows from Proposition~\ref{prop:decompX} that
\begin{equ}
X^\tau_\eps(x,t) = \sum_{L \in \LL^\tau} C_\eps(L)\,F_L(t)\, g_L(x)\;,
\end{equ}
so that we are precisely in the framework considered in Proposition~\ref{prop:generalConvergence}, provided
that we set $\JJ = \LL^\tau$ for our index set.

With these notations at hand, it then
 follows from \eref{e:defZ}, Proposition~\ref{prop:momentGaussian}, and the definition of $\CK^\tau$ that
\begin{equ}[e:corX]
\E F_L(t)F_{L'}(t')
= {1 \over |[L]|\,|[L']|}\sum_{\bar L \in [L]\atop \bar L' \in [L']} \sum_{P \in \CP^\tau} \one_{\bar L \sqcup \bar L' \in \LL^\tau_P}  \CK^\tau(P,\bar L \sqcup \bar L';t-t')\;.
\end{equ}
Note now that if $\bar L \sqcup \bar L' \sims L \sqcup L'$ then, by the definition of $\sim$, one also
has $\bar L \sim L$ and $\bar L' \sim L'$. As a consequence, we can replace $\CK$ by $\CK_\sym$ in \eref{e:corX},
so that the claim follows Proposition~\ref{prop:generalConvergence}, noting that we can
restrict ourselves to equivalence classes of $\CP^\tau$ under isometries by Lemma~\ref{lem:isom}.
\end{proof}

\begin{remark}
Of course, since $\CK_\sym^\tau$ is constructed from a finite number of copies of $\CK^\tau$, we 
also have the same criterion with $\CK^\tau$ instead. However, it turns out that in some of the
situations that we are lead to consider, $\CK^\tau$ fails to satisfy \eref{e:boundKernel}, while $\CK^\tau_\sym$ 
does, due to some cancellations.
\end{remark}

\subsection{Reduction to simpler trees}

There is one situation in which the estimate of $\CK^\tau_\sym(P,\cdot;\cdot)$ for one tree can benefit 
from bounds on a simpler tree. This is when we consider a tree $\bar \tau$ of the form $\bar \tau = [\tau,\bullet]$
and a pairing $\bar P$ consisting of pairing the two copies of $\tau$ according to some pairing $P \in \CP^\tau$
and then pairing the two remaining leaves. We denote by $\CP_s^{\bar \tau}$ the set of all such pairings.

In this case, we have:

\begin{proposition}\label{prop:simplePairing}
Let $\bar \tau$ and $\bar P \in \CP_s^{\bar \tau}$ be as above and assume that the bound \eref{e:boundKernel1}
holds for $\CK_\sym^\tau(P,\cdot;\cdot)$ with some $\alpha \ge 0$. Then, the bounds \eref{e:boundKernel}
hold for $\CK_\sym^{\bar \tau}(\bar P,\cdot;\cdot)$ with $\bar \alpha < {3\over 2} \wedge \bigl(\alpha + {1\over 2}\bigr)$ and $\bar \beta < {1 \wedge \bar \alpha\over 2}$.
\end{proposition}

\begin{proof}
Let $L \in \LL^\tau_P$ be a labelling with $\rho(L) = k$. Then, for every $m \in \Z_\star$, we can construct a
corresponding labelling $\bar L \in \LL^{\bar\tau}_{\bar P}$ with $\rho(\bar L) = m$ by labelling the two paired copies of $\tau$ according
to $L$, using the labels $(k,-k)$ for the edges joining the roots of the copies of $\tau$ to the roots of the copies of $\bar \tau$,
and assigning the labels $(m-k, k-m)$ to the two remaining edges.
With this notation, it follows from the definition of $\CK^\tau_\sym$ that we have the identity
\begin{equ}
\CK_\sym^{\bar \tau}(\bar P,\bar L;\delta) = k^2\int_{-\infty}^0\int_{-\infty}^\delta \CK_\sym^\tau(P,L,s-s') e^{-(k-m)^2|s-s'| - m^2(\delta-s-s')}\,ds\,ds'\;.
\end{equ}
In particular, it follows from Lemma~\ref{lem:diffKernel} that
\begin{equs}
\CK_\sym^{\bar \tau}(\bar P,\bar L;0) &\lesssim  {k^2\|\CK_\sym^{\tau}(P,L;\cdot)\|_\infty \over m^2 (k^2 + m^2)}\;, \\
|\CK_\sym^{\bar \tau}(\bar P,\bar L;\delta) - \CK_\sym^{\bar \tau}(\bar P,\bar L;0)| &\lesssim 
\bigl(1\wedge \delta m^2\bigr){k^2\|\CK_\sym^{\tau}(P,L;\cdot)\|_\infty \over m^2(k^2+m^2)}\;.
\end{equs}
For $\alpha \ge 1$, the numerators of these expressions are summable by assumption, so that the corresponding
bound holds. 
For $\alpha \le 1$, we use the bound
\begin{equ}
{k^2 \over m^2(k^2+m^2)} \le {k^{2\alpha} \over m^{2+2\alpha}}\;,
\end{equ}
so that the claim follows at once.
\end{proof}

\subsection{General summability criterion}
\label{sec:constructGraph}

In this section, we introduce a graphical criterion to verify whether
 $\CK_\sym^\tau$ satisfies the bounds  of Theorem~\ref{theo:abstractConv} for a given pairing $P \in \CP^\tau$.
Our criterion consists of two steps: in a first step, we associate to a given pair 
$(\tau,P)$ a family of weighted graphs $\GG_\kappa^\tau(P)$.
Elements of $\GG_\kappa^\tau(P)$ all share the same underlying graph and only differ by the weights given to their
edges. In a second step, we need to check that $\GG_\kappa^\tau(P)$ contains at least one element that can be reduced to
a loop-free graph by a certain reduction procedure.

The weighted graphs in $\GG_\kappa^\tau(P)$ are built in several steps in the following way.\\[.3em]

\noindent\textbf{1.~Construction of the underlying graph.} 
Informally, we build a graph $(\CG,\CE)$ by taking two disjoint copies of $\tau$, joining all the
vertices that belong to the same pair of $P$, as well as the two roots, and then erasing all
``superfluous'' vertices that only have two incoming edges.
We will also henceforth denote by $\CT \subset \CE$ the spanning tree given by 
the interior edges of the two copies of the original tree $\tau$, together with
the new edge $\bar e$ connecting the two roots. 

\begin{example}
For a typical pairing of the tree $\bigtree$, we obtain the following graph, where some oriented
pairing $P$ is depicted on the left (with two nodes belonging to $P$ if they are connected by a black arc)
and the corresponding oriented graph $(\CG,\CE)$ is depicted on the right, with the
spanning tree $\CT$ drawn in grey and the glued edges in black:
\begin{equ}[e:exampleTree]
\mhpastefig{tree112pair4}\qquad
\mhpastefig{tree112pair4abstract}
\end{equ}
The distinguished edge $\bar e$ joining
the two copies of the root vertex is drawn as a thicker grey line.
\end{example}

Formally, this can be achieved by
setting $\CG = (\CV(\tau)\sqcup \CV(\tau))/\simp$ and 
\begin{equ}
\CE = \{(\rho,\bar \rho)\} \cup (\CE(\tau)\sqcup \CE(\tau)) / \simp\;,
\end{equ}
where $\CV(\tau)$ and $\CE(\tau)$ are the vertex (resp.\ edge) set of $\tau$ and $\rho$ and $\bar \rho$
are the two copies of the root. We will henceforth use the shorthand $\bar e = (\rho,\bar \rho)$ for the distinguished
edge connecting the two roots, as this will sometimes play a special role.

Here, the equivalence relation $\simp$ is defined on $\CV(\tau)\sqcup \CV(\tau)$ 
by setting $u \simp v_\downarrow$ and $v \simp u_\downarrow$ for every pair $(u,v)\in P$. 
This then induces a natural equivalence relation on the edge set by $(u,u_\downarrow) \simp (v,v_\downarrow)$.
Since $\tau$ is a binary tree, every vertex of the graph $(\CG,\CE)$ constructed in this way is of degree exactly $3$.
It inherits the ordering of $\tau$, but this ordering does not
extend to the edges ``glued'' by $\simp$, since they are always glued in ``opposite directions''. However,
if we order the pairs in $P$, then this naturally defines an ordering on all of $\CE$.\\[.3em]

\noindent\textbf{2.~Temporarily weigh edges.} Build a weighting $\bar \CL_0 \colon \CE \to \R$ of
the graph by
giving the weight $\kappa$ to the edge $\bar e$ connecting the two roots, the weight $-1$
to the remaining edges in $\CT$, and the weight $0$ to the remaining edges in $\CE \setminus \CT$.\\[.3em]

\noindent\textbf{3.~Treat small loops.}
It turns out that occurrences of certain small loops (a pair of vertices connected by two edges)
cause summability problems that have to be cured by
a special procedure.

There are two types of such loops: either one of its edges belongs to $\CT$ (let's call these ``type $1$''),
or both of its edges belong to $\CE \setminus \CT$ (``type $2$'').
Loops of the first kind are the only ``dangerous'' ones, and they are handled by the following special procedure. 
For each such loop, we shift a weight ${1\over 3}$ into the loop from one of its adjacent 
edges. More precisely, we build a weight $\CL_0$ from $\bar \CL_0$
by performing the substitution
\begin{equ}
\mhpastefig{loopType1}
\end{equ}
Note that the small loop appearing in the example \eref{e:exampleTree} is of  type 1.
It does not matter which one of the two adjacent edges we shift the weights to.\\[.3em]

\noindent\textbf{4.~Finalise the weights of the edges.}
We now finally construct the family $\GG_\kappa^\tau$ of weightings of the graph $(\CG,\CE)$. 
Denote by $\CL_0\colon \CE\to\R$ the weighting obtained
at the end of the previous step and denote by $\CW \colon \CG \to \R_+$ the map defined by
$\CW(v) = 0$ for $v \in \bar e$ and $\CW(v) = 2$ otherwise. In other words, we associate a weight $2$
to every vertex except the two roots.

Define now $\CE_0 = \emptyset$ and $\CG_0 = \emptyset$ and
recursively construct subsets $\CE_n\subset  \CE$ and $\CG_n\subset  \CG$ in the following way.
Assuming that $\CE_n$ and $\CG_n$ have already been constructed and that $\CL_n$ has been defined, 
we pick an arbitrary 
vertex $v \in \CG \setminus \CG_n$ and consider the set $E_v^n$ of edges attached to $v$
that are not in $\CE_n$. We then choose an arbitrary function $W_n \colon E_v^n \to \R_+$ with 
$\sum_{e \in E_v^n} W_n(e) = \CW(v)$ and we set $\CL_{n+1} = \CL_n + W_n$, $\CE_{n+1} = \CE_n \cup E_v^n$,
$\CG_{n+1} = \CG_n \cup \{v\}$. The construction terminates when $\CE_n =  \CE$,
and we denote the weight constructed in this way by $\CL$.

Loosely speaking, we distribute the weights ``$2$'' given by $\CW$ 
among each vertex's neighbouring edges, with the constraint that once the weight of a given vertex
has been distributed, none of its adjacent edges can receive any more weight from its other vertex.

\begin{definition}\label{def:labels}
The set $\GG_\kappa^\tau(P)$ consists of all the possible weightings
$\CL\colon \CE \to \R$ that can be obtained from the procedure outlined above
and that are such that $\CL(e) > 0$ for all edges $e$ belonging to a small loop.
\end{definition}

\begin{remark}
Since it will usually be advantageous to have only positive weights left, and since the elements in the 
original spanning tree $\CT$ have weight $-1$ before the vertex weights are distributed, 
it is usually be a good idea in Step~4
to traverse vertices in $\CG$ in a way that respects the ordering of $\CT$, i.e.\ from the outside 
of $\CT$ towards the two root vertices. 
\end{remark}

The point of this construction is that it is quite straightforward to obtain a bound on
$\CK^\tau_\sym(P,\cdot,\cdot)$ from the weights in $\GG_\kappa^\tau(P)$, but understanding why
this is so requires a few notions of elementary graph theory, which we now present. 

Given an arbitrary directed graph $(\CG,\CE)$, we write $e \to v$ for an edge $e$ entering a vertex $v$ 
(i.e.\ $e = (u,v)$ for some $u\in\CG$) and
$e \leftarrow v$ for an edge $e$ exiting $v$ (i.e.\ $e = (v,u)$ for some $u\in\CG$).
With this notation, the \textit{integral cycle group} $\cC(\CG,\CE)$ of a graph is given by all labellings
$L \colon \CE \to \Z$ such that, for each vertex $v \in \CG$, Kirchhoff's law is satisfied in the sense that
\begin{equ}
\sum_{e \to v} L_e = \sum_{e \leftarrow v} L_e\;.
\end{equ}
Note that even though we used the fact that we specified an orientation to define $\cC(\CG,\CE)$, 
it really does not depend on it. Indeed, if we consider two different orientations on the same graph,
we should identify elements in their respective cycle groups if they agree on those edges that do not
change orientation and have opposite signs on those edges that do change orientation. 
We also introduce a notation for the set of nowhere vanishing elements of the cycle group:
\begin{equ}
\cC_\star(\CG,\CE) = \{L \in \cC(\CG,\CE) \,:\, L_e \neq 0\;\forall\; e \in \CE\}\;.
\end{equ}
The reason why we introduced this notation is the following fact:

\begin{lemma}\label{lem:identif}
There is a canonical identification of $\LL^\tau_P$ with $\cC_\star(\CG,\CE)$, 
where $(\CG,\CE)$ is the graph associated to $\tau$ and $P$ as in Step~1 above.
\end{lemma}

\begin{proof}
Elements $L \in \LL^\tau_P$ are defined on $\CE(\tau)\sqcup \CE(\tau)$ with the canonical 
orientation that goes from the leaves to the roots, and they do satisfy Kirchhoff's law there. 
Furthermore, for any two edges $e$, $e'$ that are identified under $\simp$, one has
$L_e = -L_{e'}$. As a consequence, a choice of orientation on $P$ corresponds to a choice
of representative in each equivalence class of $\CE(\tau)\sqcup \CE(\tau)$ under $\simp$.
Denoting this choice of representative by $\pi$, we then define an element 
$C\in\cC_\star(\CG,\CE)$ by $C_e = L_{\pi(e)}$ for $e \neq \bar e$ and $C_{\bar e} = \pm |\rho(L)|$,
with the sign determined in such a way that Kirchhoff's law also holds at the roots.
\end{proof}

As an abelian group, $\cC(\CG,\CE)$ is isomorphic to $\Z^n$ for some $n\ge 0$, called the 
\textit{dimension} of $\cC$. An \textit{integral basis} $\CB \subset\cC(\CG,\CE)$ then consists of exactly $n$
elements, with the property that every element of $\cC(\CG,\CE)$ can be written uniquely as
\begin{equ}
L = \sum_{B \in \CB} L_B\,B\;,\qquad L_B \in \Z\;.
\end{equ}
In other words, an integral basis provides a decomposition that realises the isomorphism with $\Z^n$.

We call an element $L$ of the cycle group \textit{elementary} if there exists a simple cycle of $\CG$ 
(i.e.\ one traversing each edge at most once) such
that $L$ takes the values $\pm 1$ on the edges belonging to the cycle and $0$ otherwise.
There are exactly two elementary elements in $\cC(\CG,\CE)$ for each simple cycle, one for each orientation.
Finally, for any spanning tree $\CT$ of a graph $(\CG,\CE)$, we can construct a collection $\CB_\CT$ 
of elementary cycles by considering, for each edge $e \in \CE \setminus \CT$, the unique (modulo orientation) 
cycle passing through
$e$ that otherwise only traverses edges in $\CT$.
One then has the following classical result, which can be found for example
in \cite{Graphs}:

\begin{proposition}\label{prop:cycleBasis}
For each spanning tree $\CT$ of $(\CG,\CE)$, the collection $\CB_\CT$ forms an integral basis of $\cC(\CG,\CE)$.
\end{proposition}

Let now $(\CG,\CE,\CL)$ be a weighted graph, i.e.\ $\CL$ is a real-valued function over the edge set
$\CE$. We then introduce the following definition:

\begin{definition}
A weighted graph $(\CG,\CE,\CL)$ is \textit{summable} if
\begin{equ}
\sum_{L\in \cC_\star (\CG,\CE)} \prod_{e\in \CE} |L_e|^{-\CL_e} < \infty\;.
\end{equ}
\end{definition}

With all of these definitions in place, we are now finally ready to state the
criterion for the bounds on $\CK_\sym$ announced earlier.

\begin{theorem}\label{theo:graphKernel}
Let $\GG_\kappa^\tau(P)$ be as above and let $\kappa \in (0,2)$. 
If there exists an element of $\GG_\kappa^\tau(P)$ that is summable,  
then the kernel $\CK_\sym^\tau(P,\cdot)$ satisfies the bounds \eref{e:boundKernel} with $\alpha = 2 - {\kappa \over 2}$
and $\beta <{1\wedge \alpha \over 2}$.
\end{theorem}

\begin{proof}
It follows from the definition \eref{e:defKernel} that one can rewrite $\CK_\sym^\tau$ 
in a natural way as
\begin{equ}
\CK_\sym^\tau(P,L;\delta) = \int_{-\infty}^0\int_{-\infty}^\delta e^{L_\rho^2(\delta - s - s')} \CF^\tau_\sym(P,L;s-s')\,ds\,ds'\;.
\end{equ}
Here, the fact that $\CF^\tau_\sym$ only depends on the difference between $s$ and $s'$ is a consequence of the 
invariance of the integrand in \eref{e:defKernel} under translations, but this is not relevant.
Both claimed bounds then follow at once from Lemma~\ref{lem:diffKernel} if we are able to show that the
constants $\CF^\tau_\sym(P,L) \eqdef \sup_\delta |\CF^\tau_\sym(P,L;\delta)|$ satisfy the summability condition
\begin{equ}[e:wantedK]
\sum_{L\in \cC_\star(\CG,\CE)} |L_{\bar e}|^{-\kappa}\CF^\tau_\sym(P,L) < \infty\;,
\end{equ}
where we made the slight abuse of notation of considering $L$ as an element in $\cC_\star(\CG,\CE)$ instead
of $\LL^\tau_P$, which is justified by Lemma~\ref{lem:identif}.

Denote now $\bar \CE = \CE \setminus \{\bar e\}$ and $\bar \CT = \CT \setminus \{\bar e\}$.
It then follows from the expression \eref{e:defKernel} that $\CF^\tau_\sym(P,L,\eta)$ can be written as
\begin{equs}
\CF^\tau_\sym(P,L;\eta) &=  {1\over |[L]_P|}\sum_{\bar L \sims L} \Bigl(\prod_{e \in \bar \CT} \bar L_{e}\Bigr) \int 
\exp \Bigl(- \sum_{e\in \bar \CE} |\bar L_e|^2 |\delta T_e| \Bigr)\,\mu_\eta(dT)\\
&\eqdef \int \CG^\tau_P(L) \,\mu_\eta(dT)\;.\label{e:defG}
\end{equs}
Here, $\mu_\eta$ is the measure on $\R^{\CG}$ which fixes the vertices adjacent to $\bar e$ to 
$0$ and $\eta$ respectively, and is given by Lebesgue measure, restricted to $\TT^\tau_0\times \TT_\eta^\tau$,
for the remaining components. As before, $\delta T_e = T_v - T_u$ for any edge
$e = (u,v)$.

Consider now the case when the graph contains loops of type $1$. 
By Proposition~\ref{prop:cycleBasis}, we can assume without loss of generality that the 
graph $(\CG,\CE)$, the cycle $L$,
and the collection of ``times'' $T$
are locally given by the configuration
\begin{equ}[e:configG]
\mhpastefig{looplabel}
\end{equ}
with $r \le s$, and $k,m \in \Z_\star$ with $k+m \neq 0$. The left edge necessarily belongs to $\CT$, 
but the right edge could be either in $\CT$ or not.
With this notation, let us write $M$ for the subset of $\CE$ containing the two edges that 
form the loop under consideration. It then follows from \eref{e:defG} that $\CG^\tau_P$ can be factored as
\begin{equ}
\CG^\tau_P(L) = {1\over 2}\CJ_{k,m}(s-r)\, \CG^\tau_{P,M}(L)\;,
\end{equ}
where he prefactor $\CJ$ is given by
\begin{equ}
\CJ_{k,m}(s-r) \eqdef (k+m)e^{- ((k+m)^2+m^2) |s-r|} + (k-m)e^{-((k-m)^2+m^2) |s-r|}\;,
\end{equ}
and the remainder $\CG^\tau_{P,M}$ is given by
\begin{equ}
\CG^\tau_{P,M}(L) = {1\over |[L]_P|}\sum_{\bar L \sims L} \Bigl(\prod_{e \in \bar \CT\setminus M} \bar L_{e}\Bigr)  
\exp \Bigl(- \sum_{e\in \bar \CE\setminus M} |\bar L_{(u,v)}|^2 |\delta T_e| \Bigr)\;.
\end{equ}
This follows from the fact that the configuration with $m$ replaced by $-m$ in \eref{e:configG} belongs to $[L]_P$ by the definition
of $\sims$, and that the factor $\CG^\tau_{P,M}(L)$ does not depend on $m$.
We then have the bound

\begin{lemma}\label{lem:normaliseLoop}
For every $\eps \in [0,1]$, there exist constants $c$ and $C$ such that the bound
\begin{equ}
|\CJ_{k,m}(\delta)| \lesssim e^{ -c(m^2+(k+m)^2) \delta } |k|^\eps \bigl(|m| + |k+m|\bigr)^{1-\eps}\;,
\end{equ}
holds for all $k,m \in \Z_\star$ with $|k| \neq |m|$ and for all $\delta > 0$.
\end{lemma}

\begin{proof}
Using the identity
\begin{equ}
ac + bd = {1\over 2} \bigl((a+b)(c+d) + (a-b)(c-d)\bigr)\;,
\end{equ}
and the fact that $|e^{-x} - e^{-y}| \le (1\wedge |x-y|) e^{-(x\wedge y)}$, we obtain the bound
\begin{equ}
|\CJ_{k,m} | \le e^{ -(k^2 + 2m^2-2|km|) |s-r|} \bigl(2|k| + |m|(1\wedge 4|km| |s-r|)\bigr)\;.
\end{equ}
At this stage, we make use of the fact that 
$\sup_{x> 0} xe^{-ax} \le 1/a$ and that
there exists a constant $c> {1\over 3}$ such that 
$k^2+2m^2 - 2|km| > c(k^2 + m^2)$. This implies that 
\begin{equ}
 |s-r| e^{ -(k^2 + 2m^2-2|km|) |s-r|} \lesssim {e^{-{1\over 3}(k^2+m^2)|s-r|}  \over k^2 + m^2}\;,
\end{equ}
so that we conclude that the bound
\begin{equs}
|\CJ_{k,m} | &\lesssim e^{ -c(m^2+(k+m)^2) |s-r|} \Bigl(|k| + |m|\Bigl(1\wedge {|km| \over k^2 + m^2}\Bigr)\Bigr)\\
&\lesssim e^{ -c(m^2+(k+m)^2) |s-r|} \bigl(|k| + |m|^{1-\eps} |k|^\eps\bigr)\;,
\end{equs}
holds for every $\eps \in [0,1]$. This bound is equivalent to the one in the statement. 
\end{proof}

For any $L \in \cC_\star(\CG,\CE)$, denote now by $\CS L\colon \CE \to \R$ the function 
given by $\CS L_e = |L_e| + |L_{e'}|$ if the two edges $e$ and $e'$ are part of a loop
of type $1$, and $\CS L_e = |L_e|$ otherwise. The above considerations
show that with this notation, there exists a constant $c>0$ such that one then has the bound 
\begin{equs}\label{e:boundFT}
\CF^\tau_\sym(P,L;\eta) &\lesssim \int \Bigl(\prod_{e \in \bar \CE} |\CS L_{e}|^{-\CL_0(e)} e^{-c (\CS L)_e^2|\delta T_e|} \Bigr)\, \mu_\eta(dT)\\
&\lesssim \int_{\R^{\bar \CG}} \Bigl(\prod_{e \in \bar \CE} |\CS L_{e}|^{-\CL_0(e)} e^{-c (\CS L)_e^2|\delta T_e|} \Bigr)\, \prod_{u \in \bar \CG} dT_u\;,
\end{equs}
where $\CL_0$ is the weighting constructed in Step~3. Here, the passage from the first to the second line is trivial
since the integrand is positive by construction, so that integrating over a larger domain can only increase the value of the integral.
(Here, we use the convention that $T_u = 0$ or $\eta$ respectively for $u \in \bar e$.)

To conclude the proof, we note that a repeated application of H\"older's inequality yields the bound
\begin{equ}[e:boundHolder]
\int_\R \exp \Bigl(- \sum_{j=1}^n a_j |x - x_j|\Bigr) \lesssim \prod_{j=1}^m a_j^{-\ell_j}\;,
\end{equ}
for any $a_1,\ldots, a_n \in \R$ and any exponents $\ell_j > 0$ with $\sum_{j=1}^n \ell_j = 1$.
We now fix an arbitrary order on $\bar \CG$ and we apply \eref{e:boundHolder} repeatedly, every time integrating
over the time variable associated to the corresponding element of $\bar \CG$. 

Each such integration corresponds exactly to one iteration of Step~4 of the
construction of $\CL\in \GG_\kappa^\tau$.
This shows that, for any of the weights $\CL\in \GG_\kappa^\tau$, one has indeed the bound
\begin{equ}
\CF^\tau_\sym(P,L) \le \prod_{e \in \bar \CE} |\CS L_{e}|^{-\CL(e)}\;.
\end{equ}
Since $|\CS L_e| \ge |L_e|$ and since we only retain weights such that $\CL(e) > 0$ 
for $e$ belonging to a loop, the claim then follows.
\end{proof}

\subsection{On the summability of graphs}

As a consequence of the results in the previous subsection, 
the construction of $X^\tau$ is now reduced to verifying the existence 
of a summable graph in $\GG_\kappa^\tau(P)$ for every pairing $P \in \CP^\tau$.
It is therefore useful to have a simple criterion to check the summability of a graph.
This is achieved by the following algorithm:

\begin{algorithm}\label{algo}
Apply the following operations successively, until the procedure stabilises.
Edges with weight $0$ are removed and consecutive edges without intermediate 
branching point are merged:
\minilab{e:merge}
\begin{equs}
\mhpastefig{mergeedge}\label{e:merge1}\\[.5em]
\mhpastefig{eraseedge}\label{e:merge2}
\end{equs}
Simple loops are erased, provided that their total weight 
is strictly greater than $1$:
\begin{equ}[e:loop]
\mhpastefig{eraseloop}
\end{equ}
Small loops are ``flattened'', provided that their weights $\alpha$ and $\beta$ add to a value strictly greater than $1$:
\begin{equ}[e:cycle]
\mhpastefig{mergeloop}
\end{equ}
with $\gamma = \alpha + \beta - 1$ if $\alpha\vee \beta < 1$,
$\gamma = \alpha \wedge \beta$ if $\alpha \vee \beta > 1$,
and $\gamma < \alpha \wedge \beta$ if $\alpha \vee \beta = 1$.
\end{algorithm}

The main result of this subsection is then the following.

\begin{proposition}\label{prop:criterionSum}
Let $(\CG, \CE,\CL)$ be a weighted graph such that the application of Algorithm~\ref{algo} yields
a loop-free graph. Then, $(\CG, \CE,\CL)$ is summable.
\end{proposition}

\begin{proof}
Since, for a loop-free graph, $\cC(\CG,\CE)$ contains only one element (the one that associates $0$ to every edge),
it suffices to check that for each step of the algorithm, we can show that the original graph is summable, provided that the
simplified graph is summable.

We define a function $F \colon \cC (\CG,\CE) \to \R$ by
\begin{equ}[e:defF]
F(C) = \prod_{e\in \CE} (1 \vee |C_e|)^{-\CL_e}\;,
\end{equ}
so that we want to verify the summability of $F$.
Consider now one of the steps of the algorithm and denote by $(\CG,\CE,\CL)$ the graph before the
step and by $(\bar \CG, \bar \CE,\bar \CL)$ the graph after the step. Similarly, we denote by
$\bar F$ the function associated as in \eref{e:defF} to the weighted graph $(\bar\CG, \bar \CE, \bar \CL)$. 
Again, orientations do not matter,
but it is convenient to fix an orientation for the sake of definiteness. We will therefore assume from now
on that all the edges appearing in \eref{e:merge} and \eref{e:cycle} are oriented from left to right.

For each of the operations appearing in Algorithm~\ref{algo}, there is an obvious projection
operator
\begin{equ}
\Pi \colon  \cC (\CG,\CE) \to  \cC (\bar \CG,\bar \CE) \;.
\end{equ}
For those edges unaffected by the merging / erasing operation, we identify $\Pi C$ with $C$ in the obvious way.
In the case of the  merging operation \eref{e:merge1}, if we denote by $f$ and $f'$ the two edges being merged
and by $\bar f$ the resulting edge in $\bar \CE$, we set $(\Pi C)_{\bar f} = C_f = C_{f'}$. 
In the case of \eref{e:merge2} and \eref{e:loop} there is nothing to do since $\bar \CE$ is identified with a subset of $\CE$.
In the case \eref{e:cycle}, denoting by $f$, $f'$ and $\bar f$ the old and new edges as before,
we set $(\Pi C)_{\bar f} = C_f + C_{f'}$. (In all cases, the identification is very natural if we think of elements in 
$\cC (\CG,\CE)$ as describing flows on the graph. It is also clear that $\Pi C$ then again describes a flow on the new graph.)

For the first two operations, the preservation of summability is now obvious, since $\Pi$ is a bijection and one
has the identity
\begin{equ}
\bar F(\Pi C) = F(C)\;.
\end{equ}
For the  operation \eref{e:loop}, observe that, denoting the flow in the loop by $k$, one has the identity
\begin{equ}
\sum_{\cC (\CG,\CE)} F(C) =  \sum_{\bar C \in \cC (\bar \CG,\bar \CE)} \bar F(\bar C) {\sum_{k \in \Z}} (1\vee|k|)^{-\alpha}\;.
\end{equ}
Therefore, since $\alpha > 1$ by assumption, it does follow that the summability of $F$ implies that of $\bar F$.

Finally, for the  operation \eref{e:cycle}, denote by $C_0$ the elementary cycle going through 
the two edges that are being merged so that any two elements in $\Pi^{-1} \bar C$ differ
by an integer multiple of $C_0$.
With this notation, one then has the identity
\begin{equs}
\sum_{\cC (\CG,\CE)} F(C) &= \sum_{\bar C \in \cC (\bar \CG,\bar \CE)} \sum_{C\in \Pi^{-1}\bar C} F(C)\\
&= \sum_{\bar C \in \cC (\bar \CG,\bar \CE)} \bar F(\bar C) {\sum_{k \in \Z}} {(1\vee |\bar C_{\bar f}|)^\gamma \over (1\vee|k|)^{\alpha} (1\vee |\bar C_{\bar f}-k|)^{\beta}}\;,
\end{equs}
where as before $\bar f$ is the new edge replacing the loop.
It is straightforward to check that the conditions on $\alpha$, $\beta$, and $\gamma$ given below \eref{e:merge2}
are precisely the conditions  guaranteeing that
\begin{equ}
\sup_{a \in \Z} {\sum_{k \in \Z}} {(1\vee |a|)^\gamma \over (1\vee|k|)^{\alpha} (1\vee|a-k|)^{\beta}} < \infty\;,
\end{equ}
so that the summability of $\bar F$ does indeed imply that of $F$, thus concluding the proof.
\end{proof}

\begin{remark}
It is clear from the proof that another allowed step would be to decrease the weight of any edge.
In particular, edges with positive weights can also be contracted to a node. However, we will always consider weights
in $\GG_\kappa^\tau$ such that this step is unnecessary.

One may legitimately ask whether the criterion given in Proposition~\ref{prop:criterionSum} is sharp.
This is not known to the author and is probably not the case, even though the author is not
aware of any counterexample. An obvious necessary condition for summability is that 
$\sum_{e \in \CC} \CL_e > 1$ for every elementary cycle $\CC\subset \CT$, but it is unfortunately
easy to construct counterexamples showing that this naive condition is not sufficient, even within the class of homogeneous graphs of degree $3$. 
(Take the tetrahedron and give each edge the same weight $\alpha \in ({1\over 3}, {1\over 2})$.)
\end{remark}

Before we proceed, we summarise the results of the preceding subsections in one convenient statement:
\begin{proposition}\label{prop:graphicalAlgorithm}
Let $\tau$ be a binary tree with at least two interior vertices.
Set $\kappa = 4-2\alpha$ and,
for any pairing $P \in \bar \CP^\tau$, denote by $(\CE,\CG)$ and $\GG_\kappa^\tau(P)$ the graph
and set of weightings constructed in Section~\ref{sec:constructGraph}. 
If, for a given $P\in \bar \CP^\tau$, there exists $\CL \in \GG_\kappa^\tau(P)$ such that the application of Algorithm~1
allows to contract the graph to a point, then the bounds \eref{e:boundKernel} do hold for $\CK_\sym^\tau(P,\cdot;\cdot)$.
\end{proposition}

\begin{proof}
This is just a combination of Theorem~\ref{theo:graphKernel} with 
Proposition~\ref{prop:criterionSum}.
\end{proof}

\subsection[Construction of $X^3$]{Construction of $X^\3$}

Were are now in a position to apply the abstract result of the previous two subsections to the construction of $X^\3$.
If we discard pairings such that $\LL_P^\tau$ is empty (by Remark~\ref{rem:equiv}, these are the pairings containing at least one of 
the two top pairs of leaves),
it can be checked by inspection that the remainder of $P \in \CP^\tau/(S_\tau\times S_\tau)$ for $\tau = \smalltree$ consists
of exactly three elements,
which can be represented graphically as follows:
\begin{equ}[e:pairings3]
\mhpastefig{tree12pair1}\;,\qquad
\mhpastefig{tree12pair3}\;,\qquad
\mhpastefig{tree12pair2}\;.
\end{equ}

\begin{proposition}\label{prop:boundX3}
For every $P \in \CP^\3$, the bounds \eref{e:boundKernel} hold for $\CK_\sym^\3(P,\cdot;\cdot)$ 
for every $\alpha < {3\over 2}$ and every $\beta < {1\over 2}$.

As a consequence, there exists a process $X^\3$ with sample paths that are almost surely continuous with values in $\CC^\alpha$ for every $\alpha < {3\over 2}$. Furthermore, $X^\3_\eps \to X^\3$ in probability in $\CC([-T,T],\CC^\alpha)\cap \CC^\beta([-T,T],\CC)$
for every $\beta < {1\over 2}$ and every $T>0$.
\end{proposition}

\begin{proof}
The second claim follows from Theorem~\ref{theo:abstractConv}, so that it suffices to check that the bounds \eref{e:boundKernel} hold 
for each of the pairings $P$ depicted in \eref{e:pairings3}.
The first pairing is treated by Proposition~\ref{prop:simplePairing}, noting that the required bounds on
$\CK_\sym^\2$ were already obtained in the proof of Proposition~\ref{prop:procX1}.

The second pairing is treated by Proposition~\ref{prop:graphicalAlgorithm}, noting that the following element
belongs to $\GG^\tau_{1+\delta}(P)$ for every $\delta > 0$:
\begin{equ}[e:pairingDangle]
\mhpastefig{tree12pair1abstract}
\end{equ}
It is straightforward to verify that Algorithm~\ref{algo} terminates and yields a loop-free graph.

Unfortunately, the last remaining pairing does not seem to be covered by Proposition~\ref{prop:graphicalAlgorithm},
so we need to treat it ``by hand''.
A generic labelling for this pairing looks like the following:
\begin{equ}[e:secondPairing]
\mhpastefig{tree12pair2label}
\end{equ}
The kernel $K^\tau$ associated to \eref{e:secondPairing} can then be written as
\begin{equ}[e:exprKernel]
K^\tau(P,\hat L; t'-t) = \CF(k,\ell,m) \int_{\TT_t^\tau}\int_{\TT_{t'}^\tau}\exp(-\CI_{k,\ell,m}(T,T'))\,\mu_t(dT)\,\mu_{t'}(dT')\;,
\end{equ}
where the prefactor $\CF$ is given by
\begin{equ}
\CF(k,\ell,m) = (k+\ell)(k+m)(k+\ell+m)^2\;,
\end{equ}
whereas the exponent $\CI$ is given by $\CI = \CI_1 + \CI_2$ with
\begin{equs}
\CI_1 &= k^2|r-r'| + (k+\ell)^2 (s-r) + (k+m)^2 (s'-r') + \ell^2 |s'-r| + m^2 |r'-s|\;,\\
\CI_2 &= (k+\ell+m)^2 (t+t'-s-s')\;.
\end{equs}
This time, we make use of Proposition~\ref{prop:boundExpGeneral} in order to bound the integral of $\CI_1$, which 
yields the bound
\begin{equ}
\int_{-\infty}^s \int_{-\infty}^{s'} e^{-\CI_1}\,dr'\,dr \lesssim {e^{-(\ell^2\wedge m^2)|s-s'|}\over \bigl((k+\ell)^2 + \ell^2\bigr)\bigl((k+m)^2 + m^2\bigr)}\;.
\end{equ}
As in the proof of Theorem~\ref{theo:graphKernel}, it follows from Lemma~\ref{lem:diffKernel} that, in 
order to verify the assumptions of Theorem~\ref{theo:abstractConv},
it suffices to verify the summability of
\begin{equ}
K_{k\ell m} \eqdef {|k+\ell| |k+m| (k+\ell+m)^{1-2\kappa} \over \bigl((k+\ell)^2+ \ell^2\bigr)\bigl((k+m)^2+ m^2\bigr) \bigl((k+\ell+m)^2 + (\ell^2 \wedge m^2)\bigr)}\;,
\end{equ}
for every $\kappa > 0$.
Since this expression is symmetric in $(\ell,m)$, it suffices to check summability over $|\ell| > |m|$, say.
In this case, one has the bound
\begin{equs}
K_{k\ell m} &\lesssim {1 \over \bigl(|k+\ell|+ |\ell|\bigr) \bigl(|k+m|+ |m|\bigr) \bigl(|k+\ell+m| + |m|)^{1+2\kappa}}\\
&\le {1 \over |\ell| |k+m|^{\kappa\over 2}|m|^{1+{\kappa\over 2}}  |k+\ell+m|^{1+\kappa}}\;.
\end{equs}
Since, for every $\kappa > 0$ and $a \in \Z_\star$,  one has the bound
$\sum_{\ell \not\in \{0,a\}} {1\over |\ell| |\ell-a|^{1+\kappa}} \lesssim {1\over |a|}$,
it follows that
\begin{equ}
\sum_{\ell} K_{k\ell m} \lesssim {1 \over |m|^{1+{\kappa\over 2}}  |k+m|^{1+{\kappa\over 2}}}\;,
\end{equ}
which is indeed summable in $k$ and $m$ for every $\kappa > 0$, and the claim follows.
\end{proof}

\subsection[Construction of $X^4$]{Construction of $X^\4$}

Since the tree $\mhpastefig[3/5]{tree22}$ has many symmetries, the calculations for this tree
turn out to be easier than for the previous case, even though this is a larger tree. 
Discarding pairings such that $\LL_P^\tau$ is empty,
it can again be checked by inspection that the remainder of 
$P \in \CP^\tau/(S_\tau\times S_\tau)$ consists of the following three elements:
\begin{equ}[e:pairings4]
\mhpastefig{tree22pair1}\;,\qquad \mhpastefig{tree22pair2}\;,\qquad 
\mhpastefig{tree22pair3}\;.
\end{equ}
We have the following result:

\begin{proposition}
There exists a process $X^\4$ with sample paths that are almost surely continuous with values in $\CC^\alpha$ for every $\alpha < 2$. Furthermore, $X^\4_\eps \to X^\4$ in probability in $\CC([-T,T],\CC^\alpha)\cap \CC^\beta([-T,T],\CC)$
for every $\beta < {1\over 2}$ and every $T>0$.
\end{proposition}

\begin{proof}
For the three pairings shown in \eref{e:pairings4}, the algorithm of Section~\ref{sec:constructGraph} allows to verify 
that the following elements belongs to  $\GG_\kappa^\tau$:
\begingroup\scriptsize
\begin{equ}
\mhpastefig[8/9]{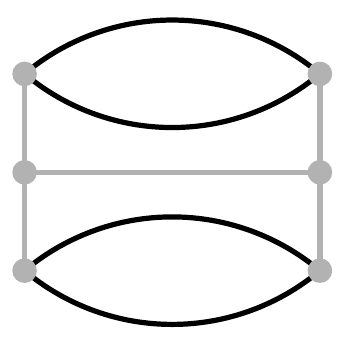}\qquad
\mhpastefig[8/9]{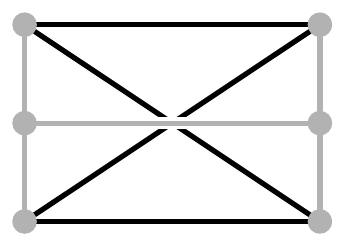}\qquad
\mhpastefig[8/9]{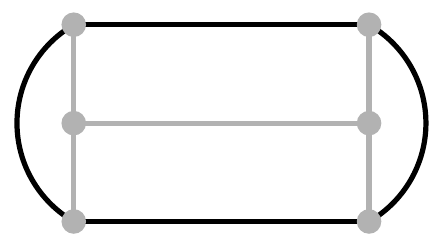}
\end{equ}\endgroup
In each case, we go through the vertices in Step~2 by ordering them from left to right and top to bottom.
For each of these three elements, it is then straightforward to verify that Algorithm~\ref{algo} indeed yields a loop-free graph
for every $\kappa > 0$.
\end{proof}

\subsection[Construction of $X^5$]{Construction of $X^\5$}

This time, because of the lack of symmetry of $\bigtree$,
there are many more cases to consider.
Indeed, if we discard again those pairings such that $\LL_P^\tau$ is empty,
it can be checked by inspection that the remainder of 
$P \in \CP^\tau/(S_\tau\times S_\tau)$ consists of $15$ elements.

However, with the tools of the previous subsections at hand, it turns out to be 
relatively straightforward to show that the following holds, where $\CP_s$ is as in Proposition~\ref{prop:simplePairing}.

\begin{proposition}\label{prop:boundK5}
For every $P \in \CP^\5 \setminus \CP_s^\5$ and for every $\kappa > 0$, one has
\begin{equ}
\sum_{L \in \LL_P^\5} |\rho(L)|^{-\kappa} \CF_\sym^\5(P,L) < \infty\;.
\end{equ}
\end{proposition}

\begin{proof}
We outline a systematic
way of constructing an element in $\GG_\kappa^\tau(P)$ for each pairing $P$. 
The spanning tree $\CT$ underlying the graphs $(\CG,\CE)$ associated to each pairing is given
by a row of five edges, which we represent as
\begin{equ}[e:skeleton]
\mhpastefig{skeleton}
\end{equ}
with the distinguished edge $\bar e$ being the one with weight $\kappa$. The remaining pairings now
consist of all possible graphs built by adding edges to \eref{e:skeleton} in such a way that
\begin{claim}
\item Every vertex is of degree exactly $3$, which is a reflection of the fact that we only consider binary trees.
\item There are no simple loops, i.e.\ loops of the kind \eref{e:loop}, for otherwise one would have $\LL_P^\tau = \emptyset$
by Remark~\ref{rem:equiv}.
\item There is at least one edge other than $\bar e$ connecting the left half of the graph to the right half, 
for otherwise one would have $\LL_P^\tau = \emptyset$ by Lemma~\ref{lem:allowedPairs}.
\item There is no edge other than $\bar e$ connecting the two vertices adjacent to $\bar e$, for otherwise one would have $P \in \CP_s^\tau$.
\end{claim}
By inspection, one can then check that, modulo isometries, the set of all such graphs consists of $12$ elements. We first consider the
following $10$ elements:
\begin{equs}
\mhpastefig[8/9]{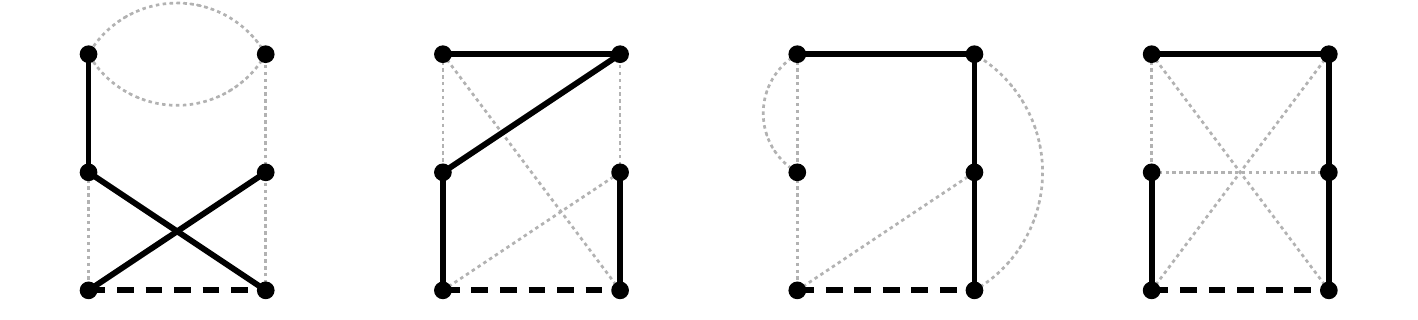}\\
\mhpastefig[8/9]{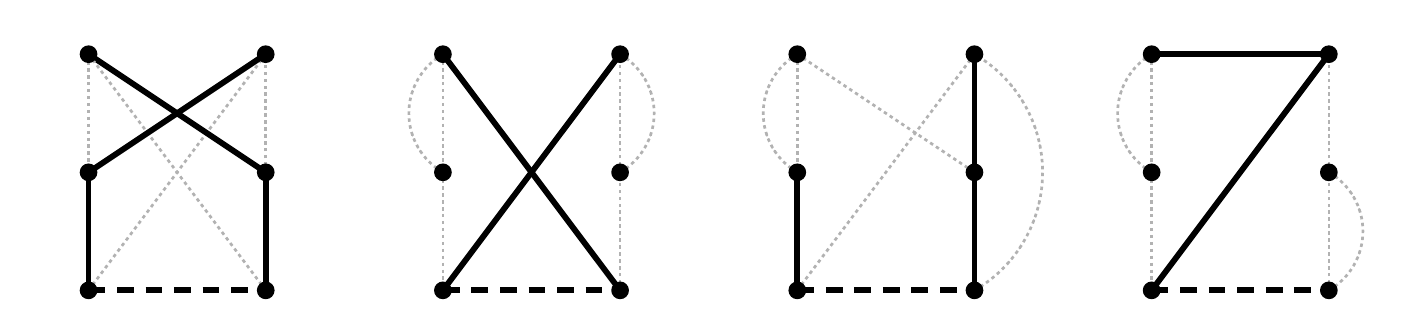}\\
\mhpastefig[8/9]{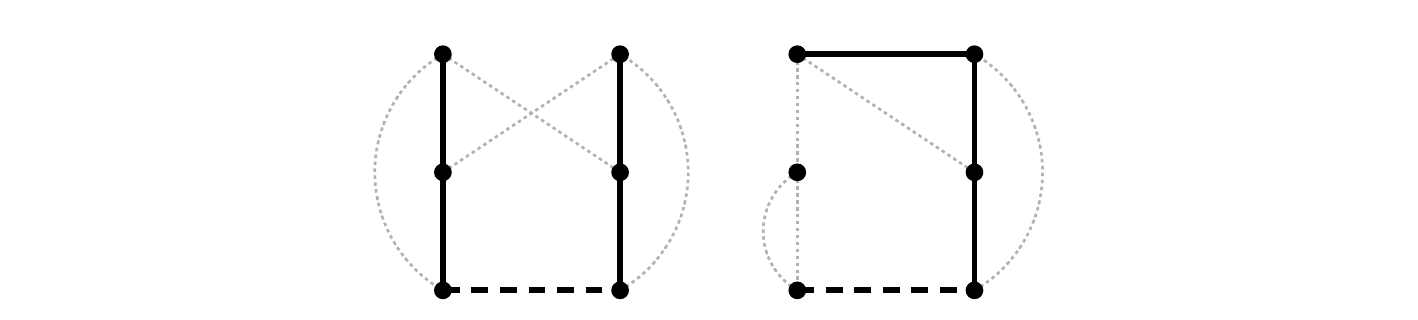} \label{e:biggraph}
\end{equs}
Observe that, for each one of the graphs $(\CG,\CE)$ appearing in this list, we have distinguished a
subset $\hat \CE \subset \CE$ of the edges by drawing it in boldface, and we have drawn the distinguished
edge $\bar e$ as a dashed line. 

For each pairing, the weighting represented by these pictures gives weight $1$ to edges in $\hat \CE$ and
weight $\kappa$ to $\bar e$. One should furthermore think of it as giving weight $0$ to the remaining
dotted edges. However, whenever a small loop (whatever its type) followed by a dotted edge 
appears in one of these graphs, we weigh such a configuration as follows:
\begin{equ}[e:looplabel]
\mhpastefig{looplabelsimple}
\end{equ}
We see that applying one step of Algorithm~\ref{algo} to such a configuration results in two consecutive edges
with weights $-{1\over 3}$ and ${1\over 3}$ respectively. As a consequence, by applying one more step of the algorithm, 
such a configuration does indeed behave 
for all practical purposes as
if the loop was erased and all edges had weight $0$. For each of the weightings represented in the figure, 
it is then a straightforward task to verify that, on the one hand they do belong to $\GG_\kappa^\tau(P)$ for every
$\kappa > 0$ and, on the other hand,
that Algorithm~\ref{algo} does indeed yield a loop-free graph. This is because, in every single case, contracting the
edges with weight $0$ yields a graph with only two vertices that are joined by a number of edges, where one
edge has weight $\kappa$ and every other edge has weight $1$.

The two remaining pairings require a slightly different weighting:
\begin{equ}
\mhpastefig{proofSpecial}
\end{equ}
In the first case, the two grey edges have weights ${\kappa\over 2}$ and $1-{\kappa\over 2}$ respectively, whereas the
weights for all the other edges are as before.
In the second case, the two grey edges have weights 
$-{2\over 3}$ and ${2\over 3}$ respectively, whereas the dotted edges all have weights ${2\over 3}$. 

Again, it can be checked in a straightforward way that in both cases, these weights 
do indeed belong to $\GG_\kappa^\tau(P)$
and that Algorithm~\ref{algo} yields a loop-free graph in both cases. 
\end{proof}

It follows almost immediately that $X^\5_\eps$ converges to a limit taking ``almost'' values in $\CC^{3\over 2}$. More precisely,

\begin{proposition}\label{prop:convergenceX5}
There exists a process $X^\5$ with sample paths that are almost surely continuous with values in $\CC^\alpha$ for every $\alpha < {3\over 2}$. Furthermore, $X^\5_\eps \to X^\5$ in probability in $\CC([0,T],\CC^\alpha)\cap \CC^\beta([0,T],\CC)$
for every $\beta < {1\over 2}$.
\end{proposition}

\begin{proof}
It follows from Proposition~\ref{prop:boundK5} and Theorem~\ref{theo:graphKernel} that the bound \eref{e:boundKernel}
holds for every $\alpha < 2$ and $\beta < {1\over2}$, provided that $P \in \CP^\5 \setminus \CP_s^\5$. 

In view of Theorem~\ref{theo:abstractConv}, it thus suffices to show that a similar bound, but this time with $\alpha < {3\over 2}$, 
holds for $P \in \CP_s^\5$. By Proposition~\ref{prop:simplePairing}, these pairings can be reduced to the study of the tree $\smalltree$. Applying 
Proposition~\ref{prop:boundX3} then concludes the proof.
\end{proof}

\section{Treatment of the constant Fourier mode}
\label{sec:constantMode}

So far, we were concerned with the construction of the processes $X^\tau$ defined as the limits
of the processes $X^\tau_\eps$ from \eref{e:defXtau}.
However, in order to define solutions to the original problem, we would really like to build a sequence of processes $Y^\tau$ 
given by the limit as $\eps \to 0$ of $Y^\tau_\eps$ defined recursively by
\begin{equ}[e:defYeps]
\d_t Y^\tau_\eps = \d_x^2 Y_\eps^\tau + \d_x Y^{\tau_1}_\eps\, \d_x Y^{\tau_2}_\eps - C^\tau_\eps\;,
\end{equ}
for $\tau = [\tau_1, \tau_2]$ 
and for constants $C^\tau_\eps$ as in \eref{e:defCtau}.
Here, we start the recursion by setting $Y^\1_\eps(t,x) = X^\1_\eps(t,x) + \sqrt2 B(t)$, for $B$ a standard Brownian motion
which is a solution to the additive stochastic heat equation.

Since only spatial derivatives appear in the right hand side of the recursion relation \eref{e:defYeps}, we see that
$\Pi_0^\perp Y^\tau_\eps = X^\tau_\eps$, so that it only remains to show that the constant Fourier modes
of $Y^\tau_\eps$ converge to a limiting real-valued stochastic process.
The proof goes in two steps. 
In a first step, we define a family of constants $K^\tau_\eps$ for $\tau = [\tau_1, \tau_2]$ by
\begin{equ}
K^\tau_\eps = \sum_{k \in \Z_\star} \E \bar X^{\tau_1}_{\eps,k}\, \bar X^{\tau_2}_{\eps,-k}\;,
\end{equ}
and we show that the following convergence result holds.

\begin{proposition}\label{prop:convConst}
For every $\tau \in \{\2,\3,\4,\5\}$, there exists a constant $\bar K^\tau$ independent of the mollifier $\phi$ such that
\begin{equ}
\lim_{\eps \to 0} \bigl(C^\tau_\eps - K^\tau_\eps\bigr) = \bar K^\tau\;.
\end{equ}
\end{proposition}

\begin{proof}
See Lemmas~\ref{lem:convK2}--\ref{lem:convK5} below, noting that one has $K^\3 = 0$ so that the statement is trivial for $\tau = \smalltree$.
\end{proof}

Once this is established we see that, in order to establish convergence of the processes $Y^\tau_\eps$ in \eref{e:defYeps}, 
it is sufficient to establish it with $C^\tau_\eps$ replaced by $K^\tau_\eps$. 
With this in mind, we define
processes $F^\tau_\eps$ by 
\begin{equ}[e:defFeps]
F^\tau_\eps(t) = \int_0^t \sum_{k\in \Z_\star} \bar X^{\tau_1}_{\eps,k}(s)\, \bar X^{\tau_2}_{\eps,-k}(s)\,ds - K^\tau_\eps\,t\;.
\end{equ}
Since $F^\tau_\eps = \Pi_0 Y^\tau_\eps + (C^\tau_\eps - K^\tau_\eps)$, the convergence of the processes $Y^\tau_\eps$
to a limit is now equivalent to the convergence of $F^\tau_\eps$ to a limit process $F^\tau$. This in turn is ensured by the following result.

\begin{proposition}
For every $\tau \in \{\2,\3,\4,\5\}$, there exists a process $F^\tau$ such that $F^\tau_\eps \to F^\tau$ in probability. Furthermore, for every $\delta > 0$,
one has $F^\2 \in \CC^{{3\over 4} - \delta}$, and $F^\tau \in \CC^{1 - \delta}$ for every $\tau \in \{\,\3,\4,\5\}$.
\end{proposition}

\begin{proof}
See Lemmas~\ref{prop:convF2} and
\ref{prop:convF4} below.
\end{proof}

\subsection{Convergence of the renormalisation constants}

In this section, we show that one does indeed have $K_\eps^\tau - C_\eps^\tau \to K^\tau$ as $\eps \to 0$,
for some constants $K^\tau$ that do not depend on the choice of mollifier $\phi$. The simplest case is when $\tau = \2$, which is 
covered by the following lemma.

\begin{lemma}\label{lem:convK2}
The identity
\begin{equ}
K_\eps^\2 \approx C_\eps^\2 + 1 = {1\over \eps} \int \phi^2(x)\,dx + 1\;,
\end{equ}
holds up to an error of order $\CO(\eps)$.
\end{lemma}

\begin{proof}
By definition, we have the identity
\begin{equ}
K_\eps^\2 = \sum_{k \in \Z_\star} \phi^2(k\eps) = \sum_{k \in \Z} \phi^2(k\eps) - 1\;.
\end{equ}
It now suffices to note that $\eps \sum_k \phi^2(k\eps)$ is a Riemann sum approximation to the integral
$\int \phi^2(x)\,dx$. Since, on the whole of $\R$, this approximation agrees with the trapezoidal rule, it is of second order, so that the claim 
follows.
\end{proof}

For $\tau = \nicetree$ on the other hand, it is much more difficult to get a handle on the corresponding constants.
In principle, the precise value of $K_\eps^\4$ does not really matter, since the important fact is only that $K_\eps^\4 + 4 K_\eps^\5$
is approximately constant as $\eps \to 0$, but since it is possible to actually compute this constant, we state the result and sketch its proof.

\begin{lemma}\label{lem:KY}
Let $\psi(x) = \phi(x) \phi'(x)$.
Then, there exists a constant $K$ independent of $\phi$ such that 
one has the identity
\begin{equ}
K_\eps^\4 \approx {4\pi \over \sqrt 3}|\log \eps|  - 8\int_{\R_+} \int_\R {x \psi(y)\phi^2(y-x) \log y \over x^2 -xy +y^2} \,dx\,dy
 + K \;,
\end{equ}
up to an error of order $\CO(\sqrt \eps \log\eps)$.
\end{lemma}

\begin{proof}
In the proof, we will denote by $K$ a generic constant independent of $\phi$. We will not keep track of the precise value of 
$K$, so that it may change without warning from expression to expression. 

Following the definition of $X^\2_\eps$ and using the correlation function of $\bar X^\1_{\eps}$, it is tedious but straightforward to check that
for $k \neq 0$, one has the identity
\begin{equ}
\E |\bar X^\2_{\eps,k}|^2 = \sum_{m \not\in \{0,k\}} {\phi^2(\eps m )\phi^2\bigl(\eps(k-m)\bigr) \over k^2 + m^2 - km} \;.
\end{equ}
(Here, we used again the shorthand $\bar X^\tau = \d_x X^\tau$.)
Using the fact that the summand is symmetric under the substitution $(k,m)\leftrightarrow (-k,-m)$ and treating separately the terms $m \in \{0,k\}$, we thus obtain
\begin{equ}[e:formKY]
K_\eps^\4 = \sum_{k \in \Z} \E |\bar X^\2_{\eps,k}|^2 = 2\sum_{k \ge 1} \sum_{m \in \Z}  {\phi^2(\eps m)\phi^2\bigl(\eps(k-m)\bigr) \over k^2 -km + m^2 } 
- 4 \sum_{k \ge 1} {\phi^2(\eps k)\over k^2}\;.
\end{equ}
Since the second term differs from $2\pi^2/3$ only by an error of order $\CO(\eps)$, we focus on the first term.
Note now that for $k$ large, we can interpret the inner sum as a Riemann sum so that, setting $\delta = {1\over k}$, we have
\begin{equs}
 \sum_{m \in \Z}  &{\phi^2(\eps m)\phi^2\bigl(\eps(k-m)\bigr) \over k^2 -km + m^2 } = {\delta \over k} \sum_{m\in \Z} {\phi^2(\eps k \,\delta m)\phi^2\bigl(\eps k(1-\delta m)\bigr) \over 1 - \delta m + (\delta m)^2} \\
&\qquad= {1\over k} \int_{\R} {\phi^2(\eps k \,x)\phi^2\bigl(\eps k(1-x)\bigr) \over 1 - x+ x^2}\,dx + {G_\delta\over k} + \CO(\eps^2 \wedge k^{-2})\\
&\qquad\eqdef C_\eps(k) + {G_\delta \over k} + \CO(\eps^2 \wedge k^{-2})\;,
\end{equs}
Where $G_\delta$ is the error in the Riemann sum approximation:
\begin{equ}
G_\delta =  \sum_{m\in \Z} {\delta \over 1 - \delta m + (\delta m)^2}  - \int_\R {dx \over 1-x+x^2}\;.
\end{equ}
Since $G_\delta = \CO(\delta)$ for $\delta \ll 1$, the term $G(\delta)/k$ is summable, so that
\begin{equ}
K_\eps^\4 = 2\sum_{k\ge 1} C_\eps(k) + K + \CO(\eps)\;.
\end{equ}
At this stage, we note that one has
\begin{equ}
\int_\R {dx  \over 1 - x + x^2} = {2\pi \over \sqrt 3}\;,
\end{equ}
from which it follows immediately that
\begin{equ}[e:estCk]
C_\eps(k) = {2\pi \over \sqrt 3 k} + \CO(\eps)\;,
\end{equ}
where the error is $\CO(\eps)$, uniformly in $k$.

We now break the sum over $C_\eps(k)$ into two parts. We first obtain from \eref{e:estCk} that
\begin{equ}
2\sum_{k=1}^{1/\sqrt \eps} C_\eps(k) = {2\pi \over \sqrt 3 }|\log \eps| + {4\pi \gamma \over \sqrt 3} + \CO(\sqrt \eps)\;,
\end{equ}
where $\gamma$ is the Euler-Mascheroni constant. For the remaining terms, we first note that $C_\eps(k) = 0$ for $k > 4/\eps$ because of the properties of $\phi$.
We then write $y = \eps k$ and we approximate the sum over $k$ by an integral:
\begin{equ}
2\sum_{k = 1/\sqrt \eps}^{4/\eps}  C_\eps(k) = 2\int_{\sqrt \eps}^4 \int_{\R} {\phi^2(y x)\phi^2\bigl(y(1-x)\bigr) \over y\bigl(1 -x+ x^2\bigr)}\,dx\,dy + \CO(\sqrt \eps)\;.
\end{equ}
Note that the error is of order $\sqrt \eps$ because, for fixed $x \in \R$, the variation of the integrand in $y$ is of order $1/\sqrt \eps$.
Integrating by parts over $y$, we see that this is equal to
\begin{equs}
{} &{2\pi \over \sqrt 3}|\log \eps| + \CO(\sqrt \eps \log\eps) \\
&- 4\int_{\R_+} \int_\R {\bigl(x \psi(yx)\phi^2\bigl(y(1-x)\bigr) + (1-x) \psi\bigl(y(1-x)\bigr)\phi^2(yx)\bigr)\log y  \over 1-x+x^2} \,dx\,dy \;.
\end{equs}
The claim now follows from the symmetry of the integrand under the substitution $x \mapsto 1-x$ and by performing the change of variables $yx \mapsto x$.
\end{proof}

Finally, we show that $K_\eps^\5$ can indeed be expressed in terms of $K_\eps^\4$.

\begin{lemma}\label{lem:convK5}
One has $4K_\eps^\5 = -K_\eps^\4 -{\pi^2 \over 3} + \CO(\eps)$.
\end{lemma}

\begin{proof}
By definition, one has the identity
\begin{equ}
X^\3_{\eps,k}(t) = \sum_{\ell \in \Z}  \int_{-\infty}^t e^{-k^2 (t-s)}  \bar X^\2_{\eps,\ell}(s) \bar X^\1_{\eps,k-\ell}(s)\,ds\;,
\end{equ}
so that
\begin{equ}[e:defK5]
K_\eps^\5 = -\sum_{k \in \Z} \sum_{\ell \in \Z} k\ell \int_{-\infty}^t e^{-k^2 (t-s)}  X^\2_{\eps,\ell}(s) \bar X^\1_{\eps,k-\ell}(s)\bar X^\1_{\eps,-k}(t)\,ds\;.
\end{equ}
To estimate this quantity, the following calculation is helpful. For $s > 0$ and $k,\ell \in \Z$, set 
\begin{equ}
H_{k,\ell}(s) = \E X^\2_{\eps,\ell}(0) \bar X^\1_{\eps,k-\ell}(0)\bar X^\1_{\eps,-k}(s)\;.
\end{equ}
For $k,\ell \neq 0$ with $k \neq \ell$, one then has the identity 
\begin{equ}
H_{k,\ell}(s) = \sum_{m \in \Z}\int_{-\infty}^0 e^{\ell^2 r} \E \bigl(\bar X^\1_{\eps,\ell-m}(r)\bar X^\1_{\eps,m}(r)\bar X^\1_{\eps,k-\ell}(0)\bar X^\1_{\eps,-k}(s)\bigr)\,dr\;.
\end{equ}
One can check that the only non-vanishing terms in this sum are those with $m = \ell-k$ and with $m=k$, which yields
\begin{equs}
H_{k,\ell}(s) &=2\phi^2(\eps k)\phi^2\bigl(\eps(k-\ell)\bigr) \int_{-\infty}^0 e^{\ell^2 r + (k-\ell)^2 r-k^2 (s-r)} \,dr \\
&= {\phi^2(\eps k)\phi^2\bigl(\eps(k-\ell)\bigr) \over \ell^2 + k^2 - k\ell}e^{-k^2 s}\;.
\end{equs}
Inserting this expression into \eref{e:defK5}, it follows that one has the identity
\begin{equ}
K_\eps^\5 = -\sum_{k \in \Z_\star } \sum_{\ell \neq k } {\ell \phi^2(\eps k)\phi^2\bigl(\eps(k-\ell)\bigr) \over 2k(\ell^2 + k^2 - k\ell)}\;,
\end{equ}
which, by performing the substitution $\ell = k-m$, we can rewrite as
\begin{equ}
K_\eps^\5 = -\sum_{k,m \in \Z_\star } {(k-m) \phi^2(\eps k)\phi^2(\eps m) \over 2k(k^2 + m^2 - km)}\;.
\end{equ}
Performing a similar substitution in \eref{e:formKY}, we obtain
\begin{equ}
K_\eps^\4 = \sum_{k,m \in \Z_\star } {\phi^2(\eps k)\phi^2(\eps m) \over k^2 + m^2 - km} -\sum_{k \in \Z_\star} {\phi^2(\eps k) \over k^2} \;,
\end{equ}
so that we have the identity
\begin{equ}
K_\eps^\4 + 4K_\eps^\5 =  \sum_{k,m \in \Z_\star } {2m-k \over k}{\phi^2(\eps k)\phi^2(\eps m) \over k^2 + m^2 - km} -{\pi^2 \over 3} + \CO(\eps)\eqdef I_\eps -{\pi^2 \over 3} + \CO(\eps)\;.
\end{equ}
At this stage, we note that, apart from the prefactor $(2m-k)/k$, the summand is symmetric under the substitution $(k,m) \leftrightarrow (m,k)$.
As a consequence, we can rewrite this sum as
\begin{equ}[e:quantity]
I_\eps = {1\over 2} \sum_{k,m \in \Z_\star } \Bigl({2m-k \over k} + {2k-m \over m}\Bigr){\phi^2(\eps k)\phi^2(\eps m) \over k^2 + m^2 - km}\;.
\end{equ} 
Since one furthermore has the identity
\begin{equ}
{2m-k \over k} + {2k-m \over m} = {2(m^2 + k^2 - km) \over km}\;,
\end{equ}
the quantity in \eref{e:quantity} is equal to 
$\bigl(\sum_{k \in \Z_\star} {\phi(k\eps)\over k}\bigr)^2$.
This vanishes identically since $\phi$ is even, thus concluding the proof.
\end{proof}

\subsection{Convergence of the fluctuation processes}

In this section, we show that the processes $F^\tau_\eps$ defined in \eref{e:defFeps} have limits as $\eps\to 0$.
We start with the process $F^\2$:
\begin{proposition}\label{prop:convF2}
There exists a limiting process $F^\2$ such that $F^\2_\eps \to F^\2$ in probability in $\CC^\alpha$ for every $\alpha < {3\over 4}$.
\end{proposition}

\begin{proof}
From the definition of $F^\2_\eps$ and the expression for the correlation functions of $\bar X^\1_\eps$, it is straightforward to see that 
\begin{equs}
\E |F^\2_\eps(t) - F^\2_\eps(s)|^2 &= \sum_{k\in \Z_\star} \phi^4(k\eps) \int_s^t\int_s^t e^{-k^2|r-r'|}\,dr\,dr' \\
&\lesssim \sum_{k\in \Z_\star} \int_s^t\int_s^{r'} {1\over |k|^{2\alpha} |r-r'|^\alpha}\,dr\,dr'\\
&\lesssim \int_s^t |s-r'|^{1-\alpha}\,dr' \lesssim |t-s|^{2-\alpha}\;,
\end{equs}
provided that $\alpha > {1\over 2}$. The fact that the sequence is actually Cauchy then follows in the same way
as in the proof of Proposition~\ref{prop:generalConvergence} below.
\end{proof}

For the more complicated trees, a more systematic approach, very similar to the previous two sections, is required.
We note that for a general tree $\tau$ with $\tau = [\tau_1,\tau_2]$, one has the identity
\begin{equs}
\E |F^\tau_\eps(t) &- F^\tau_\eps(s)|^2 = \sum_{k,\ell \in \Z_\star} \int_s^t \int_s^t \Bigl(\E \bigl(\bar X^{\tau_1}_{\eps,k}(r)\bar X^{\tau_2}_{\eps,-k}(r)\bar X^{\tau_1}_{\eps,\ell}(r')\bar X^{\tau_2}_{\eps,-\ell}(r')\bigr)\\
 &\qquad - \E \bigl(\bar X^{\tau_1}_{\eps,k}(r)\bar X^{\tau_2}_{\eps,-k}(r)\bigr)\,\E\bigl(\bar X^{\tau_1}_{\eps,\ell}(r')\bar X^{\tau_2}_{\eps,-\ell}(r')\Bigr)\,dr\,dr'\;. \label{e:exprFeps}
\end{equs}
The reason why we can rewrite it in this way is that, by definition,
\begin{equ}
K^\tau_\eps = \sum_{\ell \in \Z_\star} \E\bigl(\bar X^{\tau_1}_{\eps,\ell}(r)\bar X^{\tau_2}_{\eps,-\ell}(r)\bigr)\;,
\end{equ}
which is independent of $r$ by the stationarity of the processes $X^\tau_\eps$.

As before, we can write this expectation as a sum over pairings of the tree $\tau$ and
all cycles of the corresponding graph. Previously, we always restricted ourselves to cycles in $\cC_\star(\CG,\CE) \sim \LL^\tau_P$,
which was defined as consisting of those cycles that associate a non-zero value to every edge. 
This time however, we are precisely interested in only those cycles that do associate the value $0$
to the distinguished edge $\bar e$. We therefore denote by $\LL^\tau_{P;0}$ the set of all cycles
that associate the label $0$ to the distinguished edge $\bar e$ \textit{but not to any other edge}.

\begin{remark}
For the pairing associated to the graph depicted in \eref{e:pairingDangle}, the set $\LL^\tau_{P;0}$ is empty. Indeed, any cycle
giving the value $0$ to the horizontal edge at the bottom of the graph also necessarily gives the value $0$
to the horizontal edge at the top, but this is not allowed.
In general, all pairings yielding graphs with a small loop touching
$\bar e$ are ruled out in this way.
\end{remark}

In principle, the consequence of this is that we have to consider a potentially larger set of pairings than what we did in 
in Section~\ref{sec:constructX}. Recall indeed Lemma~\ref{lem:pairCross}, which characterised the 
pairings  yielding a non-empty set $\LL^\tau_P$ as those pairings which contain at least one element connecting 
the two copies of $\tau$. (Denote these pairings by $\hat \CP^\tau$.)
While the pairings that do not contain any such connection yield an empty set $\LL^\tau_P$,
they would certainly not yield an empty set $\LL^\tau_{P;0}$. In fact, pairings $P \in \CP^\tau\setminus \hat \CP^\tau$
are precisely those with the property that \textit{every} cycle of the corresponding graph belongs to $\LL^\tau_{P;0}$!
Fortunately, we see that the pairings in $P \in \CP^\tau\setminus \hat \CP^\tau$ are precisely those that appear in the
second line of \eref{e:exprFeps}, so that their contribution to \eref{e:exprFeps} vanishes. As a consequence, we still have the identity
\begin{equ}
\E |F^\tau_\eps(t) - F^\tau_\eps(s)|^2 = \sum_{P \in \hat \CP^\tau}\sum_{L \in \LL^\tau_{P;0}} \! C_\eps(L)\int_s^t \int_s^t \CF^\tau_\sym(P,L;r-r')\,dr\,dr' \;,
\end{equ}
where $\CF_\sym^\tau$ is defined exactly as in Section~\ref{sec:systematic} and where
\begin{equ}
C_\eps(L) = \prod_{e \in \CT} \phi^2(\eps L_e)\;.
\end{equ}

With this notation, we have the following result, which is the counterpart in this context to Theorem~\ref{theo:abstractConv}:

\begin{proposition}\label{prop:convFluct}
Let $\tau$ be a non-trivial binary tree and assume that there exists $\alpha > 0$ such that 
\begin{equ}[e:defFaC]
\Bigl|\int_s^t \int_s^t \CF^\tau_\sym(P,L;r-r')\,dr\,dr'\Bigr| \le |t-s|^{2\alpha} \CF_\alpha(P,L)\;,
\end{equ}
for some constants $\CF_\alpha(P,L)$ and every $s<t$ with $|t-s| \le 1$. 
Then, if 
\begin{equ}
\sum_{P \in \hat \CP^\tau}\sum_{L \in \LL^\tau_{P;0}} \CF_\alpha(P,L) < \infty \;,
\end{equ}
there exists a process $F^\tau$ such that 
$F^\tau_\eps \to F^\tau$ in probability in $\CC^{\bar \alpha}$ for every $\bar \alpha < \alpha$.
\end{proposition}

\begin{proof}
Using Kolmogorov's continuity criterion, the proof is virtually identical to that of Proposition~\ref{prop:generalConvergence}.
\end{proof}

It thus remains to provide a criterion for the summability of $\CF_\alpha(P,L)$ (which we can define
as the smallest constant such that \eref{e:defFaC} holds) for a given pairing $P$.
For this, we proceed similarly to Section~\ref{sec:constructX} but our construction is slightly different.
Given a pairing $P \in \hat \CP^\tau$ and given $\kappa > 0$, we now construct a 
graph $(\hat \CG, \hat \CE)$ and a set of weights $\GG_\kappa^\tau(P;0)$ by following the same procedure as in 
Section~\ref{sec:constructGraph}. There are only two differences: first, instead of setting $\CW(v) = 0$ for $v \in \bar e$ in Step~4, we actually
set $\CW(v) = \kappa$ for these two vertices. Then, at the end of the algorithm, we erase the edge $\bar e$, so that
the graph $(\hat \CG, \hat \CE)$ that we consider is actually given by $\hat \CG = \CG$ and $\hat \CE = \CE \setminus \{\bar e\}$.

We then have the following result:

\begin{proposition}\label{prop:sumFluct}
For a given $P \in \hat \CP^\tau$, let $\GG_\kappa^\tau(P;0)$ and $(\hat \CG, \hat \CE)$ be as above. If
there exists $\kappa < 2$ such that there exists a summable element in $\GG_\kappa^\tau(P;0)$, then 
\begin{equ}
\sum_{L\in \LL^\tau_{P;0}} \CF_\alpha(P,L) < \infty\;,
\end{equ}
for $\alpha = 1 - {\kappa \over 2}$.
\end{proposition}

\begin{proof}
The proof is virtually identical to that of Theorem~\ref{theo:graphKernel}, so we only focus on those steps that actually 
differ. Note first that the reason for erasing the edge $\bar e$ after the last step is that the set of cycles of $(\CG, \CE)$ that give the value $0$ to $\bar e$
and non-zero values to all other edges are, by definition, in a one-to-one correspondence with all cycles of $(\hat \CG, \hat \CE)$ that
give non-zero values to all edges in $\hat \CE$.

This immediately yields the claim for the case $\kappa = 0$ since in that case the proof of Theorem~\ref{theo:graphKernel} implies that
\begin{equ}
\sum_{L \in \LL^\tau_{P;0}} \sup_{\delta \in \R} |\CF^\tau_\sym(P,L;\delta)| < \infty\;,
\end{equ}
so that we do indeed have the required bound with $\alpha = 1$. 
If $\kappa > 0$, we use the fact that we only need a bound on the integrated quantity
\begin{equ}
\int_s^t \int_s^t \CF^\tau_\sym(P,L;r-r')\,dr\,dr'\;.
\end{equ}
As a consequence, in the last step of the proof of Theorem~\ref{theo:graphKernel}, we should replace \eref{e:boundFT} by the bound
\begin{equs}
\Bigl|\int_s^t &\int_s^t \CF^\tau_\sym(P,L;r-r')\,dr\,dr'\Bigr| \\
&\lesssim \int_s^t\int_s^t \int_{\R^{\bar \CG}} \Bigl(\prod_{e \in \hat \CE} |\CS L_{e}|^{-\CL_0(e)} e^{-c (\CS L)_e^2|\delta T_e|} \Bigr)\, \Bigl(\prod_{u \in \bar \CG} dT_u\Bigr)\,dT_{v}\,dT_{\bar v}\;,
\end{equs}
where we denote by $v$ and $\bar v$ the two vertices adjacent to $\bar e$. One then proceeds in exactly the same way, but noting that
for $a>0$ one has the inequality
\begin{equ}
\int_s^t e^{-a^2 |u-x|}\,dx \le |t-s|^\alpha |a|^{2-2\alpha} = {|t-s|^\alpha \over |a|^\kappa}\;, 
\end{equ} 
uniformly in $u$, provided that $\alpha = 1-{\kappa\over 2}$ as in the statement of the proposition. The claim
then follows by repeatedly applying H\"older's inequality, just like in the proof of Theorem~\ref{theo:graphKernel}.
\end{proof}

We now have the required tool to construct the remaining processes $F^\tau$. We have:

\begin{proposition}\label{prop:convF4}
For every $\tau \in \{\smalltree,\nicetree,\bigtree\}$, there exists a limiting process $F^\tau$ such that $F^\tau_\eps \to F^\tau$ in probability in $\CC^\alpha$ for every $\alpha < {1}$.
\end{proposition}

\begin{proof}
By Propositions~\ref{prop:convFluct} and \ref{prop:sumFluct}, it suffices to exhibit an element in $\GG_\kappa^\tau(P;0)$ for every pairing $P \in \hat \CP^\tau$.
In the case of $\tau = \smalltree$, there are only two such pairings that yield a non-empty set $\LL_{P;0}^\tau$:
\begin{equ}
\mhpastefig{tree12zeromode}
\end{equ}
(Note that we display the graphs $(\hat \CG, \hat \CE)$ for which the edge $\bar e$ has been removed, this is why the two root vertices now only have degree $2$.)
As before, dotted lines have weight  $0$, dashed lines have weight  $\kappa$, and plain lines have weight $1$. Again,
the loop and its adjacent edge in the first graph have weights as in \eref{e:looplabel}, so that after the application of Algorithm~1, it is indeed 
equivalent to 
having an edge with weight $0$. It is easy to see that both of these graphs are summable for every $\kappa > 0$, thus yielding the
claim for $\tau = \smalltree$.

In the case $\tau = \nicetree$, we can even take $\kappa = 0$. Indeed, recalling that the summability of graphs improves 
by removing edges, we note that if the criterion of Proposition~\ref{prop:graphicalAlgorithm} is satisfied for a given pairing $P$,
then  the assumption of Proposition~\ref{prop:sumFluct} is satisfied with $\kappa = 0$.

In the case $\tau = \bigtree$, the only pairings for which we have not verified that the criterion of Proposition~\ref{prop:graphicalAlgorithm}
is satisfied are those for which there is a pairing connecting the two roots. There are exactly three such pairings left
(one for each pairing of $\smalltree$). One can then check that the following weights do belong to 
$\GG_\kappa^\tau(P;0)$ with the same colour-coding conventions as before:
\begin{equ}
\mhpastefig{proofFluct}
\end{equ}
Again, it is straightforward to verify that they are all summable, thus concluding the proof.
\end{proof}

\section{Fine control of the universal process}
\label{sec:controlProcess}

The purpose of this section is to show that, for each fixed $t>0$, the process $\bar X^\5(t)$ is controlled
(in the sense of controlled rough paths) by the process $\Phi$. Recall that, by definition, 
$\bar X^\5$ is given by
\begin{equ}
\bar X^\5(t) = \d_x \int_{-\infty}^t P_{t-s} \bigl(\bar X^\1(s)\,\bar X^\3(s)\bigr)\,ds\;,
\end{equ}
where $P_t$ denotes the heat semigroup. Recall furthermore that the process $\Phi$ can be written as
\begin{equ}
\Phi(t) = \d_x \int_{-\infty}^t P_{t-s} \bar X^\1(s)\,ds\;.
\end{equ}
Comparing these two expressions, this suggests that $\bar X^\5(t)$ is controlled by $\Phi(t)$ with ``derivative process'' given by $\bar X^\3(t)$.
The aim of this section is to prove this fact, which is a crucial ingredient of our proofs.

In order to prove this, we would like to show that, for every $t>0$, it is possible to 
build a rough path over the pair
$\bigl(X^\1(t),\bar X^\3(t)\bigr)$. This would then enable us to apply the results from 
Section~\ref{sec:boundRemM} to show that
$\bar X^\5(t)$ is indeed  controlled by $\Phi$. Actually, we will show slightly less, namely that
the ``rough path norm'' for the pair $\bigl(X^\1_\eps(t),\bar X^\3_\eps(t)\bigr)$, with a suitably defined
``area process'' has uniformly bounded moments as $\eps \to 0$. It is clear from the proofs that one
actually has convergence to a limiting rough path, but the rigorous proof of his fact would be
slightly tricky so that we skip it since we do not really need it.

For this, we need to build the integral of 
$\bar X^\3(t)$ against $X^\1(t)$. More precisely, we would like to obtain control, as $\eps \to 0$,
of the quantity
\begin{equ}[e:wantedQuantity]
\int_x^y \bigl(\bar X_\eps^\3(t,z) - \bar X_\eps^\3(t,x)\bigr)\bar X_\eps^\1(t,z)\,dz\;.
\end{equ}
Unfortunately, this turns out to be impossible: the above quantity diverges logarithmically for
any pair $x$ and $y$! Fortunately, this divergence is not too difficult to control: indeed it arises
from an ``infinite constant'', which we subtract. The quantity that we will thus attempt to control
is given by
\begin{equs}
\XX_t^\eps(x,y) &\eqdef\int_x^y \bigl(\bar X_\eps^\3(t,z) - \bar X_\eps^\3(t,x)\bigr)\bar X_\eps^\1(t,z)\,dz
 - {y-x\over 2\pi} \int_0^{2\pi} \bar X_\eps^\3(t,z)\bar X_\eps^\1(t,z)\,dz\\
&=\int_x^y \bigl(\Pi_0^\perp\bigl(\bar X_\eps^\3(t)\bar X_\eps^\1(t)\bigr)(z) - \bar X_\eps^\3(t,x)\bar X_\eps^\1(t,z)\bigr)\,dz\;.\label{e:defXXbasic}
\end{equs}
Using Fourier components, we have the identity
\begin{equs}
\XX_t^\eps(x,y) &= \sum_{k\in \Z_\star} \sum_{m \in \Z} \int_x^y \bar X_{\eps,k-m}^\3(t) \bar X_{\eps,m}^\1(t) e^{ikz}\bigl(1 - e^{-i(k-m)(z-x)}\bigr)\,dz\\
&\quad - \sum_{m \in \Z} \int_x^y \bar X_{\eps,-m}^\3(t) \bar X_{\eps,m}^\1(t) e^{im(z-x)}\,dz\;. \label{e:defXXtree}
\end{equs}
Note that, since $\XX_t^\eps$ differs from \eref{e:wantedQuantity} only by a term of the form $(y-x)K_\eps(t)$ for
some process $K_\eps$, it automatically satisfies the required consistency relation \eref{e:constr} for every 
$\eps > 0$, so that it is a perfectly valid area process for the pair $\bigl(X^\1_\eps(t),\bar X^\3_\eps(t)\bigr)$.

\begin{remark}
The logarithmic divergence arising in the construction of $\XX_t^\eps$ is really \textit{the same} logarithmic divergence
$C_\eps^\5$ arising in the construction of $Y^\5$ in the following sense. If, instead of \eref{e:defXbar5} below, we considered the
same expression with $p'$ replaced by $p$ (which is  essentially the definition of $Y_{\eps,t}^\5$), then the difference 
between the rough integral $\iint$ and the usual Riemann integral would precisely result in 
a term that is constant in space and that kills the divergence in such a way that the resulting expression converges.
\end{remark}

\begin{remark}
In a somewhat vague sense, propositions~\ref{prop:simplePairing} and \ref{prop:boundK5} suggest that $X^\3$ behaves ``as if'' it was composed of
one part with regularity $\CC^{{3\over 2}-\beta}$ that is independent of $X^\1$ and one dependent part with regularity
$\CC^{2-\beta}$. Indeed, the bound of Proposition~\ref{prop:boundK5} would yield $\CC^{2-\beta}$ regularity for arbitrary $\beta > 0$
if it held for all pairings $P$, while Proposition~\ref{prop:simplePairing} precisely treats only those pairings that would arise if $X^\3$ and $X^\1$ were
independent.

If we could really decompose $\XX^\3$ as the sum of two processes with these properties, then we could
make sense of the limit $\XX_t^\eps \to \XX_t$ in almost exactly the same way as in \cite{GaussI}.
Unfortunately, while this argument is undoubtedly appealing, it doesn't seem obvious how to make it work in
a direct way. 
\end{remark}

The main technical bound of this section is the following:
\begin{proposition}\label{prop:boundXXeps}
For every $\bar \kappa > 0$, one has the bound
\begin{equ}[e:unifXXeps]
\sup_{\eps \in (0,1]} \sup_{t \in \R} \E \|\XX_t^\eps\|_{1-2\bar \kappa}^2 < \infty\;.
\end{equ}
\end{proposition}

\begin{proof}
In view of the definition \eref{e:FourierXXtree}, this is the content of propositions~\ref{prop:convergenceXX1} and \ref{prop:convergenceXX2} below.
\end{proof}

Furthermore, it follows immediately from \eref{e:corint} that, for any smooth function $f$ and 
for every $\eps > 0$, one has the identity
\begin{equ}[e:roughIntEps]
\iint_{S^1} f(z)\, \bar X_\eps^\3(t,z)\,dX_\eps^\1(t,z) = \int_{S^1} f(z)\, \Pi_0^\perp \bigl(\bar X_\eps^\3(t)\bar X_\eps^\1(t)\bigr)(z)\,dz\;,
\end{equ}
where we used $\XX_t^\eps$ as the ``area process'' required to give meaning to the rough integral on the left
hand side. 
Setting
\begin{equ}
R^\5_{\eps,t}(x,y) = \delta \bar X_{\eps,t}^\5\bigr(x,y) - \bar X_{\eps,t}^\3(x) \,\delta \Phi_{\eps,t}(x,y)\;,
\end{equ}
we then have the following consequence of Proposition~\ref{prop:boundXXeps}, which is the main result of this section:

\begin{theorem}\label{theo:controlX5}
One has $R^\5_\eps \to R^\5$ in probability in $\CC(\R, \CC_2^{1-\delta})$ for every $\delta > 0$.
\end{theorem}

\begin{proof}
It is crucial to note that, with $\XX_t^\eps$ given by \eref{e:defXXbasic}, one has the identity
\begin{equ}[e:idenX5]
\bar X_{\eps,t}^\5 = P_t \bar X_{\eps,0}^\5 + \bigl(\CM \bar X_{\eps}^\3\bigr)_t(x)\;,
\end{equ}
where we define $\CM$ as in Section~\ref{sec:boundRemM} by 
\begin{equ}[e:defXbar5]
\bigl(\CM v\bigr)_t(x) = \int_0^t \iint_{S^1} p'_{t-s} (x-y) \,v_s(y)\,dX_{\eps,s}^\1(y)\,ds\;,
\end{equ}
where $v$ is a rough path controlled by $\bar X_{\eps}^\3$ (in this case trivially with $v' = 1$ and no remainder) 
and where we use $\XX_\eps$ to define the rough integral between the two processes.
The reason why \eref{e:idenX5} holds with $\CM$ defined with a ``rough integral'' is that, 
if we replace $\iint$ by $\int$ in the definition of $\CM$, then it follows
from \eref{e:roughIntEps} that we do not change the resulting expression
since  $p_{t-s}'$ integrates to $0$.

We are now able to apply the results of Section~\ref{sec:boundRemM} with $Y = X_{\eps}^\1$, 
$Z = \bar X_{\eps}^\3$ and $\YY = \XX_\eps$. Setting
\begin{equ}
\KK^{ \kappa}_{\eps,t} \eqdef \bigl(1+\|\bar X_{\eps,t}^\3\|_{{1\over 2}- \kappa} + \|X_{\eps,t}^\1\|_{{1\over 2}- \kappa}\bigr)^2  + \|\XX_{\eps,t}\|_{1-2 \kappa}\;,
\end{equ}
it follows from
Proposition~\ref{prop:boundRemainder} with $\kappa = {1-\delta\over 4}$ and $\bar \kappa = \kappa$ 
that one has the bound
\begin{equs}
\|R^\5_{\eps,t}\|_{1-{\delta\over 2}} 
&\lesssim t^{{3 \over 8}(\delta-1)} \|\bar X_{\eps,t}^\3\|_\infty \|\Phi_0\|_{{1+\delta\over4}} 
+ \int_0^t (t-s)^{{\delta\over 4} - 1 -  \kappa} \KK_{\eps,s}^{ \kappa}\,ds\\
&\quad +\int_0^t (t-s)^{{\delta - 5\over 4} -{ \kappa\over2}} \|X_{\eps,s}^\1\|_{{1\over2}- \kappa}\,\|\bar X_{\eps,t}^\3- \bar X_{\eps,s}^\3\|_\infty \,ds \;.\label{e:boundR5eps}
\end{equs}
Note now that since $\XX_\eps$, $\bar X_\eps^\3$ and $X_\eps^\1$ are all random variables belonging to 
a finite combination of Wiener chaoses of fixed order,
it follows from propositions~\ref{prop:boundXXeps} and \ref{prop:boundX3} that, for every $p>0$, every $\kappa \in (0,{1\over 4})$,
and every time interval $[a,b]$, one has the bound
\begin{equ}
\E \int_a^b \bigl( \KK_{\eps,s}^{ \kappa}\bigr)^p\,ds <\infty\;,
\end{equ}
uniformly over $\eps \in (0,1]$. As a consequence, the second term in \eref{e:boundR5eps} is bounded in probability, uniformly over bounded intervals, provided that we choose $\kappa < \delta$ and that one chooses $p$ sufficiently large. 
Similarly, the third term is bounded in probability by Proposition~\ref{prop:boundX3}. Since the first term
is harmless (except at $t=0$, but it suffices to change the origin of time to deal with that), we have shown that, for
every $T>0$, one has the bound
\begin{equ}
\sup_{\eps \in (0,1]} \E \sup_{t \in [-T,T]} \|R^\5_{\eps,t}\|_{1-{\delta\over 2}} < \infty\;.
\end{equ}
To show that one actually has convergence towards $R^\5$, note first that convergence in the supremum
norm follows from propositions~\ref{prop:boundX3} and \ref{prop:convergenceX5}. Since one has the 
interpolation inequality 
\begin{equ}
\|R\|_{1-\delta} \le \|R\|_{1-{\delta\over 2}}^{2-2\delta\over 2-\delta}\|R\|_\infty^{\delta\over 2-\delta}\;,
\end{equ}
the claim then follows.
\end{proof}

The  remainder of this section is devoted to the proof that the bound \eref{e:unifXXeps} does indeed hold.
We are going to exploit the translation invariance so that, for simplicity, we use the shorthand notation
\begin{equ}
\XX_t^\eps(\delta) \eqdef \XX_t^\eps(0,\delta)\;.
\end{equ}
Performing the integrals in \eref{e:defXXtree}, one obtains the identity
\begin{equs}
\XX_t^\eps(\delta) &= \sum_{k,m\in \Z_\star} \bar X_{\eps,k-m}^\3(t) \bar X_{\eps,m}^\1(t) \Bigl({e^{ik\delta}-1\over k} - {e^{im\delta} - 1 \over m}\Bigr)\\
&\quad - \sum_{m \in \Z_\star} \bar X_{\eps,-m}^\3(t) \bar X_{\eps,m}^\1(t) {e^{im\delta} - 1 \over m} \label{e:FourierXXtree} \\
&\eqdef \XX_t^{(1),\eps}(\delta) + \XX_t^{(2),\eps}(\delta)\;.
\end{equs}

We first consider $\XX_t^{(1)}$ and postpone the bounds on $\XX_t^{(2)}$ to the end of this section. 
In order to bound $\XX_t^{(1)}$, we rewrite 
its second moment as 
\begin{equs}
\E \bigl(\XX_t^{(1),\eps}(\delta)\bigr)^2 &= \sum_{k,m, \bar m\in \Z_\star} \E \bigl(\bar X_{\eps,k-m}^\3(t) \bar X_{\eps,m}^\1(t)
\bar X_{\eps,-k-\bar m}^\3(t) \bar X_{\eps,\bar m}^\1(t)\bigr) \label{e:boundXX12}\\
&\qquad \times  \Bigl({e^{ik\delta}-1\over k} - {e^{im\delta} - 1 \over m}\Bigr)\Bigl({1-e^{-ik\delta}\over k} - {e^{i\bar m\delta} - 1 \over \bar m}\Bigr)\;.
\end{equs}
(In principle, one should have a $\bar k$ appearing, but the expectation is non-zero only if $\bar k = -k$, so that
the sum of all indices appearing under the expectation vanishes.)
At this stage, it is useful to note that one has the identity
\begin{equ}[e:breakPair]
\E \bigl(\bar X_{\eps,k-m}^\3(t) \bar X_{\eps,m}^\1(t)
\bar X_{\eps,-k-\bar m}^\3(t) \bar X_{\eps,\bar m}^\1(t)\bigr)
= \sum_{P \in \CP^\5} \sum_{L \in \LL^\tau_{P;k,m,\bar m}} \!\!\!\!\!C_\eps(L)\,\CF^\5(P,L)\;,
\end{equ}
where, similarly to \eref{e:defG}, we set
\begin{equ}
\CF^\tau(P,C) =   \Bigl(\prod_{e \in \bar \CT} C_{e}\Bigr) \int 
\exp \Bigl(- \sum_{e\in \bar \CE} |C_e|^2 |\delta T_e| \Bigr)\,\mu_0(dT)\;,
\end{equ}
and where we denoted by $\LL^\tau_{P;k,m,\bar m}$ the set of all cycles in $\LL^\tau_P$ such that
the edge $\bar e$ has value $k$ and such that the two edges in $\CE\setminus\CT$ adjacent to $\bar e$ have values $m$
and $\bar m$ respectively. (By convention, we orient $\bar e$ from the edge with label $m$ to the one with label $\bar m$.)

Here, $\CE$ and $\CT$ are as before the set of edges and the spanning tree
for the graph associated to the tree $\bigtree$ and the pairing $P$ in exactly the same 
way as in Section~\ref{sec:constructGraph}. Also,
just as in that section, $\bar \CE$ and $\bar \CT$ denote the same sets, but with the distinguished edge 
$\bar e$ removed. 

\begin{remark}
Since, at least at this stage, we fix the labels $k$, $m$ and $\bar m$, we cannot necessarily
replace $\CF^\tau$ by $\CF^\tau_\sym$. This will be the source of some minor headache later on.
\end{remark}

Inserting \eref{e:breakPair} into \eref{e:boundXX12}, we have the identity
\begin{equ}
\E \bigl(\XX_t^{(1),\eps}(\delta)\bigr)^2 =  \sum_{P \in \CP^\5} \XX_P^\eps(\delta)\;.
\end{equ}
The quantity $\XX_P^\eps$ appearing in this decomposition is defined by
\begin{equ}[e:defXPeps]
 \XX_P^\eps(\delta) = \sum_{k,m,\bar m \in \Z_\star}\sum_{L \in \LL^\tau_{P;k,m,\bar m}}\!\!\!\!\!C_\eps(L)\, \CF^\5(P,L) \bigl(g_k(\delta) - g_m(\delta)\bigr)\bigl(g_{-k}(\delta) - g_{\bar m}(\delta)\bigr)\;,
\end{equ}
where we used the shorthand notation
\begin{equ}
g_k(\delta) \eqdef {e^{ik\delta}-1\over k}\;.
\end{equ}
We now proceed to bounding $\XX_P(\delta)$ for each pairing $P$.

\begin{proposition}\label{prop:boundXXs}
For every $P\in \CP_s^\5$ and every $\kappa > 0$,  the bound
\begin{equ}[e:boundXPeps]
\XX_P^\eps(\delta) \lesssim \delta^{2-2\kappa}\;,
\end{equ}
holds uniformly over $\eps,\delta \in (0,1]$.
\end{proposition}

\begin{proof}
It follows from Proposition~\ref{prop:bounddiff} below that one has the bound
\begin{equ}
|g_{k}(\delta) - g_m(\delta)| \le 2|k-m|\delta^2\;.
\end{equ}
Since one also has the trivial bound $|g_k(\delta)| \le 2/|k|$, it follows that 
one has for every $\kappa > 0$ the bound
\begin{equ}[e:bounddg]
|g_{k}(\delta) - g_m(\delta)| \lesssim \delta^{2-2\kappa} |k-m|^{1-\kappa} \bigl(|k|^{-1} + |m|^{-1}\bigr)^\kappa\;.
\end{equ}

Furthermore, by the definition of $\CP_s$, every $P \in  \CP_s^\5$ is associated to a $\bar P \in \CP^\3$
such that the remaining pair in $P$ connects the two leaves attached to the roots.
Similarly, for every $L \in \LL^\5_{P;k,m,\bar m}$ there exists a cycle $\Pi L \in \LL^\3_{\bar P}$ such that $\Pi L$ agrees with $L$ on the common parts of the two graphs and such that $(\Pi L)_{\bar e} = k-m$.
Furthermore, $\Pi$ is a bijection between $\LL^\5_{P;k,m,\bar m}$ and the subset $\LL^\3_{\bar P;k-m}$ of $\LL^\3_{\bar P}$ that we just mentioned. 

We also note that if $P \in \CP_s^\5$, one has $\LL^\5_{P;k,m,\bar m} = \emptyset$ unless 
$\bar m = -m$, since otherwise the cycles in that set are not adapted to the pairing $P$. As a consequence, 
one has the identity
\begin{equs}
\sum_{L \in \LL^\tau_{P;k,m,\bar m}} &C_\eps(L)\, \CF^\5(P,C) \\
&= (k-m)^2\delta_{m, -\bar m} \sum_{\bar L \in \LL^\3_{\bar P;k-m}} \phi(\eps(k-m))\,C_\eps(\bar L)\, \CK^\3_\sym(\bar P,\bar L;0)\;.
\end{equs}
Combining this with the bound \eref{e:bounddg} and summing the resulting expression over $k$ and $m$, 
we obtain, for every $\kappa > 0$ the bound
\begin{equ}[e:boundXdeltas]
\XX_P^\eps(\delta) \lesssim \delta^{2-2\kappa} \sum_{\bar L \in \LL^\3_{\bar P}} \sum_{m \in Z_\star\atop m \neq \bar L_{\bar e}} |\bar L_{\bar e}|^{3-\kappa} |\CK^\3_\sym(\bar P,\bar L;0)| \bigl(|m+\bar L_{\bar e}|^{-1-\kappa} + |m|^{-1-\kappa}\bigr)\;,
\end{equ}
uniformly over $\eps$ and $\delta$.
Since, by Proposition~\ref{prop:boundX3}, the quantity
\begin{equ}
 |\bar L_{\bar e}|^{3-\kappa} |\CK^\3_\sym(\bar P,\bar L;0)|\;,
\end{equ}
is summable over $\LL^\3_{\bar P}$, the second term in \eref{e:boundXdeltas}
satisfies the requested bound.
The claim now follows from the fact that the first term is essentially identical to the second one, as can be seen by 
performing the substitution $m \mapsto m-\bar L_{\bar e}$.
\end{proof}

We now proceed to bound the  terms corresponding to the remaining pairings, which we subdivide
into three different classes.
Recall that $\ell_\ell^P(\tau)$ denotes the set of leaves that are part of a small loop of type $1$.
Denote furthermore by $u$ and $\bar u$ the two leaves that are attached to either of the roots of the two
copies of $\bigtree$. We then define $\CP^k(\bigtree)$ with $k\in \{0,1,2\}$ as consisting of those
pairings in $\CP^\5 \setminus \CP_s^\5$ such that $|\{u,\bar u\} \cap \ell_\ell^P(\bigtree)| = k$. We then have:

\begin{proposition}\label{prop:boundXX0}
For every $P\in \CP^0(\bigtree)$ and every $\kappa > 0$, there exists $C$ such that the bound
\eref{e:boundXPeps} holds for $\eps \in (0,1]$.
\end{proposition}

\begin{proof}
For such a pairing, one has the identity
\begin{equ}
\sum_{L \in \LL^\tau_{P;k,m,\bar m}}C_\eps(L) \CF^\tau(P,L) = \sum_{L \in \LL^\tau_{P;k,m,\bar m}}C_\eps(L)\CF^\tau_\sym(P,L)\;,
\end{equ}
where as before we denote by $\CF^\tau_\sym$ the symmetrised version of $\CF^\tau$ under the equivalence relation $\sims$.
(Here and below, we sometimes use $\tau$ instead of $\bigtree$ in order to make notations slightly less
heavy.) Indeed, if $P \in \CP^0(\bigtree)$ and $L \in \LL^\tau_{P;k,m,\bar m}$,
then $\bar L \in \LL^\tau_{P;k,m,\bar m}$ for every $\bar L$ with $\bar L \sims L$, since none of the 
loops whose labels can change sign touches the distinguished edge $\bar e$.

Recall furthermore that $|g_k(\delta)| \le \delta$ so that, for every $\kappa \in [0,1]$, one has the bound 
\begin{equs}
\delta^{2\kappa-2} |\bigl(g_k(\delta) - g_m(\delta)\bigr)\bigl(g_{-k}(\delta) - g_{\bar m}(\delta)\bigr)| &\lesssim |k|^{-2\kappa} + \bigl(|m|^{-2\kappa}\wedge |\bar m|^{-2\kappa}\bigr)\qquad\\
&\lesssim |k|^{-2\kappa} + |m|^{-\kappa_1} |\bar m|^{-\kappa_2}\;,\label{e:boundgdiff}
\end{equs}
where  $\kappa_1$ and $\kappa_2$ are any two positive exponents  with $\kappa_1 + \kappa_2 = 2\kappa$.
Recalling the definition of $\XX_P^\eps$, we see that the requested bound follows, provided that we are
able to show that, for every pairing $P \in \CP^0(\bigtree)$ and every $\kappa> 0$, one can find exponents
$\kappa_1$ and $\kappa_2$ as above such that
\begin{equ}
\sum_{k,m,\bar m\in \Z_\star} \sum_{L \in \LL^\tau_{P;k,m,\bar m}} \bigl(|k|^{-2\kappa} + |m|^{-\kappa_1} |\bar m|^{-\kappa_2}\bigr) |\CF^\5_\sym(P,L;0)| < \infty\;.
\end{equ}

It follows from Proposition~\ref{prop:boundK5} that 
\begin{equ}[e:boundDone5]
\sum_{k,m,\bar m\in \Z_\star} \sum_{L \in \LL^\tau_{P;k,m,\bar m}} |k|^{-2\kappa} |\CF^\5_\sym(P,L;0)| < \infty\;,
\end{equ}
for every $\kappa > 0$, so that
it remains to show that
\begin{equ}[e:boundWanted5]
\sum_{k,m,\bar m\in \Z_\star} \sum_{L \in \LL^\tau_{P;k,m,\bar m}} |m|^{-\kappa_1}|\bar m|^{-\kappa_2} \CF^\5_\sym(P,L;0) < \infty\;.
\end{equ}

For  $P\in \CP^0(\bigtree)$ and  $\kappa > 0$, we now define a set $\bar \GG_\kappa^{\tau,0}(P)$ of 
weightings
by following the construction in Section~\ref{sec:constructGraph}. Comparing \eref{e:boundWanted5} 
and \eref{e:boundDone5} we see that, in order to obtain a bound on \eref{e:boundWanted5}, the
only difference in the construction is that, in Step~2, we 
give weight $0$ to the distinguished edge $\bar e$, and weights $\kappa_1$ and $\kappa_2$
to the two edges adjacent to $\bar e$ that do not belong to the spanning tree $\CT$.

Retracing the proof of Theorem~\ref{theo:graphKernel} we see that, if we can exhibit an element in $\bar \GG_\kappa^{\tau,0}(P)$
such that Algorithm~\ref{algo} yields a loop-free graph, then the bound \eref{e:boundWanted5} holds.
It can be checked in a straightforward way that the following weightings do indeed belong to $\bar \GG_\kappa^{\tau,0}(P)$:
\begin{equs}
\mhpastefig{bigProofM1}\\
\mhpastefig{bigProofM2}\\
\mhpastefig{bigProofM3}
\end{equs}
Here, we use again the same convention as in the proof of Proposition~\ref{prop:boundK5}. 
The only difference is that the dashed lines
correspond to a weight $2\kappa$ if only one dashed line is present in a given graph, and $\kappa$ if two such 
lines are present. Again, it is straightforward to verify that Algorithm~\ref{algo} does indeed yield a loop-free 
graph for every $\kappa > 0$ in each instance, thus proving the claim.
\end{proof}

We now turn to the set $\CP^1$ of pairings that have one small loop touching the distinguished edge.
There are actually only two such pairings (modulo isometries), corresponding to the $8$th and the
$10$th pairing in \eref{e:biggraph}.

\begin{proposition}\label{prop:boundXX1}
For every $P\in \CP^1(\bigtree)$ and every $\kappa > 0$, the bound
\eref{e:boundXPeps} holds uniformly over $\eps,\delta \in (0,1]$.
\end{proposition}

\begin{proof}
By the definition of $\CP^1$, there is one small loop touching the distinguished edge. We can then
fix our naming convention in such a way that the edge of this loop that doesn't belong to the spanning tree $\CT$
is the one labelled $m$. With this convention, we then introduce a ``reflection map'' $\CR \colon \LL^\tau_{P;k,m,\bar m} \to \LL^\tau_{P;k,-m,\bar m}$,
which is the bijection which changes the sign of $m$ and adjust the label of the other edge of the small loop in such a way that no other label is changed.
In other words, if we decompose elements in $\LL^\tau_{P;k,m,\bar m}$ with respect to the integral basis
associated with the spanning tree $\CT$, then $\CR$ is the map that changes the sign of the coefficient in front of
the elementary cycle spanning the small loop.

With this notation, although we cannot quite replace $\CF^\tau$ by $\CF^\tau_\sym$ in \eref{e:defXPeps}, we nevertheless have the identity
\begin{equ}
\sum_{L \in \LL^\tau_{P;k,m,\bar m}} \bigl(\CF^\5(P,L) + \CF^\5(P,\CR L)\bigr) = 2\sum_{L \in \LL^\tau_{P;k,m,\bar m}}\CF^\5_\sym(P,L)\;,
\end{equ}

Using this identity, we can decompose \eref{e:defXPeps} into a symmetric part and a remainder, which yields the expression
\begin{equs}
 \XX_P^\eps(\delta) &= \sum_{k,m,\bar m \in \Z_\star}\Bigl(\sum_{L \in \LL^\tau_{P;k,m,\bar m}}\!\!C_\eps(L)\, \CF^\5_\sym(P,L;0) \Bigr)g_k(\delta)\bigl(g_{-k}(\delta) - g_{\bar m}(\delta)\bigr)\\
&\quad - \sum_{k,m,\bar m \in \Z_\star}\Bigl(\sum_{L \in \LL^\tau_{P;k,m,\bar m}} \!\!C_\eps(L)\, \CF^\5(P,L;0) \Bigr)g_m(\delta)\bigl(g_{-k}(\delta) - g_{\bar m}(\delta)\bigr)\;.
\end{equs}
The first term in this identity is bounded by
\begin{equ}
\sum_{k,m,\bar m \in \Z_\star}\sum_{L \in \LL^\tau_{P;k,m,\bar m}}\bigl| \CF^\5_\sym(P,L;0) \bigr| \,|k|^{-2\kappa}\delta^{2-2\kappa}\;,
\end{equ}
which in turn is bounded by $C \delta^{2-2\kappa}$ by Proposition~\ref{prop:boundK5}.

Similarly, the second term is bounded by
\begin{equ}[e:boundXX1]
\delta^{2-2\kappa} \sum_{k,m,\bar m \in \Z_\star}\Bigl|\sum_{L \in \LL^\tau_{P;k,m,\bar m}} \!\!C_\eps(L)\,\CF^\5(P,L;0) \Bigr|\,|m|^{-2\kappa}\;,
\end{equ}
which is not quite covered by Proposition~\ref{prop:boundK5} since the terms are weighted by a power of $m$, the label of the edge 
within the small loop, instead of $k$, the label of the distinguished edge.

Similarly to Proposition~\ref{prop:boundXX0}, we now define a set $\bar \GG_\kappa^{\tau,1}(P)$ of weightings
by again following the construction in Section~\ref{sec:constructGraph}. This time, in Step~2, we 
give weight $0$ to the distinguished edge $\bar e$, and we give instead an additional weight $2\kappa$ to the edge in 
$\CE \setminus \CT$
adjacent to $\bar e$ that belongs to a small loop of type $1$. Since the kernel $\CF$ is not symmetrised under $\CR$,
another difference in the construction of $\bar \GG_\kappa^{\tau,1}(P)$ is that in Step~3, we do not treat the loop that
touches $\bar e$.

Retracing once again the proof of Theorem~\ref{theo:graphKernel} we see that, if we can exhibit an element in $\bar \GG_\kappa^{\tau,1}(P)$
such that Algorithm~\ref{algo} yields a loop-free graph, then the quantity \eref{e:boundXX1} is bounded by $C \delta^{2-2\kappa}$
for some $C <\infty$, which is the desired bound.
Note now that, in the proof of Proposition~\ref{prop:boundK5}, the weightings that are exhibited for pairings in $\CP^1(\bigtree)$ locally look like
\begin{equ}[e:loop1]
\mhpastefig{looplabelsimplekappa}
\end{equ}
It is straightforward to check that under the rules for constructing $\bar \GG_\kappa^{\tau,1}(P)$, these configurations can
instead be weighted in the following way:
\begin{equ}[e:loop2]
\mhpastefig{looplabelsimplezero}
\end{equ}
Since Algorithm~\ref{algo} has the same effect on both of these configurations, at least if $\kappa$ is small enough 
(one can just replace them by one single edge with weight $2\kappa$), the summability of \eref{e:boundXX1} follows in the same way as before.
\end{proof}

Finally, we also have the bound

\begin{proposition}\label{prop:boundXX2}
For every $P\in \CP^2(\bigtree)$ and every $\kappa > 0$,  the bound
\eref{e:boundXPeps} holds uniformly over $\eps,\delta \in (0,1]$.
\end{proposition}

\begin{proof}
The proof is virtually identical to that of Proposition~\ref{prop:boundXX1}, so we only highlight the differences. 
We now have two small loops
touching the distinguished edge, so that the expression for $\XX_P^\eps(\delta)$ is broken into four terms, since each of the
two loops can either be symmetrised or not. Then, as before, the symmetrised loops are associated to weightings
as in the proof
of Proposition~\ref{prop:boundK5}, while the non-symmetrised loops are associated to weightings 
as in \eref{e:loop2}.
\end{proof}

Combining these results, we obtain the following result:

\begin{proposition}\label{prop:convergenceXX1}
For every $\bar \kappa > 0$, one has the bound
\begin{equ}
\sup_{\eps \in (0,1]} \sup_{t \in \R} \E \|\XX_t^{(1),\eps}\|_{1-\bar \kappa}^2 < \infty\;.
\end{equ}
\end{proposition}

\begin{proof}
Since $\XX_t^{(1),\eps}$ belongs to a finite union of fixed Wiener chaoses, it follows 
from Propositions~\ref{prop:boundXXs}--\ref{prop:boundXX2} that, for every $\kappa >0$ and every $p>0$, the bound
\begin{equ}
\E \bigr|\XX_t^{(1),\eps}(x,y)\bigr|^p \lesssim |x-y|^{p(1-\kappa)}\;,
\end{equ}
holds uniformly in $\eps\in (0,1]$, $t \in \R$, and $x,y \in S^1$.
The requested bound then follows from \cite[Cor.~4]{Max}.
\end{proof}

We now consider the term $\XX_t^{(2),\eps}(\delta)$ which can be treated in a similar manner. 
One has:

\begin{proposition}\label{prop:convergenceXX2}
The conclusions of Proposition~\ref{prop:convergenceXX1} holds,
with $(1)$ replaced by $(2)$ throughout.
\end{proposition}

\begin{proof}
In the same way as before, it suffices to show that 
\begin{equ}
\E \bigl(\XX_t^{(2),\eps}(\delta)\bigr)^2 \lesssim \delta^{2-2\kappa}\;,
\end{equ}
uniformly over $\eps, \delta \in (0,1]$.
In exactly the same way as before, we have the identity
\begin{equ}[e:XXX2]
\E \bigl(\XX_t^{(2),\eps}(\delta)\bigr)^2 = \sum_{m, \bar m\in \Z_\star} \E \bigl(\bar X_{\eps,-m}^\3(t) \bar X_{\eps,m}^\1(t)
\bar X_{\eps,-\bar m}^\3(t) \bar X_{\eps,\bar m}^\1(t)\bigr)g_m(\delta) g_{\bar m}(\delta)\;,
\end{equ}
with
\begin{equ}[e:breakPair2]
\E \bigl(\bar X_{\eps,-m}^\3(t) \bar X_{\eps,m}^\1(t)
\bar X_{\eps,-\bar m}^\3(t) \bar X_{\eps,\bar m}^\1(t)\bigr)
= \sum_{P \in \CP^\tau} \sum_{L \in \LL^\tau_{P;m,\bar m}} \!\!\!\!C_\eps(L)\,\CF^\tau(P,L)\;.
\end{equ}
Here, by analogy with the previous notations of this section, we denoted by $\LL^\tau_{P;m,\bar m}$ the set of all cycles in $\LL^\tau_{P;0}$ such that
the edge $\bar e$ (and only that edge) has value $0$  and the edges in $\CE\setminus\CT$ adjacent to $\bar e$ have values 
$m$ and $\bar m$ respectively. 

Similarly to before, for any given $\kappa \in (0,{1\over 2}]$, and for any two positive exponents $\kappa_1$ and $\kappa_2$ 
with $\kappa_1+\kappa_2 = 2\kappa$, we have the bound $|g_m(\delta)g_{\bar m}(\delta)| \le \delta^{2-2\kappa} |m|^{-\kappa_1}|\bar m|^{-\kappa_2}$, 
so that the statement reduces to the proof that
\begin{equ}[e:wantedBoundMM]
\sum_{m, \bar m \in \Z_\star}\Bigl(\sum_{L \in \LL^\tau_{P;m,\bar m}} \!\!C_\eps(L)\,\CF^\tau(P,L)\Bigr)|m|^{-\kappa_1}|\bar m|^{-\kappa_2} < \infty\;,
\end{equ}
uniformly over $\eps \in (0,1]$.

The proof follows again the same lines as before.
For every pairing $P \in \CP^\5$, we construct
a set $\hat \GG_\kappa^{\tau}(P)$ of weightings
by following the construction in Section~\ref{sec:constructGraph}. Again,
we give $\bar e$ the weight $0$, and weigh instead the two edges in $\CE \setminus \CT$ adjacent to $\bar e$
by $\kappa_1$ and $\kappa_2$ respectively. Note that, since the inner sum in \eref{e:wantedBoundMM} has $m$ and $\bar m$ fixed,
we cannot symmetrise small loops that touch the distinguished edge. As a consequence, in Step~3 of the construction, we only
treat those small loops that do not touch $\bar e$. This however is not a problem, since the loops touching $\bar e$ receive an 
additional weight $\kappa_i$ anyway, which has the same effect.
Furthermore, since this time we restrict our sum to cycles that associate to $\bar e$ the value $0$, 
we remove the edge $\bar e$ from the graph $(\CG,\CE)$ before applying Algorithm~\ref{algo}.

Retracing the proof of Theorem~\ref{theo:graphKernel} we see once again that if, for every pairing $P\in \CP^\tau$, we can exhibit an element 
in $\hat \GG_\kappa^{\tau}(P)$
such that Algorithm~\ref{algo} yields a loop-free graph, then the requested bound holds.
On the other hand, as far as the outcome of Algorithm~\ref{algo} is concerned, deleting an edge (not contracting it!) is the same as
giving it a weight larger than the largest weight in the graph. As a consequence, every weighting constructed in
Propositions~\ref{prop:boundXX0}--\ref{prop:boundXX2} also yields a weighting in $\hat \GG_\kappa^{\tau}(P)$ that is summable.
Furthermore, pairings in $\CP_s^\tau$ can be treated ``by hand'' in very much the same way as in Proposition~\ref{prop:boundXXs}.

There still remains one case to verify though. Previously, we only considered pairings such that there exists at least one pair connecting
the two instances of the tree $\tau$. This was precisely because any labelling compatible with a pairing that doesn't have this property
would associate $0$ to the root vertex, which we always ruled out. This time however, this is precisely the situation that we are 
considering, so that we cannot make this restriction. As a consequence, we also have to consider the following pairing:
\begin{equ}[e:infiniteLoop]
\mhpastefig{proofSpecial2}
\end{equ}
Here, as before, the two dashed edges have weight $\kappa$, while the loops and the remaining 
edges are weighted in such a way that, after applying one step of  Algorithm~\ref{algo}, they reduce to an edge with 
weight $1$. This weighted graph is summable by applying Algorithm~\ref{algo}, which then concludes the proof in the same way as
in Proposition~\ref{prop:convergenceXX1}.
\end{proof}

\begin{remark}
The last step is the only step in the proof of Proposition~\ref{prop:convergenceXX2} where the 
additional weight $\kappa$ is actually needed. In all other cases,
erasing $\bar e$ significantly improves the summability properties of the resulting graphs, so that the resulting expressions
would already have been summable with $g_m(\delta) = \delta$. This shows that it is precisely the presence of the
graph \eref{e:infiniteLoop} that requires the projection operator $\Pi_0$ in the definition \eref{e:defXXbasic}. Indeed, if
we remove $\Pi_0$ from this expression, we see that the difference results in a term like \eref{e:XXX2},
but with $g_m(\delta) = \delta$, so that this last pairing would result in a logarithmic divergence.
\end{remark}

\subsection{Construction of the area process}

To conclude the construction of the map $\Psi$, we show that the sequence of 
processes $\YY_\eps$ defined by \eref{e:defYY} does indeed have a limit as $\eps \to 0$:

\begin{proposition}\label{prop:convYY}
Let $\Phi_\eps$, $Y^\1_\eps$ and $\YY_\eps$ be given by \eref{e:defPhieps} 
and \eref{e:defYY}. Then, there exists a process $\YY$ such that
$\YY_\eps \to \YY$ in probability in the space $\CC(\R, \CC_2^{1-\delta})$ for every $\delta > 0$. 
\end{proposition}

\begin{proof}
The argument is essentially the same as the one given in \cite{BurgersRough}, which in turn relies very heavily
on the results in \cite{GaussI,PeterBook}, so we only explain the main steps and refer to \cite{BurgersRough}
for more details.

Recall that, by definition, the Fourier components of $Y^\1_\eps$ (for $k\neq 0$) are given by
\begin{equ}
Y^\1_{\eps,k}(t) = \sqrt 2 \phi(k\eps) \int_{-\infty}^t e^{-k^2(t-s)}\,dW_{k}(s)\;,
\end{equ}
where the $W_k$ are independent normal complex-valued Wiener processes satisfying the
reality condition $W_{-k,t} = \bar W_{k,t}$. Similarly, the Fourier components of $\Phi_\eps$ are given by
\begin{equ}
\Phi_{\eps,k}(t) = k^2 \int_{-\infty}^t e^{-k^2 (t-s)} Y^\1_{k,s}\,ds = \sqrt 2 \phi(k\eps) \int_{-\infty}^t k^2(t-s)\, e^{-k^2 (t-s)} dW_k(s)\;.
\end{equ}

An explicit calculation (boiling down to the fact that $\int_0^\infty (x-{1\over 2})e^{-2x}\,dx = 0$) then shows that if we 
define a process $\tilde \Phi_\eps$ by
\begin{equ}
\Phi_\eps = {1\over 2} Y_\eps^\1 + \tilde \Phi_\eps\;,
\end{equ}
then, for every \textit{fixed} $t \in \R$, the processes $\tilde \Phi_\eps(t)$ and $Y^\1_\eps(t)$ are independent. 
It then follows from \cite{PeterBook} that $\YY_\eps(t) \to \YY(t)$ in probability in $\CC_2^{1-\delta}$ for every $\delta > 0$.
The convergence in $\CC([-T,T], \CC_2^{1-\delta})$ then follows from Kolmogorov's continuity criterion 
by \cite[Theorem~1]{GaussI}.
\end{proof}

We now finally have all the ingredients in place for the proof of Theorem~\ref{theo:convYeps}.

\begin{proof}[of Theorem~\ref{theo:convYeps}]
The case $\tau = \bullet$ is standard, see for example \cite{DaPrato-Zabczyk92}. The cases
$\tau \in \{\2,\3,\4,\5\}$ follow by combining the results of Section~\ref{sec:constructX}, which yield the
convergence of the processes $X^\tau$, with the results of Section~\ref{sec:constantMode}, which furthermore yield the
convergence of the constant Fourier modes of $Y^\tau$. 

The remaining cases are treated by induction. Assume that $\tau = [\tau_1,\tau_2]$ with $\alpha_{\tau_1} \le \alpha_{\tau_2}$.
The case $\tau_1 = \bullet$ will be treated separately. If $\tau_1 \neq \bullet$, then we have $\alpha_{\tau_1} \ge 1$ and $\alpha_{\tau_2} > 1$ (since the case of both being equal to $1$ corresponds to the tree $\nicetree$ which was already covered).
As a consequence, we can use the induction hypothesis, combined with Proposition~\ref{prop:Holder} 
to conclude that the term
$\d_x Y^{\tau_1}_\eps\,\d_x Y^{\tau_2}_\eps$ converges to $\d_x Y^{\tau_1}\,\d_x Y^{\tau_2}$ in $\CC(\R,\CC^{\alpha_\tau-2-\delta})$ for every $\delta> 0$. The claim then follows from the definition of $Y^\tau$, combined with 
Proposition~\ref{prop:interpolation}.

It remains to treat the case when $\tau = [\bullet, \bar \tau]$ with $\alpha_{\bar \tau} > 1$. For these, we actually
show a stronger statement, namely that,  
$\d_x Y^\tau_\eps \to \d_x Y^\tau$ as a rough path controlled by $\Phi_\eps$ with derivative process 
$\d_x Y^{\bar \tau}_\eps$. By the results of this section, this is true for $\tau = \bigtree$, so that it suffices to
prove the statement for the remaining trees of this form.
Again, there are two cases. If $\alpha_{\bar \tau} \ge 1$, we view $\d_x Y^{\bar \tau}_\eps$ as a rough path controlled
by $\Phi_\eps$ with vanishing derivative process, so that Proposition~\ref{prop:boundRemainder} yields the
desired statement. If $\bar \tau$ is itself of the form $\bar \tau = [\bullet, \kappa]$, then we 
know by the induction hypothesis that $\d_x Y^{\bar \tau}_\eps$ is controlled by $\Phi_\eps$ with 
derivative process $\d_x Y^{\kappa}_\eps$ and bounds that are uniform in $\eps$.  The claim then follows
again from Proposition~\ref{prop:boundRemainder}.
\end{proof}

\begin{appendix}
\section{Useful computations}

\subsection{Wiener chaos}\label{sec:chaos}

In this section, we assume that we work on a probability space $(\Omega, \P, \CF)$ equipped
with a Gaussian structure. In other words, there exists a separable Hilbert space
$\CH$ and an isometry $\iota\colon \CH \to L^2(\Omega,\P)$ such that $\iota(h)$ is a centred Gaussian
random variable for every $h\in\CH$.

Denote now by $\CP_{k,m}$ the set of all polynomials of degree $k$ in $m$ variables.
We then write $\CI_k$ for the closure in $L^2(\Omega,\P)$ of the set 
\begin{equ}
\{P(\iota(h_1),\ldots,\iota(h_m))\,:\, P \in \CP_{k,m}\,,\; h_j \in \CH\;,\; m \ge 1\}\;.
\end{equ}
Given a separable Banach space $\CB$, we also write $\CI_k(\CB)$ for the same
space, but where $P$ is $\CB$-valued. The space $\CI_k$ is the union of the $n$th Wiener
chaoses for $\Omega$ with $n \le k$. Since the precise definition of the $n$th Wiener chaos over
a given Gaussian structure is only marginally relevant for this article, we refer to the monograph \cite{Nualart}
for more details.

A very useful fact is given by the following lemma, which follows from the hypercontractivity
of the Ornstein-Uhlenbeck semigroup on $L^2(\Omega)$ \cite{Nualart} and is also known as Nelson's estimate:
\begin{lemma}\label{lem:chaos}
Let $(\Omega,\P,\CF)$ be a Gaussian probability space, let $\CB$ be a separable Banach space,
and denote $\CI_k(\CB)$ as before. Then, for every $k,p \ge 1$ there exist constants $C_{k,p}$ such that
\begin{equ}
\E |F|^{2p} \le C_{k,p} \bigl(\E |F|^2\bigr)^p\;,
\end{equ} 
for every every $\CB$-valued random variable $F \in \CI_k(\CB)$.
\end{lemma}

Our main application of this estimate is the following bound, which loosely speaking states that
in the context of processes taking values in a fixed Wiener chaos, Sobolev regularity for a given index
often implies H\"older regularity for the same index.

\begin{proposition}\label{prop:generalConvergence}
Let $\JJ$ be a countable index set, let $T>0$, and let $\{g_\kappa\}_{\kappa \in \JJ}$ be a family
of Lipschitz continuous functions such that.
\begin{equ}
\|g_\kappa\|_\infty \le 1\;,\qquad \|g_\kappa\|_1 \le G_\kappa\;,
\end{equ}
for some  $G_\kappa \ge 1$. In general, we assume that the $g_\kappa$ are complex-valued and that there 
is an involution $\iota\colon \JJ \to \JJ$ such that $g_{\iota \kappa} = \bar g_\kappa$. 

Let furthermore $\{f_\kappa\}_{\kappa\in\JJ}$ be a family of continuous stochastic processes belonging to 
$\CI_k$ for some fixed value $k \in \N$, and write
\begin{equ}
F_{\kappa\eta}(t) =  \E f_\kappa(t) f_\eta(t)\;,\quad
\hat F_{\kappa\eta}(s,t) =  \E \bigl(f_\kappa(t) - f_\kappa(s)\bigr)\bigl(f_\eta(t) - f_\eta(s)\bigr)\;.
\end{equ}
We also assume that $f_{\iota \kappa} = \bar f_\kappa$.
Finally, let $\{C_\eps\}_{\eps \in (0,1]}$ be a family of functions $C_\eps \colon \JJ \to [0,1]$
such that
\begin{claim}
\item one has $C_\eps(\kappa) > C_{\bar\eps}(\kappa)$ for $\eps < \bar\eps$,
\item for every $\eps > 0$, the set $\{\kappa\,:\, C_\eps(\kappa) \neq 0\}$ is finite,
\item for every $\kappa \in \JJ$, one has $\lim_{\eps \to 0} C_\eps(\kappa) = 1$.
\end{claim}
For every $\eps > 0$, let $F_\eps$ be the stochastic process defined by
\begin{equ}[e:sum]
F_\eps(x,t) = \sum_{\kappa\in \JJ} C_\eps(\kappa) f_\kappa(t) g_\kappa(x)\;,
\end{equ}
and assume that there exists $\alpha \in (0,1)$ and $\beta \ge 0$ such that
\begin{equs}[e:assumFalpha]
\sum_{\kappa,\eta \in \JJ} \sup_{t \in [0,T]} |G_\kappa|^\alpha |G_\eta|^\alpha |F_{\kappa \eta}(t)| &< \infty\;,\\
\sum_{\kappa,\eta \in \JJ} \sup_{s,t \in [0,T]} {|\hat F_{\kappa \eta}(s,t)| \over |t-s|^{2\beta}} &< \infty\;,
\end{equs}
Then, for every $\gamma < \alpha$ and $\delta < \beta$, 
there exists a process $F$ taking values in $\CB_{\gamma,\delta} \eqdef \CC([0,T],\CC^\gamma) \cap \CC^\delta([0,T],\CC)$
and such that  
$F_\eps \to F$ in $L^2(\Omega,\CP,\CB_{\gamma,\delta})$. 
\end{proposition}

\begin{proof}
It suffices to show that our conditions imply that the sequence $\{F_\eps\}$ is Cauchy 
in $L^2(\Omega,\CP,\CB_{\gamma,\delta})$. Fix $0 < \eps < \bar \eps$ and write
$\Cd(\kappa)$ as a shorthand for $C_\eps(\kappa)- C_{\bar \eps}(\kappa)$, and
similarly $\Fd = F_{\eps} - F_{\bar \eps}$. By our assumption on $C_\eps$, we have $\Cd(\kappa) \ge 0$
for every $\kappa$.

We furthermore write
\begin{equ}
F_{\kappa\eta} \eqdef \sup_{t \in [0,T]} |F_{\kappa\eta}(t)|\;,\qquad
\hat F_{\kappa\eta} \eqdef \sup_{s,t \in [0,T]} {|\hat F_{\kappa\eta}(s,t)| \over |t-s|^{2\beta}}\;.
\end{equ}
An elementary calculation then shows that
\begin{equs}
\E |\Fd(x,t)|^2 &= \sum_{\kappa,\eta \in \JJ} \Cd(\kappa) \Cd(\eta) \E f_\kappa(t)\bar f_\eta(t) g_\kappa(x)\bar g_\eta(x)\\
&=\sum_{\kappa,\eta \in \JJ} \Cd(\kappa) \Cd(\eta) \E f_\kappa(t) f_{\iota \eta}(t) g_\kappa(x)g_{\iota \eta}(x)\\
&=\sum_{\kappa,\eta \in \JJ} \Cd(\kappa) \Cd(\iota \eta) F_{\kappa\eta}(t) g_\kappa(x)g_{\eta}(x)\\
&\le \sum_{\kappa,\eta \in \JJ} \Cd(\kappa) \Cd(\iota \eta) F_{\kappa\eta}\;.
\end{equs}
Similarly, we have the bound
\begin{equs}
\E |\Fd(x,t) - \Fd(y,t)|^2 &= \sum_{\kappa,\eta \in \JJ} \Cd(\kappa) \Cd(\eta) F_{\kappa\eta}(t) \bigl(g_\kappa(x) - g_\kappa(y)\bigr)\bigl(g_\eta(x) - g_\eta(y)\bigr)\\
&\le  |x-y|^{2\alpha}\sum_{\kappa,\eta \in \JJ} \Cd(\kappa) \Cd(\eta) F_{\kappa\eta} |G_\kappa|^\alpha |G_\eta|^\alpha\;,
\end{equs}
as well as
\begin{equs}
\E |\Fd(x,t) - \Fd(x,s)|^2 &= \sum_{\kappa,\eta \in \JJ} \Cd(\kappa) \Cd(\eta) \hat F_{\kappa\eta}(s,t) g_\kappa(x)g_\eta(x)\\
&\le  |t-s|^{2\beta} \sum_{\kappa,\eta \in \JJ} \Cd(\kappa) \Cd(\eta) |\hat F_{\kappa\eta}|\;.\label{e:boundFt}
\end{equs}
Making use of Lebesgue's dominated 
convergence theorem, we deduce that there exists a sequence of constants $K_{\bar \eps}$ with
$\lim_{\bar \eps \to 0}K_{\bar \eps} = 0$
such that the bounds
\begin{equ}
\E |\Fd(x,t) - \Fd(y,t)|^2 \le K_{\bar \eps}|x-y|^{2\alpha}\;,\quad
\E |\Fd(x,t) - \Fd(x,s)|^2 \le K_{\bar \eps}|t-s|^{2\beta}\;,
\end{equ}
hold uniformly for $\eps < \bar \eps$.
In particular, it follows from Lemma~\ref{lem:chaos} that bounds with the same homogeneity also 
hold for the $p$th moment for arbitrarily large $p$.
It then follows from a straightforward modification  of Kolmogorov's continuity criterion that
\begin{equ}
\lim_{\bar \eps \to 0} \sup_{\eps < \bar \eps}\E \|\Fd\|^2_{\gamma,\delta} = 0\;,
\end{equ}
where $\|\cdot\|_{\gamma,\delta}$ is the norm in $\CB_{\gamma,\delta}$. The claim now follows at once.
\end{proof}

\subsection{Bounds on simple integrals}

In this section, we collect a number of elementary bounds on various integrals that appear several times throughout
the article. 
First, it will turn out to be useful to have bounds on expressions of the type
\begin{equ}
{b f(at) - a f(bt) \over b-a}\;,
\end{equ}
with $a$, $b$, $t$ in $\R_+$, where $f \colon \R_+ \to \R$ is a smooth function. 

We have the following:
\begin{proposition}\label{prop:bounddiff}
For $a$,  $b$, $t$, and $f$ as above, one has the global bound
\begin{equ}[e:boundf]
\Bigl|{b f(at) - a f(bt) \over b-a} - f(0)\Bigr| \le 2ab t^2\, \|f''\|_\infty\;,
\end{equ}
where $\|\cdot\|_\infty$ denotes the supremum norm. If furthermore there exists a constant $K$
such that $\sup_{y\ge 0} |f(y) - yf'(y)| \le K$,  then
\begin{equ}[e:boundf2]
\Bigl|{b f(at) - a f(bt) \over b-a}\Bigr| \le K\;,
\end{equ}
independently of $a$, $b$, and $t$.
\end{proposition}

\begin{proof}
We assume without loss of generality that $b > a$ (otherwise, just reverse the roles of $a$ and $b$)
and we set $\delta \eqdef b-a$.
One then has 
\begin{equ}
{b f(at) - a f(bt) \over b-a} = f(at) - {a \over \delta} \bigl(f(bt)-f(at)\bigr)\;,
\end{equ}
so that, writing $|\CI|$ for the left hand side of \eref{e:boundf}, 
\begin{equ}
\CI = at \Bigl({1\over a}\int_0^a f'(st)\,ds - {1 \over \delta}\int_a^b f'(st)\,ds\Bigr)\;.
\end{equ}
We now use the identity $f'(st) = f'(0) + s \int_0^tf''(rs)\,dr$ and then exchange the
order of integrals, so that
\begin{equ}
\CI = at \int_0^t \Bigl({1\over a} \int_0^a s f''(rs)\,ds - {1\over \delta} \int_a^b s f''(rs)\,ds\Bigr)\,dr\;.
\end{equ}
The first claim now follows by replacing the integrands by their suprema and using the triangle inequality. 

To show the bound \eref{e:boundf2}, we make use of the identity
\begin{equ}
{b f(at) - a f(bt) \over b-a}  = {ab\over b-a}\int_a^b {f(xt) - xtf'(xt) \over x^2}\,dx\;.
\end{equ}
Since $\int_a^b {dx\over x^2} = {b-a\over ab}$, the second claim then follows from the assumption.
\end{proof}

\begin{remark}
The constant $2$ appearing in \eref{e:boundf} could actually be improved to ${3\over 2}$.
\end{remark}

Another extremely useful calculation is the following:

\begin{lemma}\label{lem:boundExpSimple}
Let $a, b > 0$. Then, for every $t,t'\in \R$, one has
\begin{equs}
\int_{-\infty}^t\int_{-\infty}^{t'} e^{-a|t-r| -a|t' - r'| - b|r-r'|}\,dr'\,dr &= {a e^{-b|t-t'|} - b e^{-a|t-t'|} \over a(a^2-b^2)}\\
&\le  {1 \over a(a+b)} \wedge  {e^{-(a\wedge b)|t-t'|} \over a|a-b|}\;.
\end{equs}
\end{lemma}

\begin{proof}
The first identity follows from a lengthy but straightforward calculation. 
The fact that both $|t-r|$ and $|t'-r'|$ have the same prefactor in
the exponent is crucial for the result, otherwise, the expression is far lengthier.

To get the bound on the second line, we first use Proposition~\ref{prop:bounddiff} to 
bound the left hand side by $1/a(a+b)$ (the constant $K$ appearing there is equal to $1$ in our case).
It then suffices to observe that
\begin{equ}
\bigl|a e^{-b|t-t'|} - b e^{-a|t-t'|} \bigr| \le (a+ b)e^{-(a\wedge b)|t-t'|}\;,
\end{equ}
to obtain the second bound.
\end{proof}

%

Another very useful bound is given by

\begin{lemma}\label{lem:intExp}
For every $s < t$ and $u,v >0$, one has
\begin{equ}
\int_s^t e^{-u|x-s| - v|x-t|}\,dx \le {4 e^{-(u\wedge v)t} \over u+v} \le {4 \over u+v}\;.
\end{equ}
\end{lemma}

\begin{proof}
One has the identity
\begin{equ}
\CI \eqdef \int_s^t e^{-u|x-s| - v|x-t|}\,dx = e^{us-vt} \int_s^t e^{(v-u)x}\,dx = {e^{-u(t-s)} - e^{-v(t-s)} \over v-u}\;.
\end{equ}
Assume now without loss of generality that $v > u$. It then follows from the above that
\begin{equ}
\CI \le {e^{-u(t-s)} \over v-u}\;.
\end{equ}
On the other hand, the integral can be estimated by the supremum of its integrand, times
the length of the domain of integration, so that 
\begin{equ}
\CI \le (t-s) e^{-u(t-s)} = {u (t-s) e^{-u(t-s)} \over u} \le {e^{-{u\over 2}(t-s)}\over u}\;,
\end{equ}
where we made use of the fact that $xe^{-x} \le e^{-x/2}$. Combining these bounds, we conclude that
\begin{equ}
\CI \le {e^{-{u \wedge v \over 2}(t-s)}\over (u\wedge v)\vee |u-v|}\;.
\end{equ}
The claim now follows from the fact that $(u\wedge v)\vee |u-v| \ge (u \vee v)/2 \ge (u+v)/4$.
\end{proof}

\begin{lemma}\label{lem:diffKernel}
Let $F \colon \R \to \R$ be such that there exist constants $K>0$ and $b \ge 0$ such that
\begin{equ}
|F(s)| \le Ke^{-b|s|}\;,
\end{equ} 
and define
\begin{equ}
\CK(t-t') \eqdef \int_{-\infty}^t \int_{-\infty}^{t'} F(s-s')e^{-a(t+t'-s-s')}\,ds'\,ds\;.
\end{equ}
Then, one has the bound 
\begin{equ}
|\CK(\delta) - \CK(0)| \le (1\wedge a\delta)\Bigl(|\CK(0)| + {4K \over (a+b)^2}\Bigr)\le 5K{1\wedge a\delta \over a(a+b)}\;.
\end{equ}
\end{lemma}

\begin{proof}
By definition, one has
\begin{equ}
\CK(\delta) - \CK(0) = \bigl(e^{-a\delta} -1\bigr) \CK(0) +\int_0^\delta \int_{-\infty}^0 F(s-s')e^{-a(\delta-s-s')}\,ds'\,ds\;,
\end{equ}
so that it suffices to bound the second term in this expression. Since $s > s'$ over the whole
domain of integration, $F$ is bounded by $K e^{-b(s-s')}$, so that
\begin{equ}
\Bigl|\int_{-\infty}^0 F(s-s')e^{-a(\delta-s-s')}\,ds'\Bigr| \le {K e^{-bs -a(\delta-s)}\over a+b}\;.
\end{equ}
The first bound now follows from Lemma~\ref{lem:intExp}, and the second bound follows from
Lemma~\ref{lem:boundExpSimple}.
\end{proof}

\begin{proposition}\label{prop:boundExpGeneral}
The bound
\begin{equs}
\int_{-\infty}^s\int_{-\infty}^{s'} &\exp \bigl(-a|r-r'| - b|r-s|-c|r'-s'|-d|r-s'|-e|r'-s|\bigr)\,dr'\,dr \\
&\qquad\le {10 e^{-(d\wedge e)|s-s'|} \over (b+d)(c+e) + a((b+d)\wedge(c+e))}\;,
\end{equs}
holds for every $s,s' \in \R$ and every $a,b,c,d,e > 0$.
\end{proposition}

\begin{proof}
Throughout, we denote the integrand by $I(r,r')$ and we write 
\begin{equ}
\CR = {e^{-{d\wedge e\over 2}|s-s'|} \over (b+d)(c+e)}\;.
\end{equ}
We can (and will from now on) assume
without loss of generality that $s' > s$, since the case $s>s'$ is obtained by making the substitution
$(r,s,b,d) \leftrightarrow (r',s',c,e)$ and $\CR$ is left unchanged by this.
We decompose the integral over $r'$ into integrals over $(-\infty,r]$, $[r,s]$, and $[s,s']$.
The first one then yields
\begin{equ}
\int_{-\infty}^r I(r,r')\,dr' = {e^{-(b+c+e)|r-s|-d|r-s'|}\over a+c+e}\;,
\end{equ}
so that 
\begin{equ}
\int_{-\infty}^s\int_{-\infty}^r I(r,r')\,dr'\,dr = {e^{-d|s-s'|}\over (a+c+e)(b+c+d+e)} \le \CR\;.
\end{equ}
In order to bound the second integral, we use the bound $|r'-s'| \le |r'-s|$, which allows us to 
use Lemma~\ref{lem:intExp}. This yields the bound
\begin{equ}
\int_{r}^s I(r,r')\,dr' \le {4 e^{-b|r-s|-d|r-s'|}\over a+c+e}\;,
\end{equ}
so that 
\begin{equ}
\int_{-\infty}^s \int_{r}^s I(r,r')\,dr' \le {4 e^{-d|s-s'|}\over (a+c+e)(b+d)} \le 4\CR\;.
\end{equ}
Similarly, we can use Lemma~\ref{lem:intExp} for the last integral, so that
\begin{equ}
\int_{s}^{s'} I(r,r')\,dr' \le {4 e^{-b|r-s|-d|r-s'|}\over a+c+e}\;,
\end{equ}
yielding in the same way as before $\int_{-\infty}^s \int_{s}^{s'} I(r,r')\,dr' \le 4\CR$. The claim now follows at once.
\end{proof}

\subsection{Function spaces}

In this appendix, we collect a few useful facts about spaces of distributions with ``negative H\"older continuity''.
Recall that if $\alpha, \beta > 0$ and we have two functions $u\in \CC^\alpha$ and $v \in \CC^\beta$, then
the product $uv$ satisfies $uv \in \CC^{\alpha \wedge \beta}$. We would like to have a similar property for 
distributions in $\CC^{-\alpha}$ for some $\alpha > 0$.

In full generality, the above bounds does of course not hold: white noise belongs to $\CC^{-\alpha}$ for every $\alpha > {1\over 2}$, but squaring it simply makes no sense whatsoever. However, one has the following:

\begin{proposition}\label{prop:Holder}
Let $\alpha \in (0,1)$ and $\beta > \alpha$. Then, the bilinear map $(u,v) \mapsto uv$ extends to a continuous map
from $\CC^{-\alpha} \times \CC^\beta$ into $\CC^{-\alpha}$.
\end{proposition}

\begin{proof}
It suffices to show that, for $u$ and $v$ smooth, one has
\begin{equ}
\int_x^y u(z) v(z)\,dz \le |x-y|^{1-\alpha} \|u\|_{-\alpha}\|v\|_\beta\;.
\end{equ}
Writing $U$ for a primitive of $u$, we can write
\begin{equ}
\int_x^y u(z) v(z)\,dz  = \int_x^y \delta v(x,z)\,dU(z) + v(x)\,\delta U(x,y)\;.
\end{equ}
It then follows from Young's theory of integration \cite{Young} that the first quantity is bounded by
$|x-y|^{1+\beta-\alpha} \|v\|_\beta \|U\|_{1-\alpha}$, provided that $\beta > \alpha$. 
Since the second quantity is bounded by $|x-y|^{1-\alpha}\|U\|_{1-\alpha}\|v\|_\infty$, the claim follows at once.
\end{proof}

\begin{remark}\label{rem:Holder}
The condition $\beta > \alpha$ is sharp. 
Indeed, it is possible to construct a counterexample showing that the multiplication operator
cannot be extended to $\CC^{-{1\over 2}}\times \CC^{1\over 2}$.
\end{remark}

Let us also collect the following properties of the heat semigroup:

\begin{proposition}\label{prop:interpolation}
Let $P_t$ denote the heat semigroup on $S^1$. Then, for every $\alpha < \beta$ with $\alpha > -1$ and $\beta - \alpha \le 2$, one has the bounds
\begin{equ}
\|P_t u \|_{\beta} \lesssim t^{\alpha-\beta \over 2} \|u\|_{\alpha}\;,\quad
\|P_t u - u\|_{\alpha} \lesssim t^{\beta - \alpha \over 2} \|u\|_{\beta}\;,
\end{equ}
where the proportionality constants are uniform over any interval $(0,T]$ with $T>0$.
\end{proposition}

\begin{proof}
The statements are standard for positive H\"older exponents and follow immediately from the scaling properties of the heat kernel.
For negative exponents, they then follow from the fact that $P_t$ commutes with $\d_x$.
\end{proof}

\end{appendix}

\bibliographystyle{./Martin}
\bibliography{./refs}

\end{document}